\documentclass[a4paper, 10pt, DIV10, headinclude=false, footinclude=false]{scrartcl}

\usepackage{amsthm, amsmath, amssymb, amsfonts}
\usepackage[utf8]{inputenc}
\usepackage[english]{babel}
\usepackage{url}
\usepackage{mathtools}

\usepackage{subcaption}
\usepackage{tikz}
\usepackage{pgfplots}
\usepackage{pgfplotstable}
\usepackage{booktabs}

\usepackage{mathrsfs}

\numberwithin{figure}{section}
\numberwithin{table}{section}
\numberwithin{equation}{section}

\newenvironment{abstr}[1]{ \vspace{.05in}\footnotesize
   \parindent .2in
   {\upshape\bfseries #1. }\ignorespaces}{\par\vspace{.1in}}
\newenvironment{Abstract}{\begin{abstr}{Abstract}}{\end{abstr}}
\newenvironment{keywords}{\begin{abstr}{Key words}}{\end{abstr}}
\newenvironment{AMS}{\begin{abstr}{AMS subject classifications}}{\end{abstr}}

\newtheorem{theorem}{Theorem}[section]
\newtheorem{lemma}[theorem]{Lemma}
\newtheorem{corollary}[theorem]{Corollary}
\newtheorem{proposition}[theorem]{Proposition}

\theoremstyle{definition}

\newtheorem{remark}[theorem]{Remark}

\DeclareMathOperator{\supp}{supp}
\newcommand\CF{\mathcal{F}}
\newcommand\CL{\mathcal{L}}
\newcommand\CP{\mathcal{P}}
\newcommand\CQ{\mathcal{Q}}
\newcommand\CT{\mathcal{T}}
\newcommand\CV{\mathcal{V}}
\newcommand{\CVi}{\mathcal V^{\mathrm{int}}}

\newcommand\CR{\mathcal{R}}

\newcommand{\UN}{\textup{N}}
\newcommand{\UP}{\textup{P}}
\newcommand{\US}{\textup{S}}
\newcommand{\UT}{\textup{T}}
\newcommand{\UY}{\textup{Y}}

\newcommand{\ba}{\mathbf{a}}
\newcommand{\bx}{\mathbf{x}}
\newcommand{\oa}{{\omega^\ba}}

\newcommand{\eq}{:=}

\newcommand{\grad}{\boldsymbol \nabla}

\newcommand{\Om}{{\Omega_-}}
\newcommand{\Op}{{\Omega_+}}

\newcommand{\Sm}{\sigma_-}
\newcommand{\Sp}{\sigma_+}

\newcommand{\ho}{\widehat \omega}
\newcommand{\hK}{\widehat K}
\newcommand{\hb}{\widehat b}
\newcommand{\hC}{\widehat C}
\newcommand{\hv}{\widehat v}
\newcommand{\hw}{\widehat w}

\newcommand{\contrast}{\mathscr{C}}

\usepackage{todonotes}

\allowdisplaybreaks[4]

\begin{document}
   
   \title{A generalized finite element method for problems with sign-changing coefficients}
   \author{Th\'eophile Chaumont-Frelet\footnotemark[1] \footnotemark[2] \and Barbara Verf\"urth\footnotemark[3]}
   \date{}
   \maketitle
   
   \renewcommand{\thefootnote}{\fnsymbol{footnote}}
   \footnotetext[1]{Inria Sophia Antipolis Méditerran\'ee, 2004 Route des Lucioles, 06902 Valbonne, France}
   \footnotetext[2]{Laboratoire J.A. Dieudonn\'e UMR CNRS 7351, Parc Valrose, 06108 Nice, France}
   \footnotetext[3]{Institut für Angewandte und Numerische Mathematik, Karlsruher Institut für Technologie (KIT), Englerstr.~2, D-76131 Karlsruhe}
   \renewcommand{\thefootnote}{\arabic{footnote}}
   
   \begin{Abstract}
   Problems with sign-changing coefficients occur, for instance, in the study of transmission problems
   with metamaterials. In this work, we present and analyze a
   generalized finite element method in the spirit of the Localized Orthogonal Decomposition,
   that is especially efficient when the negative and positive materials exhibit multiscale
   features.
   We derive optimal linear convergence in the energy norm independently of the potentially low
   regularity of the exact solution. Numerical experiments illustrate the theoretical convergence rates
   and show the applicability of the method for a large class of sign-changing diffusion problems.    
   \end{Abstract}
   
   

 \section{Introduction}
\label{sec:introduction}

Metamaterials with, for instance, negative refractive index have attracted a lot of interest
over the last years due to many applications \cite{Pen00perfectlens,SPW04metamat}. The related mathematical problems are
characterized by so-called sign-changing coefficients. At the simplest example of a diffusion
problem in a domain $\Omega \subset \mathbb R^d$, $d \in \{2,3\}$, it means that the diffusion coefficient $\sigma$ takes strictly negative values,
i.e., $\sigma \leq -|\Sm| < 0$ in some part $\Om$ of the domain, while it takes strictly positive
values, i.e., $\sigma \geq |\Sp| > 0$ in the complement $\Op$. The interface
	$\Gamma$ between $\Omega_+$ and $\Omega_-$ is then called the ``sign-changing'' interface.
Such a behavior of the coefficient in the PDE does not only appear for metamaterials with
negative effective properties \cite{SPW04metamat}, but also for electric permittivities, which can have a negative
real part for certain metals.

The change of sign of $\sigma$ has tremendous effects on the analysis and numerics.
The standard assumption of coercive bilinear forms is no longer valid, so that existence and
uniqueness of solutions have to be studied anew. Employing the approach of $\UT$-coercivity
\cite{BCZ10signchangingTcoercive}, a large progress has been made in this area in the last years
considering the diffusion problem \cite{BCC12signchangingTcoercive,BCZ10signchangingTcoercive}
as well as time-harmonic wave propagation \cite{BCC14signchangingmaxwell,BCC14signchaningmaxwell2d}
and eigenvalue problems \cite{CCC17signchangingeigenvalue}.
Essentially, the problem is well-posed if
the contrast $|\Sp|/|\Sm|$ lies outside a so-called ``critical interval'' $I = [1/r,r]$,
where $r \geq 1$ depends on the geometry of $\Gamma$.

When discretizing these problems with the standard finite element method, the questions of
existence and uniqueness of the discrete solution as well as convergence rates for the error
immediately arise. Simply speaking, they have been answered positively in two different
scenarios, namely a) if the mesh satisfies certain symmetry properties around the
interface $\Gamma$, which is denoted as $\UT$-conformity
\cite{BCC18signchangingmesh}, or b) if the contrast $|\Sp|/|\Sm|$ is outside
an enlarged critical interval $\widetilde I = [1/\widetilde r,\widetilde r]$,
where $\widetilde r > r$ \cite[Section 5.1]{CC13Tcoercivemesh}.

Besides the a priori stability and error analysis, a posteriori error indicators and their reliability and efficiency have been studied for the standard finite element method as well \cite{CV18signchangingapost,NV11signchaningapost}.
Furthermore, an optimization-based scheme which does not require symmetric meshes is introduced in \cite{AHL17optsignchanging}.
Apart from continuous Galerkin methods, we also mention that schemes in the discontinuous Galerkin framework have been presented and analyzed in \cite{CC13signchangingdG} and \cite{LR19signchangingHDG}.

The main contribution of the present work is the introduction and numerical analysis of a
generalized finite element method in the spirit and framework of the Localized Orthogonal
Decomposition (LOD) \cite{HP13oversampl,MP14LOD,Pet15LODreview}. 
The LOD is especially targeted at so-called multiscale problems, where the coefficient is subject to rapid spatial variations. Standard discretization schemes need to resolve all these features with their computational grid leading to an enormous and often infeasible computational effort.
The basic idea of the LOD is to construct a low-dimensional solution space with very good $H^1$-approximation properties with respect to the exact solution. As standard finite element functions on a coarse mesh alone
do not yield a faithful approximation space, problem-dependent multiscale functions are added.
The latter are defined as solutions of local fine-scale problems.
Since its introduction in \cite{HP13oversampl,MP14LOD}, the LOD has been successfully
applied in various situations, where we mention in particular the reduction of the pollution effect
for high-frequency Helmholtz problems
\cite{GP15scatteringPG,P17LODhelmholtz,PV19helmholtzhighcontrast}.
The efficient implementation of the method is outlined in \cite{EHMP16LODimpl}.
Note that the LOD is closely connected to domain
decomposition methods \cite{KPY16LODiterative,KY16LODiterative,PVV19dd}.
Further, it can also be interpreted in the context of homogenization \cite{GP17lodhom}. If $\sigma$ is (locally) periodic one can thereby recover traditional (analytical) homogenization results, see \cite{BDT19,BR16} for such results in the case of periodic sign-changing coefficients.

We analyze the stability and convergence of the proposed method
when $d=2$ or $3$, under the assumption that
the interface is resolved by the mesh and that the contrast is ``sufficiently large''.
While this restriction means that the interface $\Gamma$ is essentially ``macroscale'',
$\sigma$ is allowed to exhibit a rough and multiscale behavior in $\Om$ and $\Op$.
Under these assumptions, the present method allows for optimal convergence orders on
uniform meshes, even in the presence of corner singularities, which is already known for
positive discontinuous diffusion coefficients. In contrast with standard FEM
\cite{CC13Tcoercivemesh}, considerable complications arise in the analysis of the LOD method
in the presence of sign-changing coefficients.
Indeed, while the LOD method has been analyzed for a rather large class
of inf-sup stable problems in \cite[Chap.~2]{Maier20}, these general arguments
cannot be directly applied here, because of the inherently non-local procedure involved
by the $\UT$-coercivity approach.

While our numerical analysis assumes an interface-resolving mesh as well as an hypothesis
on the contrast, we present numerical experiments with general meshes, that do not
necessarily resolve the sign-changing interface(s), as well as contrasts close
to the critical interval. Although they are limited to two-dimensional settings, these
results are very promising, and indicate the efficiency
of the method in highly heterogeneous media. Finally, we mention that we consider
the diffusion problem here, but the arguments and techniques might also be generalized to
other settings such as the Helmholtz equation.

The paper is organized as follows. In Section \ref{sec:setting}, we introduce our model problem.
Our generalized finite element method is motivated in an ideal form
in Section \ref{sec:lod}. There, we also discuss three main challenges of the ideal method, namely the well-posedness of the construction in the case of sign-changing coefficients, as well as the localization and discretization of the multiscale basis. The dedicated arguments required to take into account
$\UT$-coercivity in the context of LOD are discussed in Section \ref{sec:intpol}. The fully practical LOD is finally presented and analyzed in Section \ref{sec:lodpractical}.
In Section
\ref{sec:numexperiment}, we present several numerical experiments illustrating our theory
and showing the applicability of the method even for meshes that do not resolve the interface,
and contrasts close to the critical interval.
Some technical finite element estimates related to quasi-interpolation are collected in Appendix
\ref{sec:appendix}.

\section{Settings}

\label{sec:setting}  
In this section we introduce the required functional spaces, the model problem
under consideration and discuss the notion of $\UT$-coercivity.

	\subsection{Domain and coefficient}
	
	We consider a polytopal domain $\Omega \subset \mathbb R^d$, with $d \in \{2,3\}$.
	We assume that $\overline{\Omega} = \overline{\Op} \cup \overline{\Om}$,
	where $\Omega_\pm \subset \Omega$ are two non-overlapping subsets of $\Omega$. We denote by
	\begin{equation*}
	\Gamma \eq \partial \Op \cap \partial \Om
	\end{equation*}
	the boundary shared by the two subsets. For the sake of simplicity, we assume that
	$\Gamma$ is polytopal. However, we do not require specific assumption about the
	topology of $\Omega_\pm$. In particular, $\Op$ and/or $\Om$ can be multi-connected.
	
	We consider a diffusion coefficient $\sigma \in L^\infty(\Omega)$ such that
	$\sigma|_\Om \leq -\Sm$ and $\sigma|_\Op \geq \Sp$, where $0 < \Sp \leq \Sm < +\infty$
	are fixed real numbers. Note that by symmetry, one could alternatively consider
	$\sigma|_{\Omega_-} \geq \sigma|_{\Omega_+}$, but we only analyze the other
	case for the sake of simplicity. In the remaining of this work, we will call the
	positive real number $\contrast \eq \sigma_-/\sigma_+$ the  ``contrast''.
	
	\subsection{Functional spaces}
	
	Throughout this work, if $D \subset \Omega$, $L^2(D)$ is the usual Lebesgue
	space of square integrable functions. We denote by $(\cdot,\cdot)_D$ and
	$\|\cdot\|_{0,D}$ the usual inner product and norm of $L^2(D)$. We employ
	the same notations for the inner product and norm of $\left (L^2(D)\right )^d$.
	Classically,
	$H^1(D) \eq \left \{ v \in L^2(D) \; | \; \grad v \in \left (L^2(D)\right )^d\right \}$
	denotes the usual Sobolev space, and if $\gamma \subset \partial D$, we employ the notation
	\begin{equation*}
	H^1_{\gamma}(D) \eq \left \{ v \in H^1(D) \; | \; v|_{\gamma} = 0 \right \},
	\end{equation*}
	that we equip with its usual semi-norm $|\cdot|_{1,D}$.
	Note that, thanks to Poincar\'e inequality, $|\cdot|_{1,D}$ is actually
	a norm on $H^1_\gamma(D)$ as long as $\gamma$ has a strictly positive surface measure.
	Unless stated otherwise, we will always employ the notation $H^1_\gamma(D)$ when
	the surface measure of $\gamma$ is strictly positive, and equip the space with
	$|\cdot|_{1,D}$ as norm. In particular, this norm is considered when defining
	the norm of linear operators. We will also use the usual notation
	$H^1_0(\Omega) \eq H^1_{\partial \Omega}(\Omega)$, and consider the equivalent
	norm
	\begin{equation*}
	|v|_{1,\sigma,\Omega}^2 \eq \int_\Omega |\sigma| |\grad v|^2
	\end{equation*}
	for $v \in H^1_0(\Omega)$.
	
	\subsection{Model problem}

Given $f \in L^2(\Omega)$, we
seek $u \in H^1_0(\Omega)$ such that
\begin{equation}
\label{eq:modelpb}
a(u,v) = (f,v)_{\Omega},
\end{equation}
where
\begin{equation*}
a(u,v) \eq (\sigma \grad u,\grad v)_{\Omega}.
\end{equation*}

	When $\sigma$ is positive, the bilinear form $a(\cdot,\cdot)$ actually
	corresponds to the inner product associated with the norm $|\cdot|_{1,\sigma,\Omega}$.
	In particular, $a(\cdot,\cdot)$ is coercive, and the well-posedness of \eqref{eq:modelpb}
	follows from Lax-Milgram Lemma. Here, the sign-change of $\sigma$ prevents the coercivity
	of $a(\cdot,\cdot)$. Following an approach known as $\UT$-coercivity, we will show that instead,
	assuming that the contrast is sufficiently large, $a(\cdot, \cdot)$ satisfies an inf-sup condition, ensuring
	the well-posedness of \eqref{eq:modelpb}.

	\begin{remark}
		Throughout the whole work, we assume $f\in L^2(\Omega)$. While the model problem and the generalized finite element method can be defined for $f\in H^{-1}(\Omega)$ as well, $f\in L^2(\Omega)$ is required to obtain convergence of the method, see Proposition \ref{prop:lodideal}.
	\end{remark}

	\subsection{Symmetrization and $\UT$-coercivity}
	\label{subsec:Tcoercive}
	
	In the broadest sense, we say that the bilinear form $a(\cdot,\cdot)$ is
	$\UY$-coercive if there exists an operator $\UY \in \CL(H^1_0(\Omega))$ such that
	\begin{equation*}
	a(v,\UY v) \geq a_\star |v|_{1,\sigma,\Omega}^2 \quad \forall v \in H^1_0(\Omega),
	\end{equation*}
	for some fixed $a_\star > 0$. This definition is equivalent to the usual ``inf-sup'' condition.
	It guarantees the well-posedness of Problem \eqref{eq:modelpb}, and the stability constant
	then depends on $a_\star$ and $\|\UY\|_{\CL(H^1_0(\Omega))}$.
	
	In the context of problems with a sign-changing coefficient, a particular form of
	$\UT$-coercivity based on symmetrization has been developed
	\cite{BCC12signchangingTcoercive,BCC18signchangingmesh}. The key idea is to
	design an operator $\UY$ that ``flips'' the sign of its argument in $\Omega_-$. 
	This construction relies on a symmetrization operator $\US$, that maps
	$H^1_{\partial \Omega}(\Om)$ into $H^1_{\partial \Omega}(\Op)$.

	In the following, we assume that $\Om$ and $\Op$ are such that a ``symmetrization''
	operator $\US$ is available. By symmetrization operator, we mean that
	that $\US \in \CL(H^1_{\partial \Omega}(\Om);H^1_{\partial \Omega}(\Op))
	\cap \CL(L^2(\Om);L^2(\Op))$, is a linear mapping that preserves the trace
	on $\Gamma$, i.e., $\US v|_\Gamma = v|_\Gamma$ for all $v \in H^1_{\partial \Omega}(\Om)$.
	We refer the reader to \cite{BCC12signchangingTcoercive} for examples of such symmetrization
	operators. Further, we point out that rather general polypotal interfaces can be treated
	\cite{BCC18signchangingmesh} if one uses the concept of so-called \emph{weak $\UT$-coercivity}.
	We discuss the extension of our work to this setting in Section \ref{sec_weak_Tcoer}.
	
	Using $\US$, we may now define an operator
	$\UT \in \CL(H^1_0(\Omega))$ for which $a(\cdot,\cdot)$ is $\UT$-coercive.
	Specifically, if $u \in H^1_0(\Omega)$, we set
\begin{equation*}
\UT u = \left \{
\begin{array}{ll}
-u & \text{ in } \Om
\\
u-2\US u & \text{ on } \Op.
\end{array}
\right .
\end{equation*}
One easily sees that we have
	\begin{equation}\label{eq:TL2stab}
	|v-\UT v|_{1,\Op} \leq C_\pm(\UT) |v|_{1,\Om}
	\text{ and }
	\|v-\UT v\|_{0,\Op} \leq C_\pm^0(\UT) \|v\|_{0,\Om}
	\end{equation}
	for all $v \in H^1_0(\Omega)$, with
	\begin{equation*}
	C_\pm(T) \eq 2\|\US\|_{\CL(H^1_{\partial \Omega}(\Om);H^1_{\partial \Omega}(\Op))},
	\qquad
	C_\pm^0(T) \eq 2\|\US\|_{\CL(L^2(\Om);L^2(\Op))}.
	\end{equation*}

	We can readily employ $\UT$ to establish an inf-sup condition for $a(\cdot,\cdot)$
	on $H^1_0(\Omega)$, thus showing that Problem \eqref{eq:modelpb} is well-posed
	\cite{BCC12signchangingTcoercive}. However, later in the analysis of the LOD method,
	we will need to show a similar inf-sup condition, but on the kernel of some
	quasi-interpolation operator, instead
	of the whole $H^1_0(\Omega)$ space. For this reason, we first give a general result
	linking the existence of an operator $\UY \in \CL(V)$ for some $V \subset H^1_0(\Omega)$
	and the inf-sup stability of $a(\cdot,\cdot)$ on $V$.

\begin{theorem}
	\label{prop:T_coercivity}
	Let $V \subset H^1_0(\Omega)$ be a closed subspace. Assume that
	there exists an operator $\UY \in \CL(V)$ such that
	\begin{subequations}
		\label{eq_T}
		\begin{equation}
		\label{eq_T_minus}
		\left .\left (\UY v\right ) \right |_{\Om} = -v|_{\Om}
		\end{equation}
		and
		\begin{equation}
		\label{eq_T_stab}
		|v-\UY v|_{1,\Op} \leq C_{\pm}(\UY) |v|_{1,\Om},
		\end{equation}
	\end{subequations}
	for all $v \in V$. Then, we have
	\begin{equation}
	\label{eq_T_coercivity}
	a(v,\UY v)
	\geq
	\left (
	1
	-
	\frac{C_\pm(\UY)}{2}
	\left (\frac{\sup_\Op \sigma}{\inf_\Op \sigma}\right ) \sqrt{\frac{\Sp}{\Sm}}
	\right )
	|v|_{1,\sigma,\Omega}^2
	\end{equation}
	for all $v \in V$.
\end{theorem}

\begin{proof}
	Pick an arbitrary element $v \in V$. Taking advantage of \eqref{eq_T_minus}, we may write
	\begin{align}
	\nonumber
	a(v,\UY v)
	&=
	(\sigma \grad v,\grad(\UY v))_\Op
	+
	(\sigma \grad v,\grad(\UY v))_\Om
	\\
	\nonumber
	&=
		(\sigma \grad v,\grad(\UY v))_\Op
		-
		(\sigma \grad v,\grad v)_\Om
	\\
	\nonumber
	&=
	(|\sigma| \grad v,\grad(\UY v))_\Op
	+
	(|\sigma| \grad v,\grad v)_\Om
	\\
	\label{tmp_T_coercivity_1}
	&=
	|v|_{1,\sigma,\Omega}^2 - (|\sigma| \grad v,\grad (v-\UY v))_\Op
	\end{align}
	Then, we derive that
	\begin{equation}
	\label{tmp_T_coercivity_2}
	\begin{aligned}
	|(\sigma \grad v,\grad(v-\UY v))_\Op|
	&\leq
	(\sup_\Op \sigma) |v|_{1,\Op}|v-\UY v|_{1,\Op}
	\\
	&\leq
	\left (\frac{\sup_\Op \sigma}{\inf_\Op \sigma}\right )
	\Sp |v|_{1,\Op}|v-\UY v|_{1,\Op}
	\\
	&\leq
	C_\pm(\UY)
	\left (\frac{\sup_\Op \sigma}{\inf_\Op \sigma}\right )
	\Sp |v|_{1,\Op}|v|_{1,\Om}
	\\
	&\leq
	\frac{C_\pm(\UY)}{2}
	\left (\frac{\sup_\Op \sigma}{\inf_\Op \sigma}\right )
	\sqrt{\frac{\Sp}{\Sm}} |v|_{1,\sigma,\Omega}^2,
	\end{aligned}
	\end{equation}
	where we have employed Young's inequality
	\begin{equation*}
	\Sp |v|_{1,\Op}|v|_{1,\Om}
	=
	\sqrt{\frac{\Sp}{\Sm}} \sqrt{\Sp} |v|_{1,\Op} \sqrt{\Sm} |v|_{1,\Om}
	\leq
	\frac{1}{2} \sqrt{\frac{\Sp}{\Sm}} |v|_{1,\sigma,\Omega}^2.
	\end{equation*}
	Estimate \eqref{eq_T_coercivity} then follows from
	\eqref{tmp_T_coercivity_1} and \eqref{tmp_T_coercivity_2}.
\end{proof}

	Recalling \eqref{eq:TL2stab}, an immediate consequence of Theorem
	\ref{prop:T_coercivity} is that if the contrast is sufficiently large, $a(\cdot,\cdot)$
	is $\UT$-coercive, and Problem \eqref{eq:modelpb} is well-posed.

	\begin{corollary}
		Under the assumption that
		\begin{equation}
		\label{eq_assumption_contrast}
		\sqrt{\contrast}
		>
		\left (\frac{\sup_\Op \sigma}{\inf_\Op \sigma}\right )
		\|\US\|_{\CL(H^1_{\partial \Omega}(\Om);H^1_{\partial \Omega}(\Op))}
		\end{equation}
		we have $a(v,\UT v) \geq \alpha |v|_{1,\sigma,\Omega}^2$  for all $v \in H^1_0(\Omega)$,
		with
		\begin{equation*}
		\alpha
		\eq
		1
		-
		\frac{C_\pm(\UT)}{2}
		\left (\frac{\sup_\Op \sigma}{\inf_\Op \sigma}\right )
		\sqrt{\frac{\Sp}{\Sm}} > 0.
		\end{equation*}
		In particular, Problem \eqref{eq:modelpb} is well-posed.
	\end{corollary}

\begin{remark}
	In the above, we (arbitrarily) assumed that $\Sp \leq \Sm$. This is not a
	restrictive assumption, since in the case
	where $\Sm \leq \Sp$, we can always get back to this situation by applying
	a minus sign on both sides of \eqref{eq:modelpb}.
	In particular, when we write that the contrast is ``sufficiently large'',
	it actually means it is ``sufficiently far away from the critical interval''.
	Similarly, one could choose to define a symmetrization operator $\US$
		mapping from $\Omega_+$ to $\Omega_-$. We only consider one direction for the
		sake of simplicity. We refer the reader to \cite{BCC12signchangingTcoercive}
		for a detailed discussion.
\end{remark}

\section{An ideal generalized finite element method}
\label{sec:lod}
 
	In this section, we are concerned with the discretization of our model problem \eqref{eq:modelpb}. We first introduce some finite element notation in Section \ref{subsec:meshes}.
	The generalized finite element method is built upon a quasi-interpolation operator, which we briefly introduce in Section \ref{subsec:oswald}, and then present the idea of the generalized finite element method in Section \ref{subsec:ideal}.
	Finally, in Section \ref{sec_weak_Tcoer}, we discuss the extension of the method to problems
	satisfying a ``weak'' $\UT$-coercivity condition.

\subsection{Preliminaries and notation}
\label{subsec:meshes}
	We consider a shape-regular quasi-uniform triangulation $\CT_H$ of $\Omega$.
	$\CT_H$ is supposed to be a coarse mesh in the sense that it does not necessarily
	resolve the variations and oscillations of $\sigma$ nor the sign-changing interface
	$\Gamma$.
	
	\begin{remark}
		The definition of the method does not require the triangulation $\CT_H$ to fit
		the sign-changing interface $\Gamma$, and numerical experiments seem to indicate
		that the method is efficient without this requirements. In our theoretical analysis
		however, we require that $\Gamma$ (but not the oscillation of $\sigma$)
		is resolved by $\CT_H$ as a technical assumption.
	\end{remark}

Given an element $K \in \CT_H$ the notations
\begin{equation*}
H_K \eq \sup_{x,y \in K} |x - y|, \qquad
\rho_K \eq \sup \{ r > 0 \; | \; \exists x \in K; \; B(x,r) \subset K \},
\end{equation*}
respectively denote the diameter of $K$ and the radius of the largest balled contained in $K$.
The assumptions of shape-regularity and quasi-uniformity imply the existence of a constant
$\kappa > 1$ such that
\begin{equation*}
\frac{H}{\rho} \leq \kappa,
\end{equation*}
where $H \eq \max_{K \in \CT_H} H_K$ and $\rho \eq \min_{K \in \CT_H} \rho_K$.

$\CV_H$ is the set of vertices of $\CT_H$, and $\CVi_H$ is the set of ``interior''
vertices that do not lie on $\partial \Omega$.
If $\ba \in \CV_H$, we denote by $\psi^\ba$ the associated hat function and set
$\oa \eq \supp \psi^\ba$.
We further split $\CVi_H$ into three categories of vertices:
\begin{equation*}
\CV_H^- \eq \left \{
\ba \in \CVi_H \; | \;
\ba \in \Om
\right \}, \quad
\CV_H^+ \eq \left \{
\ba \in \CVi_H \; | \;
\ba \in \Op
\right \},
\quad
\CV_H^0 \eq \left \{
\ba \in \CVi_H \; | \;
\ba \in \Gamma
\right \}.
\end{equation*}
For $K \in \CT_H$, let
	$\CV(K) \subset \CV_H$ denote the set of vertices of $K$.
If $\ba \in \CV_H$, then
\begin{equation*}
\CT_H^\ba \eq \{ K \in \CT_H \quad | \quad \ba \in \CV(K) \},
\end{equation*}
is the associated local mesh, and $\sharp \ba \eq \operatorname{card} \CT_H^\ba$
is the number of elements touching $\ba$.

	The standard conforming finite element space of lowest order Lagrange elements is denoted by
	$V_H\subset H^1_0(\Omega)$, i.e., 
	\[V_H=\{v\in H^1_0(\Omega)\quad |\quad v|_K\in \CP_1(K)\,\,\, \forall K\in \CT_H\},\]
	where $\CP_1$ denotes the polynomials of total degree at most one.
	Note that a standard finite element discretization of \eqref{eq:modelpb} using $V_H$ will fail to produce faithful approximations if $\CT_H$ does not resolve the variations of $\sigma$.
	
	For any $D \subset \Omega$, $\UN(D)$ denotes the element patch around $D$ defined as
	\[\UN(D)=\bigcup\{K\in \CT_H, \overline{K}\cap \overline{D}\neq \emptyset\}.\]
	For later use, we also define inductively the $m$-layer patch $\UN^m(D)$ around $D$ via $\UN^m(D)=\UN(\UN^{m-1}(D
	))
	$ for $m\geq 2$. The idea of patches is illustrated in Figure \ref{fig:patch} in Section \ref{subsec:local} below.

	\subsubsection{Reference element and associated constants}
	\label{section_reference_element}
	
	In the remaining of this work, $\hK \subset \mathbb R^d$ is an
	arbitary but fixed ``reference'' simplex with diameter $\widehat H = 1$.
	We denote by $\hb \in \CP_{d+1}(\hK)$ the ``bubble'' function of
	$\hK$ that we define as the product of its $(d+1)$ barycentric
	coordinate functions. Since $0 \leq \hb \leq 1$ on $\hK$,
	$\|\cdot\|_{0,\hK}$ and $\|\hb^{1/2} \cdot\|_{0,\hK}$
	are equivalent norms on $\CP_1(\hK)$ and we denote by
	\begin{equation}
	\label{eq_bubble_norm}
	\hC_{\mathrm norm} \eq \sup_{q \in \CP_1(\hK) \setminus \{0\}}
	\frac{\|q\|_{0,\hK}}{\|\hb^{1/2} q\|_{0,\hK}},
	\end{equation}
	the upper constant (note that since $\hb \leq 1$, the constant is $1$ in the other direction).
	The constants
	\begin{equation}
	\label{eq_inf_inv}
	\hC_{\mathrm inf}
	\eq
	|\hK|^{1/2}
	\sup_{q \in \CP_1(\hK) \setminus \{0\}}
	\frac{\|q\|_{0,\infty,\hK}}{\|q\|_{0,\hK}}
	\qquad
	\hC_{\mathrm inv} \eq \sup_{q \in \CP_1(\hK) \setminus \{0\}}
	\frac{|q|_{1,\hK}}{\|q\|_{0,\hK}}
	\end{equation}
	will also be useful in the sequel.

	\subsubsection{Reference patches and associated constants}
	\label{section_reference_patch}
	
	We assume that there exists a finite set of ``reference patches''
	$\widehat \CR$ such that for all $\ba \in \CV_H$,
	there exist $\widehat \CT \in \widehat \CR$ and a bilipschitz invertible mapping
	$\CF: \ho \to \oa$, $\ho$ being the domain associated with $\widehat \CT$,
	such that for each $\hK \in \widehat \CT$ the restriction of $\CF$ to $\hK$ is affine,
	and $\CF(\hK) = K$ for some $K \in \CT_H^{\ba}$. Since the mesh $\CT_H$ is quasi-uniform,
	we may further assume that there exists two constants $c_\star,c^\star$ such that
	$c_\star H \leq |D \CF(\widehat \bx)| \leq c^\star H$ for all $\widehat \bx \in \ho$.
	For the sake of simplicity, we also assume without loss of generality that
	$\CF^{-1}(\ba) = \boldsymbol 0$.

	We assume that all elements $\hK$ in the reference patches satisfy
	$H_{\hK} \leq 1$ and $\rho_{\hK} \geq \widehat \rho$.
	
	We employ the notation
	\begin{equation}
	\label{eq_poincare_tmp}
	\hC_{\mathrm P} \eq \max_{\CT \in \widehat \CR} \sup_{\hv \in H^1(\ho)}
	\frac{\|\hv\|_{0,\ho} + |\hv|_{1,\ho}}{|\ell(\hv)| + |\hv|_{1,\ho}},
	\end{equation}
	where
	\begin{equation*}
	\ell(\hv) \eq \frac{1}{\sharp \widehat \CT}
	\sum_{\hK \in \widehat \CT} (\CP_{\hK} \hv)(\boldsymbol 0).
	\end{equation*}
	As we detail in Appendix \ref{sec:appendix}, the above definition makes sense
	and we have $\hC_{\mathrm P} < +\infty$.
	
	We will also need the constants
	\begin{equation*}
	\hC_{\mathrm u} \eq \max_{\widehat \CT \in \widehat \CR}
	\left (
	\frac{\max_{K \in \widehat \CT} |K|}{\min_{K \in \widehat \CT} |K|}
	\right )
	\end{equation*}
	as well as
	\begin{equation*}
	C_{\mathrm u}^\ba \eq
	\frac{\max_{K \in \CT_H^\ba} |K|}{\min_{K \in \CT_H^\ba} |K|}
	\qquad
	C_{\mathrm u} \eq \max_{\ba \in \CV_H} C_{\mathrm u}^\ba.
	\end{equation*}

	\subsection{The quasi-interpolation operator}
	\label{subsec:oswald}
	
	The LOD method hinges on a stable quasi-interpolation operator $I_H: H^1_0(\Omega) \to V_H$.
	Here, following \cite{Pet15LODreview}, we consider a standard Oswald-type quasi-interpolation
operator
$I_H: H^1_0(\Omega) \to V_H$. For $v \in H^1_0(\Omega)$, it is defined as
\begin{equation}\label{eq:defintpol}
I_H v \eq \sum_{\ba \in \CVi_H} m^\ba(v) \psi^\ba,
\end{equation}
with
\begin{equation*}
m^\ba(v) \eq \frac{1}{\sharp \ba} \sum_{K \in \CT_H^\ba} (P_K v)(\ba),
\end{equation*}
where $P_K v$ denotes the $L^2(K)$ projection onto $\CP_1(K)$.

$I_H$ is a projection onto $V_H$ ($I_H \circ I_H = I_H$) and we furthermore have
\begin{equation}
\label{eq:IHstabapprox}
\|v-I_H v\|_{0,K}+ H\|\grad (v-I_H v) \|_{0,K} \leq C_I H\|\grad v\|_{0, \UN(K)} \qquad \forall K \in \CT_H
\end{equation}
for all $v \in H^1_0(\Omega)$, see \cite[Sec.~4, eq.~(16)]{Pet15LODreview} and the references therein.
While for the sake of simplicity, we work with the above mentioned operator $I_H$,
we emphasize that other quasi-interpolation operators could be considered, and we
refer the reader to \cite{EHMP16LODimpl} for the required properties.

	\subsection{Motivation and presentation of the ideal method}
	\label{subsec:ideal}
	
	The aim of this section is to construct a generalized finite element method in order to approximate the solution $u$ of \eqref{eq:modelpb} on the coarse mesh $\CT_H$ even if $\sigma$ is a multiscale coefficient and the standard finite element method on $\CT_H$ therefore fails to produce a faithful approximation.
	The idea is to construct a generalized finite element space $\widetilde{V}_H$ of the same dimension as $V_H$, but with better approximation properties.
	We will explain and introduce this idea in detail following the lines of thought for an elliptic diffusion problem (see, e.g., \cite{MP14LOD,Maier20,Pet15LODreview} for more details) and discuss the occurring challenges.
	
	We note first that the projection property of $I_H$ implies the decomposition of $H^1_0(\Omega)$ into the finite element space $V_H$ and the finescale space $W:=\ker(I_H)$, i.e. $H^1_0(\Omega) = V_H\oplus W$.
	We stress that $W$ represents the space of functions with potential finescale oscillations and is
	infinite-dimensional.
	Since $I_H$ is a stable quasi-interpolation onto $V_H$, $I_H u$ already contains many
	characteristic coarse features of the exact solution and hence, may be a sufficiently good
	approximation in many cases. Note, however, that $I_H u$ is typically not found by the finite element method.
	
	A simple calculation shows that it holds for any $v\in H^1_0(\Omega)$ that
	\[a(I_H u, v)= (f, v)_\Omega-a((\operatorname{id} - I_H) u, v).\]
	The last term on the right-hand side vanishes if we restrict the test functions $v$ to the space
	\[\widetilde{V}_H:=\{v\in H^1_0(\Omega) \quad |\quad a(w,v)=0\,\,\, \forall w\in W\}.\]
	This means that $I_Hu$ can be characterized as a Petrov-Galerkin solution with ansatz space $V_H$ and test space $\widetilde{V}_H$.
	This ideal (test) space comes with the decomposition $H^1_0(\Omega)=\widetilde{V}_H\oplus W$ which additionally is orthogonal with respect to $a(\cdot, \cdot)$.
	We now provide an alternative characterization of $\widetilde{V}_H$ that in particular will show that $\dim V_H =\dim \widetilde{V}_H$.
	For this, we introduce a so-called correction operator $\CQ: H^1_0(\Omega)\to W$ by 
	\begin{equation}\label{eq:corrector}
	a(w, \CQ v)=a(w, v)\qquad \text{for all}\quad w\in W.
	\end{equation}
	As a direct consequence, we obtain $a(w, (\operatorname{id}-\CQ)v)=0$ for all $w\in W$ and hence, the characterization of $\widetilde{V}_H$ from above.
	This implies
	\[\widetilde{V}_H = (\operatorname{id}-\CQ)H^1_0(\Omega)=(\operatorname{id}-\CQ)\bigl(I_H(H^1_0(\Omega))+(\operatorname{id}-I_H)H^1_0(\Omega)\bigr)=(\operatorname{id}-\CQ)V_H\]
	because $I_H(H^1_0(\Omega))=V_H$, $(\operatorname{id}-I_H)H^1_0(\Omega)=W$ and $(\operatorname{id}-\CQ)W=\{0\}$.
	
	To sum up, we have $\widetilde{V}_H=(\operatorname{id}-\CQ)V_H$ and, hence, the desired property $\dim \widetilde{V}_H=\dim V_H$. We will use the space $\widetilde{V}_H$ not only as test space, but also as ansatz space in our generalized finite element (Galerkin) method.
	This means that we seek $u_H\in V_H$ such that
	\begin{equation}\label{eq:lodideal}
	a((\operatorname{id}-\CQ)u_H, (\operatorname{id}-\CQ)v_H)=(f, (\operatorname{id}-\CQ)v_H)_\Omega\qquad \text{for all}\quad v_H\in V_H.
	\end{equation}
	A direct consequence of this construction is that $I_H(\operatorname{id}-\CQ)u_H=u_H=I_H u$.
	
	Before providing an a priori error estimate for this (ideal) generalized finite element method, let us discuss the challenges and open problems with the approach presented so far. These challenges will be addressed in the ensuing sections.
	
	\begin{enumerate}
		\item \textbf{Well-posedness of the corrector problems \eqref{eq:corrector}:} Since $a(\cdot, \cdot)$ is not coercive and only satisfies an inf-sup condition over $H^1_0(\Omega)$, we need to show such an inf-sup condition over the space $W$ as well
		(note that in contrast, coercivity is automatically inherited on $W$ for coercive problem).
		More precisely, we will construct in Section \ref{sec:intpol} below an operator $\UT_H\in \CL(W)$ and show that there is $\alpha_\kappa>0$ (independent of $H$) such that for a sufficiently large contrast, we have 
		\begin{equation}\label{eq:infsupW}
		a(w,\UT_H w)\geq \alpha_\kappa|w|_{1,\sigma, \Omega}^2 \qquad \text{for all}\quad w\in W.
		\end{equation}
		
		\item \textbf{Non-locality of the correctors:}
		The corrector problems \eqref{eq:corrector} are global finescale problems and therefore
		as expensive to solve as the original problem on a fine mesh.
		In Section \ref{subsec:local}, we will show how to localize the computation of the correctors to patches of elements. This localization step is motivated by a decay of the correctors which is exponential in units of $H$.
		Due to the $\UT$-coercivity of our problem, (technical) modifications in the construction of the patches for the localized correctors need to be introduced in comparison to standard elliptic problems.
		
		\item \textbf{Infinite-dimensionality of the fine-scale space $W$:}
		Although the corrector problems are localized in Section \ref{subsec:local} as discussed above, they are not yet ready to use since the space $W$ is still infinite-dimensional.
		In practice we therefore introduce a second, fine triangulation $\CT_h$ of $\Omega$ and discretize the corrector problems using this mesh.
		This final step towards a practical method is discussed in Section~\ref{subsec:fullydiscrete}.
	\end{enumerate}
	
	Because of the second and third challenge we call the generalized finite element method in this section ``ideal''.
	We close its presentation with illustrating its good approximation properties, which will be preserved even through the localization and discretization of the corrector problems.

	\begin{proposition}\label{prop:lodideal}
		Assume that the corrector problems \eqref{eq:corrector} are well-posed, i.e., \eqref{eq:infsupW} holds. Then, we have
		\begin{equation}\label{eq:infsupideal}
		\inf_{v_H\in V_H\setminus\{0\}}\sup_{\psi_H\in V_H\setminus\{0\}}\frac{a((\operatorname{id}-\CQ)v_H, (\operatorname{id}-\CQ)\psi_H)}{|v_H|_{1, \sigma,\Omega}\, |\psi_H|_{1,\sigma, \Omega}}\geq \tilde \alpha.
		\end{equation}
		where $\widetilde \alpha = \alpha C_I^{-2}\, \|\UT\|_{\CL(H^1_0(\Omega))}$ and $C_I$ is the interpolation constant from \eqref{eq:IHstabapprox}.
		Moreover, the unique solution $u_H$ of \eqref{eq:lodideal} fulfills the following error estimate
		\[|u-(\operatorname{id}-\CQ)u_H|_{1, \sigma,\Omega}\leq \alpha_\kappa^{-1} C_I\, \|\UT_H\|_{\CL(W)}\, H\|f\|_{0, \Omega}.\]
	\end{proposition} 
Note that the inf-sup condition automatically implies the well-posedness of \eqref{eq:lodideal}.
Further, we stress that $\|\UT_H\|_{\CL(W)}$ is independent of $H$.
The linear convergence of the error in Proposition
\ref{prop:lodideal} is optimal for lowest-order elements and moreover, this result is
independent of the regularity of the exact solution (which may be arbitrarily low, since
$\sigma \in L^\infty(\Omega)$). Proposition \ref{prop:lodideal} is classical
for the LOD applied to inf-sup stable problems and we refer to \cite[Chapter 2]{Maier20} for a proof.
We emphasize that the assumption $f\in L^2(\Omega)$ is essential to obtain the linear rate, cf.~\cite{Pet15LODreview} for a general discussion.

\subsection{Weak T-coercivity}
\label{sec_weak_Tcoer}
In this paragraph, we briefly discuss how our results transfer to the case that $a(\cdot, \cdot)$
is weakly $\UT$-coercive, which means that instead of \eqref{eq_T_coercivity}, $a(v, \UT v)$ only
satisfies a G\aa rding-type inequality \cite{BCZ10signchangingTcoercive}, namely
\[ a(v,\UT v) \geq \alpha |v|_{1,\sigma,\Omega}^2 - \mu \|v\|_{0,\Omega}^2, \]
where $\alpha,\mu > 0$ are positive constants.
As mentioned, this concept allows to treat rather general sign-changing problems with polytopal interfaces, see \cite{BCC12signchangingTcoercive,BCC18signchangingmesh}.
	We stress that in the case of the Helmholtz equation, one also considers a sesquilinear form satisfying a similar G\aa rding inequality (without an operator $\UT$, though).

Assuming in addition that the solution $u$ to \eqref{eq:modelpb} is unique, i.e., \eqref{eq:modelpb} is well-posed, the problem can
be approximated with the proposed generalized finite element method, but the described
theory does not immediately apply. In particular, the study of the well-posedness of the corrector problems
and their exponential decay requires additional arguments.

However, the Helmholtz equation (with positive coefficients),
was analyzed in \cite{GP15scatteringPG,P17LODhelmholtz,PV19helmholtzhighcontrast}.
In particular, it is shown that the corrector problems are well-posed under a resolution
condition on $H$ because the $L^2$-perturbation in the G\aa rding inequality can be absorbed
for functions in the kernel $W$ due to the property \eqref{eq:IHstabapprox} of $I_H$.
The authors believe that this argument carries over to the weakly $\UT$-coercive setting for problems
with sign-changing coefficients, so that we can establish strong $\UT_H$-coercivity of
$a(\cdot, \cdot)$ over $W$ under a resolution condition (smallness assumption) on $H$.

\section{{\UT}-coercivity in the kernel of $I_H$}
\label{sec:intpol}

	As described above, the LOD method relies on ``corrector'' problems set in
	the kernel $W$ of $I_H$. The purpose of this section is to show that the bilinear
	form $a(\cdot,\cdot)$ is inf-sup stable over $W$. To do so, we build
	a discrete counterpart $\UT_H$ of the the operator $\UT$ that maps
	the kernel $W$ into itself.

\subsection{Preliminary results}
\label{subsec:discretization}

	We start by recording two preliminary results in Lemma \ref{lem:estimatema}
	and Lemma \ref{lem:poincare}. The first is concerned with the scaling of
	the weights $m^\ba$ appearing in the definition of $I_H$, while the second
	is a Poincar\'e-type inequality for functions in $W$. As the proofs are
	rather technical, we postpone them to Appendix \ref{appendix_IH} to ease the reading.
	
	\begin{lemma}\label{lem:estimatema}
		Let $\ba\in \CV_H$. The estimate
		\begin{equation}
		\label{eq_estimate_ma}
		|m^\ba(v)|
		\leq
		\hC_{\mathrm inf}
		\left (\frac{1}{\min_{K \in \CT_H^\ba} |K|}\right )^{1/2}
		\|v\|_{0,\oa}
		\end{equation}
		holds true for all $v \in H^1_0(\Omega)$.
	\end{lemma}
	
	\begin{lemma}\label{lem:poincare}
		Let $\ba\in\CV_H$ and assume that $w \in H^1(\oa)$ satisfies $m^\ba(w)=0$. Then, it holds
		\[\|w\|_{0,\oa}
		\leq
		C_{\mathrm P} H |w|_{1,\oa},
		\]
		where
		\begin{equation*}
		C_{\mathrm P} \eq \frac{\sqrt{\hC_{\mathrm u} C_{\mathrm u}} \hC_{\mathrm P}}{\widehat \rho} |w|_{1,\oa}.
		\end{equation*}
	\end{lemma}

\subsection{A discrete operator $\UT_H$}

	A key ingredient in the construction of the operator $\UT_H$ is
	the introduction of a ``dual weight function'' $\eta^\ba$ associated with
	each vertex $\ba \in \CV_H^+ \cup \CV_H^0$. The purpose of such functions,
	is to ``rectify'' the original $\UT$ operator so that $\UT_H$ maps into
	the kernel $W$ of $I_H$. Importantly, these functions need to be supported
	in $\Omega_+$, so that $\UT_H$ has the same ``symmetrization'' property as $\UT$
	(see \eqref{eq_T_minus}).
	
	The actual construction of the dual functions is technical, so that the proof
	of the following Lemma is delayed until Appendix \ref{appendix_eta}.
	
	\begin{lemma}\label{lem:defetaa}
		For all $\ba\in \CV_H^+\cup \CV_H^0$, there exists $\eta^\ba\in H^1_0(\Omega)$ with $\supp\eta^\ba\subset \Omega_+$ such that
		\[m^{\ba'}(\eta^\ba)=\delta_{\ba',\ba} \qquad \forall \ba'\in \CV_H\]
		and
		\[|\eta^\ba|_{1, \Omega_+}\leq \hC_{\mathrm norm} \hC_{\mathrm inv} \max_{K \in \CT_H^\ba}\frac{|K|^{1/2}}{\rho_K}.\]
	\end{lemma}

We are now ready to introduce our ``discrete'' $\UT$-operator. For $v \in H^1_0(\Omega)$,
it is defined as a modified version of $\UT$ by:
\begin{equation}
\label{eq:defTH}
\UT_H v = \UT v - \sum_{\ba \in \CV_H^0 \cup \CV_H^+} m^\ba(\UT v) \eta^\ba.
\end{equation}

We will establish in the next Section that $a(\cdot, \cdot)$ is indeed $\UT_H$-coercive over $W$.
	In addition, let us remark that, as shown in the appendix, $\supp \eta_\ba \subset \oa \cap \Omega_+$ for all $\ba \in \CV_H$. As a result, we have
	\begin{subequations}
		\label{eq:suppTH}
		\begin{equation}
		\supp(\UT_H v) \cap \Omega_- = \supp(\UT v) \cap \Omega_-
		\end{equation}
		and
		\begin{equation}
		\supp(\UT_H v) \cap \Omega_+ =
		\left \{
		K \in \CT_H \; | \;
		\supp(\UT v) \cap \overline{K} \neq \emptyset
		\right \} \cap \Omega_+
		\quad
		\end{equation}
	\end{subequations}
	for all $v \in H^1_0(\Omega)$.

	\begin{remark}
		An instinctive choice for the discrete operator is $\UT_H \eq (\operatorname{id} - I_H) \circ \UT$,
		or equivalently
		\begin{equation*}
		\UT_H v = \UT v - \sum_{\ba \in \CV_H^0 \cup \CV_H^+} m^\ba(\UT v) \psi^\ba,
		\quad v \in H^1_0(\Omega).
		\end{equation*}
		While this definition automatically maps into $W$, it is not satisfactory, since
		this operator does not flip the sign of the argument in $\Omega_-$ (requirement
		\eqref{eq_T_minus}). Indeed, the corrections at the vertices lying on the interface
		would ``leak'' in $\Omega_-$ as the support of $\psi^\ba$ intersects $\Omega_-$ in
		this case.
	\end{remark}
	
	\begin{remark}
		In the definition of $\UT_H$, the correction functions $\eta^\ba$ are supported
		in $\Op$, since we chose to symmetrize ``from $\Om$ to $\Op$''. If the other direction,
		is considered (i.e. $\US \in \CL(H^1_{\partial \Omega}(\Op); H^1_{\partial \Omega}(\Om))$),
		then these functions have to be supported in $\Om$. As can be seen from the proof of
		Lemma \ref{lem:defetaa} in the appendix, it is easy to design $\eta^\ba$ with support in
		$\Om$ instead of $\Op$, so that every ``direction'' can be considered for symmetrization
		purposes.
	\end{remark}

\subsection{{\UT}-coercivity in the kernel of $I_H$}
\label{subsec:intpol}

	We are now ready to establish the inf-sup stability of $a(\cdot,\cdot)$ over $W$,
	under the assumption that the contrast $\contrast$ is sufficiently large.
	The proof is based on Theorem \ref{prop:T_coercivity} with the operator
	$\UY \eq \UT_H$. Hence, we verify that $\UT_H$ satisfies the requirements
	of Theorem \ref{prop:T_coercivity}.
We first show that $\UT_H \in \CL(W)$.
\begin{lemma}
	We have $\UT_H \in \CL(W)$ with the operator norm bounded independently of $H$.
\end{lemma}

\begin{proof}
	We need to show that $\UT_H w \in W$ for every $w \in W$.
	Let us thus pick an arbitrary $w \in W$, so that
	\begin{equation}
	\label{tmp_weigth_ortho}
	m^\ba(w) = 0 \quad \forall \ba \in \CVi_H.
	\end{equation}
	Then, let $\ba \in \CVi_H$, we have
	\begin{align*}
	m^\ba(\UT_H w)
	&=
	m^\ba(\UT w) - \sum_{\ba' \in \CV_H^0 \cup \CV_H^+}
	m^{\ba'}(\UT w) m^{\ba}(\eta^{\ba'})
	\\
	&=
	m^\ba(\UT w) - \sum_{\ba' \in \CV_H^0\cup\CV_H^+}
	m^{\ba'}(\UT w) \delta_{\ba',\ba},
	\end{align*}
	and it follows that $m^\ba(\UT_H v) = 0$ whenever $\ba \in \CV_H^0 \cup \CV_H^+$.
	If on the other hand $\ba \in \CV_H^-$, recalling
	\eqref{eq_TH_minus} and observing that $\oa \subset \Omega_-$, we have
	\begin{align*}
	m^\ba(\UT_H w) = m^\ba(\UT w) = -m^\ba(w) = 0
	\end{align*}
	since $w \in W$. This shows that $I_H(\UT_H v) = 0$.
	The $H$-independent bound on the operator norm of $\UT_H$ follows by the scalings of $m^\ba$ and $\eta^\ba$
	(see Lemmas \ref{lem:estimatema} and \ref{lem:defetaa}).
\end{proof}

We now verify that $\UT_H$ satisfies requirement \eqref{eq_T} of Theorem
\ref{prop:T_coercivity}.

\begin{lemma}
	\label{lemma_T_H}
	Let $w \in W$, it holds that
	\begin{equation}
	\label{eq_TH_minus}
	\left . \left (\UT_H w\right ) \right |_\Om = -w|_\Om.
	\end{equation}
	In addition, we have
	\begin{equation}
	\label{eq_TH_stab}
	|w-\UT_H w|_{1,\Op} \leq C_\pm(\UT_H) |w|_{1,\Om},
	\end{equation}
	with
	\begin{equation*}
	C_\pm(\UT_H)
	\eq
	C_\pm(\UT)
	+
	2 (d+1) \hC_{\mathrm P} \hC_{\mathrm norm} \hC_{\mathrm inf}\hC_{\mathrm inv} C_{\mathrm u}\kappa
	\sqrt{2 + C_\pm^0(\UT)^2},
	\end{equation*}
	where $\kappa$, $C_\pm(\UT)$, and $C_\pm^0(\UT)$ are introduced in Section \ref{subsec:Tcoercive}
	and the other constants are explained in
	Sections \ref{section_reference_element} and \ref{section_reference_patch}.
\end{lemma}

	\begin{remark}
		We emphasize that the constant $C_\pm(\UT_H)$ is bounded independently of the
		mesh size $H$. Actually, it only depends on the original operator $\UT$, and the
		mesh shape-regularity parameter $\kappa$.
	\end{remark}

\begin{proof}
	Identity \eqref{eq_TH_minus} is a direct consequence from the fact that
	$\supp \eta^\ba \subset \Omega_+$ for all $\ba \in \CV_H^0 \cup \CV_H^+$.
	We thus focus on \eqref{eq_TH_stab}.
	Let $w \in W$. We have
	\begin{align*}
	w-\UT_H w
	&=
	w - \biggl(\UT w - \sum_{\ba \in \CV_H^0 \cup \CV_H^+} m^\ba(\UT w) \eta^\ba\biggr )
	\\*
	&=
	\biggl ( w - \UT w \biggr ) - \sum_{\ba \in \CV_H^0 \cup \CV_H^+} m^\ba(w-\UT w) \eta^\ba,
	\end{align*}
	so that
	\begin{equation*}
	|w - \UT_H w|_{1,\Op}
	\leq
	|w - \UT w|_{1,\Op}
	+
	\biggl |
	\sum_{\ba \in \CV_H^0 \cup \CV_H^+} m^\ba(w-\UT w) \eta^\ba
	\biggr |_{1,\Op}.
	\end{equation*}
	We have
	\begin{equation*}
	\biggl |
	\sum_{\ba \in \CV_H^0 \cup \CV_H^+} (w-\UT w,m^\ba) \eta^\ba
	\biggr |_{1,\Omega_+}^2
	\leq
	(d+1)
	\sum_{\ba \in \CV_H^0 \cup \CV_H^+} |m^\ba(w-\UT w)|^2|\eta^\ba|_{1,\Op}^2.
	\end{equation*}
	Then, for each $\ba \in \CV_H^0 \cup \CV_H^+$, it holds with Lemmas
	\ref{lem:estimatema} and \ref{lem:defetaa} that
	\begin{align*}
	|m^\ba(w-\UT w)||\eta^\ba|_{1,\Op}
	&\leq
	\left (
		\hC_{\mathrm inf}
		\left (\frac{1}{\min_{K \in \CT_H^\ba} |K|}\right )^{1/2}
		\|w-\UT w\|_{0,\oa}
		\right )
		\left (\hC_{\mathrm norm}\hC_{\mathrm inv} \max_{K \in \CT_H^\ba} \frac{|K|^{1/2}}{\rho_K}\right )
	\\
	&\leq
	\hC_{\mathrm norm} \hC_{\mathrm inf} \hC_{\mathrm inv}
	\left (\frac{\max_{K \in \CT_H^\ba} |K|}{\min_{K \in \CT_H^\ba} |K|}\right )^{1/2}
	\frac{1}{\rho} \|w-\UT w\|_{0,\oa}
	\\
	&\leq
	\frac{\hC_{\mathrm norm} \hC_{\mathrm inf} \hC_{\mathrm inv}C_{\mathrm u}^{1/2}}{\rho}
	\|w-\UT w\|_{0,\oa}.
	\end{align*}
	Furthermore, we have
	\begin{align*}
	\|w-\UT w\|_{0,\oa}^2
	&=
	\|w-\UT w\|_{0,\oa \cap \Om}^2
	+
	\|w-\UT w\|_{0,\oa \cap \Op}^2
	\\
	&=
	2\|w\|_{0,\oa \cap \Om}^2
	+
	\|w-\UT w\|_{0,\oa \cap \Op}^2.
	\end{align*}
	Therefore, we obtain by combining these two estimates
	\begin{align*}
	\biggl |
	\sum_{\ba \in \CV_H^0 \cup \CV_H^+} (w-\UT v,m^\ba) \eta^\ba
	\biggr |^2
	&\leq
	\Bigl (
	\frac{\hC_{\mathrm norm} \hC_{\mathrm inf} \hC_{\mathrm inv}C_{\mathrm u}^{1/2}}{\rho}
	\Bigr )^2 (d+1)^2
	\bigl (2 \|w\|_{0,\Om}^2 + \|w-\UT w\|_{0,\Op}^2\bigr )
	\\
	&\leq
	\Bigl (
	\frac{\hC_{\mathrm norm} \hC_{\mathrm inf} \hC_{\mathrm inv}C_{\mathrm u}^{1/2}}{\rho}
	\Bigr )^2 (d+1)^2
	\left (2 + C_\pm^0(\UT)^2\right )\|w\|_{0,\Om}^2.
	\end{align*}
	Moreover, we have by Lemma \ref{lem:poincare}
	\begin{align*}
	\|w\|_{0,\Om}^2
	&\leq
	\frac{1}{d+1} \sum_{\ba \in \CV_H^-} \|w\|_{0,\oa}^2
	\leq
	\frac{(\widehat C_{\mathrm{P}} H)^2}{d+1} \sum_{\ba \in \CV_H^-} |w|_{1,\oa}^2
	\leq
	\widehat C_{\mathrm{P}}^2 H^2 |w|^2_{1,\Om}.
	\end{align*}
	Hence, combining all the foregoing estimates, we finally deduce
	\begin{align*}
	\biggl |
	\sum_{\ba \in \CV_H^0 \cup \CV_H^+} (w-\UT w,m^\ba) \eta^\ba
	\biggr |
	&\leq
	(d+1)\hC_{\mathrm P}\hC_{\mathrm norm} \hC_{\mathrm inf} \hC_{\mathrm inv}C_{\mathrm u}^{1/2}
	\frac{H}{\rho}
	\sqrt{2 + C_\pm^0(\UT)^2}\,|w|_{1,\Om},
	\end{align*}
	and the result follows.
\end{proof}

	We can now conclude this section with Theorem \ref{theorem_inf_sup_kernel}
	establishing $\UT_H$-coercivity of $a(\cdot,\cdot)$ in the kernel $W$.
	The proof is direct consequence of Theorem \ref{prop:T_coercivity} and Lemma \ref{lemma_T_H}.

	\begin{theorem}
		\label{theorem_inf_sup_kernel}
		Under the assumption that
		\begin{equation}
		\label{eq:contrast}
		\sqrt{\contrast} > \frac{C_\pm(\UT_H)}{2}
		\left (\frac{\sup_\Op \sigma}{\inf_\Op \sigma}\right ),
		\end{equation}
		we have
		\begin{equation}\label{eq:WTcoercive}
		a(w,\UT_H w) \geq \alpha_\kappa |w|_{1,\sigma,\Omega}^2 \qquad \forall w \in W
		\end{equation}
		where
		\begin{equation*}
		\alpha_\kappa \eq 1 - \frac{C_\pm(\UT_H)}{2}
		\left (\frac{\sup_\Op \sigma}{\inf_\Op \sigma}\right )
		\sqrt{\frac{\Sm}{\Sp}}.
		\end{equation*}
	\end{theorem}

\section{Towards a practical method}
\label{sec:lodpractical}
In this section, we address the second and third challenge discussed in Section \ref{subsec:ideal}, namely the localization of the corrector computations and their discretization.
To avoid the proliferation of constants, we use the notation $a\lesssim b$ (resp. $a \gtrsim b$)
if $a\leq Cb$ (resp. $a \geq Cb$) with a constant $C$ that only depends on $\kappa$, $\alpha_\kappa$, $\sigma_+$, $\sigma_-$, and $\|\sigma\|_{L^\infty(\Omega)}$.
We also write $a\approx b$ when $a\lesssim b$ and $a\gtrsim b$.

\subsection{Localized correctors}
\label{subsec:local}
In this section, we will show how to localize the computation of the correctors defined in \eqref{eq:corrector}.
Note that due to linearity, $\CQ$ can be written as $\CQ=\sum_{K\in \CT_H}\CQ_K$, where $\CQ_K$ is defined via
\[a(\CQ_K v_H, w)=a_K(v_H, w)\qquad \text{for all} \quad w\in W.\]
Here and in the following, $a_D(\cdot, \cdot)$ denotes the restriction of $a(\cdot, \cdot)$ to a subdomain
	$D\subset \Omega$.

We emphasize that the present localization analysis requires a dedicated treatment,
due to the underlying usage of $\UT$-coercivity. Indeed, the arguments for general
inf-sup stable problems presented in \cite[Chapter~2]{Maier20} requires a ``locality assumption''
in the inf-sup condition. This locality assumption essentially requires that
for $w \in W$, there exists a function $w^\star \in W$ that realizes the inf-sup
condition such that $\|w^\star\|_{1,D} \lesssim \|w\|_{1,\widetilde D}$ for $D \subset \Omega$,
where $\widetilde D$ is a slightly ``oversampled'' version of $D$. In view of the nature of
the operator $\UT$, that involves a symmetrization around $\Gamma$, this assumption is fundamentally
violated here.

Recall the definition of the $m$th layer patch $\UN^m(D)$ around $D\subset \Omega$ from Section \ref{subsec:meshes}.
The shape regularity implies that there is a bound $C_{\operatorname{ol}, m}$ (depending only on $m$) of the number of the elements in the $m$-layer patch, i.e.,
\begin{equation}
\label{eq:Colm}
\max_{T\in \CT_H}\operatorname{card}\{K\in \CT_H\quad |\quad K\subset \UN^m(T)\}\leq C_{\operatorname{ol}, m}.
\end{equation}
We note that since $\CT_H$ is quasi-uniform, $C_{\operatorname{ol}, m}$ grows at most polynomially
with $m$.

As stated above, we need to modify the usual proof because $\UT_H$ involves a symmetrization
operator and thus, is inherently non-local. This is why we introduce the following ``symmetric''
patches $\UP^m(K):=(\UP^m(K)\cap \Omega_-)\cup (\UP^m(K)\cap \Omega_+)$ by
\[\UP^m(K) \cap \Omega_- := \UN^m(K) \cap \Omega_-,\]
and
\[\UP^m(K) \cap \Omega_+ :=\{K^\prime\in \CT_H\quad|\quad K^\prime\cap\supp(\UT v)\neq \emptyset \quad\text{ for all }\quad v\in H^1_0(\UN^m(K))\} \cap \Omega_+.\]
We emphasize that this does not require the mesh $\CT_H$ to be symmetric.
In view of \eqref{eq:suppTH}, the idea of $\UP^m(K)$ is that, for any function $v\in H^1_0(\Omega)$ with
$\supp v\subset \UP^m(K)$ we now have $\supp \UT_H v\subset \UP^m(K)$ as well.
Some examples of $\UP^1(K)$ for an interface-resolving, but non-symmetric mesh are illustrated in Figure \ref{fig:patch}.

We now have an exponential decay of $\CQ_K$ outside those symmetric patches, as stated in
the following proposition, whose proof is postponed to Section \ref{subsec:decay}.

\begin{proposition}\label{prop:expdecay}
	There is $0<\tilde{\gamma}<1$, independent of $H$, such that for any $K\in \CT_H$ and all $v_H\in V_H$
	\[\|\CQ_K v_H\|_{1,\Omega\setminus\UP^m(K)}\lesssim \tilde{\gamma}^m\|v_H\|_{1,K}.\]
\end{proposition}

In order to localize the corrector problems, we introduce the space
\[W(\UP^m(K)):=\{w\in W\quad |\quad w=0\quad\text{in}\quad \Omega\setminus \UP^m(K)\}\]
and define for any $v_H\in V_H$ the localized element corrector $\CQ_{K,m}v_H\in W(\UP^m(K))$
as the solution of
\begin{equation}\label{eq:correctorlocal}
a_{\UP^m(K)}(\CQ_{K,m}v_H, w)=a_K(v_H, w)\qquad \text{for all}\quad w\in W(\UP^m(K)).
\end{equation}
Due to $\UT_H\in \CL(W)$ and the definition of $\UP^m(K)$, these
localized corrector problems are well-posed because the $\UT_H$-coercivity of
$a(\cdot,\cdot)$ thereby carries over from $W$ to $W(\UP^m(K))$.

We emphasize that, if $\UN^m(K)\cap \Gamma=\emptyset$, $\UP^m(K)$ consists of two
disconnected domains and $\CQ_{K,m}v_H$ is even zero outside the standard patch
$\UN^m(K)$ because of the localized right-hand side in \eqref{eq:correctorlocal}. 
Hence, we can solve \eqref{eq:correctorlocal} on $\UN^m(K)$ (as in the usual LOD)
in the case $\UN^m(K)\cap \Gamma=\emptyset$, resulting in the standard localized element
corrector problems. In other words, we only need to define new and larger patches for $Q_{K,m}$
for elements $K$ close to the interface $\Gamma$. The truncated correction operator $\CQ_m$
is now defined as the sum of these element correctors, i.e., $\CQ_m:=\sum_{K\in \CT_H}\CQ_{K,m}$.

\begin{figure}
	\begin{subfigure}{0.3\textwidth}
		\centering
		\begin{tikzpicture}[scale=0.7]
		\path[fill=lightgray] (0,2)--(3,2)--(3,0)--(0,0)--cycle;
		\path[fill=gray] (0,4)--(2,4)--(3, 5)--(3,6)--(1,6)--(0,5)--cycle;
		\path[fill=black] (1, 5)--(2,5)--(2,6)--cycle;
		\draw[thick, dashed, blue] (0,2)--(2,2)--(3,1)--(3,0)--(1,0)--(0,1)--cycle;
		\draw(0,0) rectangle (4,6);
		\draw(1,0)--(1,6);
		\draw(2,0)--(2,6);
		\draw(3,0)--(3,6);
		\draw(0,1)--(4,1);
		\draw(0,2)--(4,2);
		\draw[very thick, red](0,3)--(4,3);
		\draw(0,4)--(4,4);
		\draw (0,5)--(4,5);
		\draw (0,0)--(4,4);
		\draw (0,1)--(4,5);
		\draw(0,2)--(4,6);
		\draw(0,3)--(3,6);
		\draw (0,4)--(2,6);
		\draw(0,5)--(1,6);
		\draw(1,0)--(4,3);
		\draw(2,0)--(4,2);
		\draw(3,0)--(4,1);
		\end{tikzpicture}
	\end{subfigure}
	\begin{subfigure}{0.3\textwidth}
		\centering
		\begin{tikzpicture}[scale=0.7]
		\path[fill=gray] (0,0)--(0,2)--(1, 3)--(3,3)--(3,2)--(1,0)--cycle;
		\path[fill=black] (1, 1)--(1,2)--(2,2)--cycle;
		\draw(0,0) rectangle (4,6);
		\draw(1,0)--(1,6);
		\draw(2,0)--(2,6);
		\draw(3,0)--(3,6);
		\draw(0,1)--(4,1);
		\draw(0,2)--(4,2);
		\draw[very thick, red](0,3)--(4,3);
		\draw(0,4)--(4,4);
		\draw (0,5)--(4,5);
		\draw (0,0)--(4,4);
		\draw (0,1)--(4,5);
		\draw(0,2)--(4,6);
		\draw(0,3)--(3,6);
		\draw (0,4)--(2,6);
		\draw(0,5)--(1,6);
		\draw(1,0)--(4,3);
		\draw(2,0)--(4,2);
		\draw(3,0)--(4,1);
		\end{tikzpicture}
	\end{subfigure}
	\begin{subfigure}{0.3\textwidth}
		\centering
		\begin{tikzpicture}[scale=0.7]
		\path[fill=lightgray] (0,3)--(3,3)--(3,0)--(1,0)--(1,1)--(0,1)--cycle;
		\path[fill=gray] (0,3)--(2,3)--(3, 4)--(3,6)--(2,6)--(0,4)--cycle;
		\path[fill=black] (1, 4)--(2,4)--(2,5)--cycle;
		\draw[thick, dashed, blue] (0,3)--(2,3)--(3,2)--(3,0)--(2,0)--(0,2)--cycle;
		\draw(0,0) rectangle (4,6);
		\draw(1,0)--(1,6);
		\draw(2,0)--(2,6);
		\draw(3,0)--(3,6);
		\draw(0,1)--(4,1);
		\draw(0,2)--(4,2);
		\draw[very thick, red](0,3)--(4,3);
		\draw(0,4)--(4,4);
		\draw (0,5)--(4,5);
		\draw (0,0)--(4,4);
		\draw (0,1)--(4,5);
		\draw(0,2)--(4,6);
		\draw(0,3)--(3,6);
		\draw (0,4)--(2,6);
		\draw(0,5)--(1,6);
		\draw(1,0)--(4,3);
		\draw(2,0)--(4,2);
		\draw(3,0)--(4,1);
		\end{tikzpicture}
	\end{subfigure}
	
	\vspace{2ex}
	\begin{subfigure}{0.3\textwidth}
		\centering
		\begin{tikzpicture}[scale=0.7]
		\path[fill=gray] (0,2)--(2,4)--(3, 4)--(3,2)--(2,1)--(0,1)--cycle;
		\path[fill=black] (1, 2)--(2,2)--(2,3)--cycle;
		\draw (0,0) rectangle (4,6);
		\draw(1,0)--(1,6);
		\draw(2,0)--(2,6);
		\draw(3,0)--(3,6);
		\draw(0,1)--(4,1);
		\draw(0,2)--(4,2);
		\draw[very thick, red](0,3)--(4,3);
		\draw(0,4)--(4,4);
		\draw (0,5)--(4,5);
		\draw (0,0)--(4,4);
		\draw (0,1)--(4,5);
		\draw(0,2)--(4,6);
		\draw(0,3)--(3,6);
		\draw (0,4)--(2,6);
		\draw(0,5)--(1,6);
		\draw(1,0)--(4,3);
		\draw(2,0)--(4,2);
		\draw(3,0)--(4,1);
		\end{tikzpicture}
	\end{subfigure}
	\begin{subfigure}{0.3\textwidth}
		\centering
		\begin{tikzpicture}[scale=0.7]
		\path[fill=lightgray] (2,2)--(3,3)--(3,2)--cycle;
		\path[fill=gray] (0,3)--(1,4)--(3, 4)--(3,3)--(1,1)--(0,1)--cycle;
		\path[fill=black] (1, 2)--(1,3)--(2,3)--cycle;
		\draw(0,0) rectangle (4,6);
		\draw(1,0)--(1,6);
		\draw(2,0)--(2,6);
		\draw(3,0)--(3,6);
		\draw(0,1)--(4,1);
		\draw(0,2)--(4,2);
		\draw[very thick, red](0,3)--(4,3);
		\draw(0,4)--(4,4);
		\draw (0,5)--(4,5);
		\draw (0,0)--(4,4);
		\draw (0,1)--(4,5);
		\draw(0,2)--(4,6);
		\draw(0,3)--(3,6);
		\draw (0,4)--(2,6);
		\draw(0,5)--(1,6);
		\draw(1,0)--(4,3);
		\draw(2,0)--(4,2);
		\draw(3,0)--(4,1);
		\end{tikzpicture}
	\end{subfigure}
	\begin{subfigure}{0.3\textwidth}
		\centering
		\begin{tikzpicture}[scale=0.7]
		\path[fill=lightgray] (1,2)--(1,1)--(3,1)--(3,3)--(2,2)--cycle;
		\path[fill=gray] (0,2)--(2,2)--(3, 3)--(3,5)--(2,5)--(0,3)--cycle;
		\path[fill=black] (1, 3)--(2,3)--(2,4)--cycle;
		%\draw[thick, dashed, blue] (0,2)--(2,2)--(3,1)--(3,0)--(1,0)--(0,1)--cycle;
		\draw(0,0) rectangle (4,6);
		\draw(1,0)--(1,6);
		\draw(2,0)--(2,6);
		\draw(3,0)--(3,6);
		\draw(0,1)--(4,1);
		\draw(0,2)--(4,2);
		\draw[very thick, red](0,3)--(4,3);
		\draw(0,4)--(4,4);
		\draw (0,5)--(4,5);
		\draw (0,0)--(4,4);
		\draw (0,1)--(4,5);
		\draw(0,2)--(4,6);
		\draw(0,3)--(3,6);
		\draw (0,4)--(2,6);
		\draw(0,5)--(1,6);
		\draw(1,0)--(4,3);
		\draw(2,0)--(4,2);
		\draw(3,0)--(4,1);
		\end{tikzpicture}
	\end{subfigure}
	\caption{Illustration of $\UP^1(K)$ for different triangles $K$. The red line is the interface $\Gamma$, $\Omega_-$ is the upper half and $\Omega_+$ the lower half. Triangle $K$ in black, $\UN^1(K)$ consists of $K$ and additional elements in gray, $\UP^1(K)$ consists of $\UN^1(K)$ and additional elements in light gray. In the top line, dashed blue lines indicate the area of $\UN^1(K)$ under symmetrization.}
	\label{fig:patch}
\end{figure}
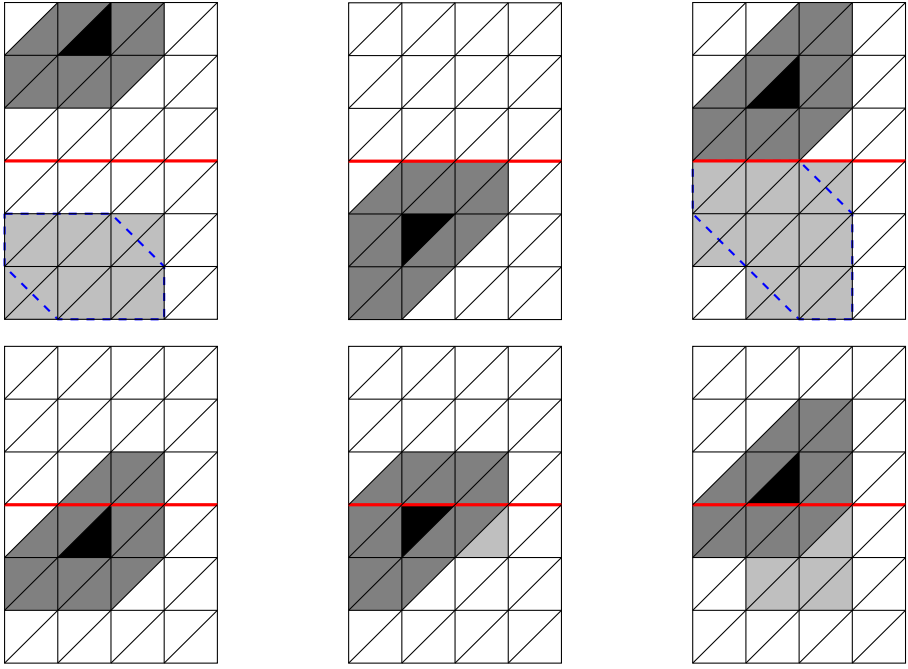

Due to the exponential decay of the idealized correctors, we have the following estimate of the truncation or localization error, which again is proved in Section \ref{subsec:decay}.
\begin{theorem}\label{thm:localization}
	There exists $0<\gamma <1$, independent of $H$, such that for any $v_H\in V_H$
	\[\|(\CQ-\CQ_m)v_H\|_{1, \Omega}\lesssim C_{\mathrm{ol}, m}^{1/2}\, \gamma^m \|v_H\|_{1, \Omega}.\]
\end{theorem}

In our generalized finite element method, we now replace $\CQ$ in \eqref{eq:lodideal} by $\CQ_m$,
exactly in the spirit of LOD. Hence, we seek $u_{H,m}\in V_H$ such that
\begin{equation}\label{eq:lod}
a((\operatorname{id}-\CQ_m)u_{H,m}, (\operatorname{id}-\CQ_m)v_H)=(f, (\operatorname{id}-\CQ_m)v_H)_\Omega\qquad \text{for all}\quad v_H\in V_H.
\end{equation} 
The numerical analysis relies on the error estimate for the ideal method in Proposition \ref{prop:lodideal} and the fact that the localization is a small perturbation thereof.

\begin{theorem}\label{thm:lod}
	Let $m\gtrsim |\log (C_{\operatorname{ol}, m}^{1/2} \tilde \alpha)|$ with the inf-sup constant $\tilde \alpha$ of Proposition \ref{prop:lodideal}.
	Then 
	\eqref{eq:lod} is well-defined and the unique solution $u_{H,m}$ satisfies the error estimates
	\begin{align}
	\|u-(\operatorname{id}-\CQ_m)u_{H,m}\|_{1, \Omega}&\lesssim (H+C_{\operatorname{ol}, m}^{1/2}\,\gamma^m)\|f\|_{0, \Omega}, \label{eq:errorLOD}\\*
	\|u-u_{H,m}\|_{0, \Omega}&\lesssim H\inf_{v_H\in V_H} |u-v_H|_{1, \Omega}+C_{\operatorname{ol},m}^{1/2}\,\gamma^m (H+C_{\operatorname{ol}, m}^{1/2}\gamma^m)\|f\|_{0, \Omega}. \label{eq:errormacroLOD}
	\end{align}
\end{theorem}
Note that the oversampling condition $m\gtrsim |\log (C_{\operatorname{ol}, m}^{1/2} \tilde\alpha)| $ is independent of $H$.
Since $C_{\operatorname{ol}, m}$ grows only polynomially in $m$, it is fulfillable.
We emphasize that, in order to balance the terms $H$ and $\gamma^m$ in the error estimates, the stronger, but standard, oversampling condition $m\approx |\log(C_{\operatorname{ol}, m}^{1/2} H)|$ is required.
We summarize that under this (standard) oversampling condition, the method is well-posed, we have linear convergence in the $H^1(\Omega)$-norm (see \eqref{eq:errorLOD}) and up to quadratic convergence of the FE part in the $L^2(\Omega)$-norm (see \eqref{eq:errormacroLOD}).
Note that the second term in \eqref{eq:errormacroLOD} is of order $H^2$ for $m\approx |\log(C_{\operatorname{ol}, m}^{1/2} H)|$.
The exact convergence rate for the FE part depends on the (higher) regularity of the model problem (encoded in the best approximation of $V_H$), but we have at least linear convergence.
To be more precise, \eqref{eq:errormacroLOD} gives a convergence order of $H^{1+s}$ if the exact solution is in $H^{1+s}(\Omega)$. This should be contrasted with the convergence order $H^{2s}$ in $L^2(\Omega)$ for the standard FEM.

\begin{proof}
	The well-posedness of \eqref{eq:lod} follows from an inf-sup condition on $V_{H,m}$
	(see \cite{PV19helmholtzhighcontrast} for instance). This directly yields quasi-optimality
	and the error estimate \eqref{eq:errorLOD}, where we refer to \cite[Chapter~2]{Maier20} for details.
	
	Moreover, a standard duality argument can be employed to show \[\|u-(\operatorname{id}-\CQ_m)u_{H,m}\|_{0, \Omega}\lesssim (H+C_{\operatorname{ol}, m}^{1/2}\,\gamma^m)\|u-(\operatorname{id}-\CQ_m)u_{H,m}\|_{1, \Omega},\]
	i.e., quadratic convergence in the $L^2(\Omega)$-norm. We refer to, e.g., \cite{PV19helmholtzhighcontrast} for details.
	
	Finally, we have that
	\[\|u-u_{H,m}\|_{0, \Omega}\leq \|u-I_H u\|_{0, \Omega}+\|I_H u-u_{H,m}\|_{0, \Omega}\lesssim H|u-I_Hu|_{1, \Omega}+\|I_H u-u_{H,m}\|_{1, \Omega}.\]
	Due to the stability and projection property of $I_H$, we have $|u-I_Hu|_{1, \Omega}\lesssim \inf_{v_H\in V_H}|u-v_H|_{1, \Omega}$ so that it remains to estimate $\|I_H u-u_{H,m}\|_{1, \Omega}$.
	We note that by the definition of $\CQ$ and the stability of $I_H$ it holds that
	\[\|I_H u-u_{H,m}\|_{1, \Omega}=\|I_H(\operatorname{id}-\CQ)(I_H u-u_{H,m})\|_{1, \Omega}\lesssim \|(\operatorname{id}-\CQ)(I_H u-u_{H,m})\|_{1, \Omega}.\]
	Due to Proposition \ref{prop:lodideal}, there exists $\psi_H\in V_H$ with $\|\psi_H\|_{1, \Omega}=1$ such that
	\[\|(\operatorname{id}-\CQ)(I_H u-u_{H,m})\|_{1, \Omega}\lesssim a((\operatorname{id}-\CQ)(I_H u-u_{H,m}), (\operatorname{id}-\CQ)\psi_H).\]
	The definition of $\CQ$, Galerkin orthogonality and Theorem \ref{thm:localization} give that
	\begin{align*}
	\|(\operatorname{id}-\CQ)(I_H u-u_{H,m})\|_{1, \Omega}&\lesssim a((\operatorname{id}-\CQ)I_H u-(\operatorname{id}-\CQ)u_{H,m}, (\operatorname{id}-\CQ)\psi_H)\\
	&=a(u-(\operatorname{id}-\CQ_m)u_{H,m}, (\operatorname{id}-\CQ)\psi_H)\\
	&=a(u-(\operatorname{id}-\CQ_m)u_{H,m}, (\CQ_m-\CQ)\psi_H)\\
	&\lesssim C_{\operatorname{ol}, m}^{1/2}\,\gamma^m \|u-(\operatorname{id}-\CQ_m)u_{H,m}\|_{1, \Omega}.
	\end{align*}
	Combination with the estimate for $\|u-(\operatorname{id}-\CQ_m)u_{H,m}\|_{1, \Omega}$ finishes the proof.
\end{proof}

\subsection{Fully discrete method}
\label{subsec:fullydiscrete}
We now address the final challenge to obtain a fully practical generalized finite element method: the fact that the spaces $W$ and $W(\UP^m(K))$ are still infinite-dimensional.
In practice we therefore introduce a second, fine triangulation $\CT_h$ of $\Omega$ as well as the corresponding Lagrange finite element space $V_h$.
$\CT_h$ should be a shape-regular refinement of $\CT_H$, but note that $\CT_h$ is not required to be quasi-uniform.
The corrector problems \eqref{eq:correctorlocal} are then defined on the discrete space $W(\UP^m(K))\cap V_h$ and yield discrete localized correctors $\CQ_{m,h}$.

This requires the mesh $\CT_h$ to be sufficiently fine in the sense that all multiscale features and jumps of $\sigma$ are resolved, and in particular it needs to be
$\UT$-conforming. We point out that then,
$\UT(V_h) \subset V_h$, and one easily checks that $\UT_H(W \cap V_h) \subset (W\cap V_h)$.
As a result, the authors strongly believe the above analysis will still hold true with minor
modifications due to the additional discretization. We refer the reader to \cite{GP15scatteringPG}
for details on the proof of the exponential decay in this case.

The corresponding solution $u_{H,h,m}$ of our generalized finite element method \eqref{eq:lod} then approximates the FEM solution $u_h\in V_h$ on the fine mesh.
In particular, we have by the triangle inequality that \[\|u-(\operatorname{id}-\CQ_{m,h})u_{H,h,m}\|_{1, \Omega}\leq \|u-u_h\|_{1, \Omega}+\|u_h-(\operatorname{id}-\CQ_{m,h})u_{H,h,m}\|_{1, \Omega}.\]
	With the above mentioned modifications, an estimate for $\|u_h-u_{H,h,m}\|_{1, \Omega}$ similar to Theorem \ref{thm:lod} should hold, namely
	\[\|u_h-(\operatorname{id}-\CQ_{m,h})u_{H,h,m}\|_{1, \Omega}\lesssim (H+C^{1/2}_{\operatorname{ol},m}\gamma^m)\|f\|_{0, \Omega} \]
	with a constant hidden in $\lesssim$ that is independent of $H$, $h$, and $m$.
	Since $\CT_h$ is a fine, not necessarily quasi-uniform, and $\UT$-conforming triangulation, it is reasonable to assume that the finescale discretization error $\|u-u_h\|_{1, \Omega}$ is sufficiently small in comparison to the LOD error $\|u_h-u_{H,h,m}\|_{1, \Omega}$.
	Finally, we note that $u_h$ is not needed for computing $u_{H,h,m}$. However, in numerical experiments where often $u$ is not available, we use $u_h$ as reference solution and evaluate the error $\|u_h-(\operatorname{id}-\CQ_{m,h})u_{H,h,m}\|_1$ only.
	
	Concerning the practical implementation of the LOD, we refer the interested reader to \cite{EHMP16LODimpl}, where, for instance, the (parallel) computation of the correctors is addressed in detail.
	In comparison with a standard finite element method on a fine (adaptive) mesh, our method has the advantage of a much smaller linear system to be solved at the cost of a slightly more dense matrix and additional computations (in form of the local correctors) during the assembly of the stiffness matrix.
	Therefore, the method is particularly attractive if a standard finite element method on a fine grid is not feasible due to the size of the system or if the same multiscale problem has to be solved for many different right-hand sides.

\subsection{Proof of the localization error}
\label{subsec:decay}
This section is devoted to the proofs of Proposition \ref{prop:expdecay} and Theorem \ref{thm:localization}.
In the proofs we will frequently make use of cut-off functions. We collect some properties for them in the following.
Let $\eta\in H^1(\Omega)$ be a function with values in the interval $[0,1]$ satisfying the bound $\|\grad \eta\|_{L^\infty(\Omega)}\lesssim H^{-1}$ and let $\mathcal{R}:= \supp(\grad \eta)$.
Given any subset $D\subset \Omega$ as the union of elements in $\CT_H$, any $w\in W$ satisfies that
\begin{align}
\|w\|_{0,D}&\lesssim H\|\grad w\|_{0,\UN(D)}, \label{eq:poincarekernel}\\
\|(\operatorname{id}-I_H)(\eta w)\|_{0,D}&\lesssim H\|\grad(\eta w)\|_{0,\UN(D)},\label{eq:intpolcutoff}\\
\|\grad(\eta w)\|_{0,D}&\lesssim \|\grad w\|_{0,D\cap \supp \eta}+\|\grad w\|_{0,\UN(D\cap \mathcal{R})}. \label{eq:gradcutoff}
\end{align}
These properties are proved in \cite[Lemma 2]{GP15scatteringPG}.

\begin{proof}[Proof of Proposition \ref{prop:expdecay}]
	Fix $K\in \CT_H$, $v_H\in V_H$ and $m$.
	Set $\phi:=\CQ_K v_H\in W$ and $\widetilde{\phi}=(\operatorname{id}-I_H)(\eta \phi)$ with the piecewise linear and globally continuous cut-off function $\eta$ defined via
	\[\eta=0 \quad\text{in}\quad \UP^{m-4}(K),\qquad \eta=1\quad\text{in}\quad \Omega\setminus \UP^{m-3}(K).\]
	We write $\mathcal{R}=\supp(\grad \eta)$ and use in the following $\UN^k(\mathcal{R})=\UP^{m-3+k}(K)\setminus \UP^{m-4-k}(K)$.
	Note that $\|\grad \eta\|_{L^\infty(\mathcal{R})}\lesssim H^{-1}$.
	Then
	\[\|\phi\|_{1, \Omega\setminus \UP^m(K)}=\|\phi-I_H\phi\|_{1, \Omega\setminus \UP^m(K)}\leq \|\widetilde{\phi}\|_{1, \Omega}.\]
	We have $\UT_H\widetilde{\phi}\in W$ with support outside $K$ due to the definition of $\UP^m(K)$.
	Hence,
	\begin{align*}
	\|\phi\|^2_{1, \Omega\setminus \UP^m(K)}&\leq \|\widetilde{\phi}\|^2_{1, \Omega}\lesssim \alpha_\kappa^{-1}a(\widetilde{\phi}, \UT_H\widetilde{\phi})=\alpha_\kappa^{-1}a(\widetilde{\phi}-\phi, \UT_H\widetilde{\phi}).
	\end{align*}
	Note that $\supp(\widetilde{\phi}-\phi)\cap \supp(\UT_H\widetilde{\phi})\subset\UN^{1}(\mathcal{R})$ and $\|\UT_H \tilde{\phi}\|_{\UN(\mathcal{R})}\lesssim \|\tilde{\phi}\|_{1, \UN^2(\mathcal{R})}$ due to the definitions of $\UP^m(K)$ and $\mathcal{R}$.
	Hence, we obtain with the continuity of $a(\cdot, \cdot)$ 
	\begin{align*}
	\alpha_\kappa \|\phi\|^2_{1, \Omega\setminus\UP^m(K)}&\lesssim \|\widetilde{\phi}-\phi\|_{1,\UN^{1}(\mathcal{R})}\,\|\UT_H\widetilde{\phi}\|_{1, \UN^{1}(\mathcal{R})}\\*
	&\lesssim \|\widetilde{\phi}-\phi\|_{1, \UN^{1}(\mathcal{R})}\, (\|\widetilde{\phi}-\phi\|_{1, \UN^{2}(\mathcal{R})}+\|\phi\|_{1, \UN^{2}(\mathcal{R})}).
	\end{align*}
	Employing that $I_H \phi=0$ and the properties \eqref{eq:intpolcutoff} as well as \eqref{eq:gradcutoff}, we deduce
	\[\|\widetilde{\phi}-\phi\|_{1, \UN^{2}(\mathcal{R})}=\|(\operatorname{id}-I_H)((1-\eta)\phi)\|_{1, \UN^2(\mathcal{R})}\lesssim \|\phi\|_{1,\UN^{3}(\mathcal{R})}\]
	and analogously $\|\widetilde{\phi}-\phi\|_{1, \UN^{1}(\mathcal{R})}\lesssim \|\phi\|_{1,\UN^{2}(\mathcal{R})}$.
	All in all, this gives
	\[\|\phi\|^2_{1,\Omega\setminus\UP^m(K)}\leq \tilde{C}\|\phi\|^2_{1,\UP^m(K)\setminus \UP^{m-7}(K)}=\tilde{C}\|\phi\|^2_{1,\Omega\setminus \UP^{m-7}(K)}-\tilde{C}\|\phi\|^2_{1,\Omega\setminus\UP^m(K)}\]
	for some constant $\tilde{C}$.
	This yields
	\[\|\phi\|^2_{1,\Omega\setminus\UP^m(K)}\leq \frac{\tilde{C}}{1+\tilde{C}}\,\|\phi\|^2_{1,\Omega\setminus \UP^{m-7}(K)}\]
	The repeated application of this argument finishes the proof with $\tilde{\gamma}=\frac{\tilde{C}}{1+\tilde{C}}<1$.
\end{proof}

Note that the constant hidden in $\lesssim$ in Proposition \ref{prop:expdecay} depends on the interpolation constant, the norm of $\UT_H$, the continuity constant of $a(\cdot, \cdot)$ and on $\alpha_\kappa^{-1}$. In particular the latter may become very large depending on the contrast, see \cite{CV18signchangingapost} and Section \ref{sec:numexperiment}.

\begin{proof}[Proof of Theorem \ref{thm:localization}]
	We start by proving the following local estimate
	\begin{equation}\label{eq:truncerrorlocal}
	\|(\CQ_K-\CQ_{K,m})v_H\|_{1, \Omega}\lesssim \tilde{\gamma}^m \|v_H\|_{1,K}
	\end{equation}
	for some $0<\tilde{\gamma}<1$ and for any $v_H\in V_H$ and $K\in \CT_H$ as well as any (fixed) $m$.
	Note that $\CQ_{K,m} v_H$ is the Galerkin approximation of $\CQ_K v_H$ on the subspace $W(\UP^m(K))\subset W$.
	Due to the $\UT_H$-coercivity of $a(\cdot, \cdot)$ over $W(\UP^m(K))$, we have the following standard quasi-optimality 
	\begin{equation}\label{eq:correctorgalerkin}
	\|(\CQ_K-\CQ_{K,m})v_H\|_{1, \Omega}\lesssim \inf_{w_{K,m}\in W(\UP^m(K))} \|\CQ_K v_H-w_{K,m}\|_{1, \Omega}.
	\end{equation}
	We choose now  $w_{K,m}:=(\operatorname{id}-I_H)(\eta\CQ_Kv_H)$ with a piecewise linear, globally continuous cut-off function $\eta$ defined via
	\[\eta=0\quad \text{in}\quad \Omega\setminus \UP^m(K), \qquad \eta=1\quad \text{in}\quad \UP^{m-2}(K).\]
	Inserting this choice of $w_{K,m}$ into \eqref{eq:correctorgalerkin} and noting that $I_H(\CQ_Kv_H)=0$, we obtain
	\[\|(\CQ_K-\CQ_{K,m})v_H\|_{1, \Omega}\lesssim\|(\operatorname{id}-I_H)((1-\eta)\CQ_K v_H)\|_{1, \Omega}\lesssim \|\CQ_K v_H\|_{1, \Omega\setminus \UP^m(K)},\]
	where the last inequality follows from the properties \eqref{eq:intpolcutoff} and \eqref{eq:gradcutoff} similar to the arguments in the proof of Proposition \ref{prop:expdecay}.
	Combination with Proposition \ref{prop:expdecay} gives \eqref{eq:truncerrorlocal}.
	
	To prove Theorem \ref{thm:localization}, we define, for a given simplex $K\in \CT_H$ and a given number of layers $m$, the piecewise linear, globally continuous cut-off function $\eta_K$ via
	\[\eta_K=0\quad \text{in}\quad  \UP^{m+1}(K), \qquad \eta_K=1\quad \text{in}\quad \Omega\setminus\UP^{m+2}(K).\]
	For a given $v_H\in V_H$, denote $w:=(\CQ-\CQ_m)v_H=\sum_{K\in \CT_H}w_K$ with $w_K:=(\CQ_K-\CQ_{K,m})v_H$.
	By the $\UT_H$-coercivity of $a(\cdot, \cdot)$ over $W$, we have
	\[\alpha_\kappa \|w\|^2_{1,\Omega} \lesssim \alpha_\kappa|w|^2_{1,\sigma,\Omega}\leq \sum_{K\in \CT_H}a(w_K, \UT_H w)\leq \sum_{K\in \CT_H} (A_K+B_{K,1}+B_{K,2}),\]
	where, for any $K\in \CT_H$, we abbreviate
	\[A_K\!:=|a(w_K, (1-\eta_K)\UT_H w)|, \quad\! B_{K,1}\!:=|a(w_K, (\operatorname{id}-I_H)(\eta_K \UT_H w)|, \quad\! B_{K,2}\!:=|a(w_K, I_H(\eta_K\UT_H w))|.\]
	Because $(\operatorname{id}-I_H)(\eta \UT_H w)\in W$ with support outside $\UP^m(K)$, we have $B_{K,1}=0$.
	Using the property \eqref{eq:gradcutoff}, the stability of $I_H$ \eqref{eq:IHstabapprox} and $\|\UT_H w\|_{\UN(\{\eta\neq 1\})}\lesssim \|w\|_{\UN^2(\{\eta\neq 1\})}$, we deduce
	\begin{align*}
	A_K\lesssim \|w_K\|_{1, \Omega}\,\|w\|_{1,\UN^{1}(\{\eta\neq 1\})},\qquad B_{K,2}\lesssim \|w_K\|_{1, \Omega}\,\|w\|_{1,\UN^{2}(\{\eta\neq 1\})}.
	\end{align*}
	Combining these estimates and observing that $\{\eta\neq 1\}=\UP^{m+2}(K)$, we obtain
	\[\alpha_\kappa\|w\|^2_{1, \Omega}\lesssim \sum_{K\in \CT_H}\|w_K\|_{1, \Omega}\|w\|_{1, \UP^{m+4}(K)}\lesssim C_{\operatorname{ol}, m}^{1/2}\|w\|_{1, \Omega}\Bigl(\sum_{K\in \CT_H}\|w_K\|^2_{1, \Omega}\Bigr)^{1/2},\]
	which in combination with \eqref{eq:truncerrorlocal} finishes the proof.
\end{proof}

\section{Numerical experiments}
\label{sec:numexperiment}

The numerical examples were carried out in MATLAB based upon preliminary code developed at the Chair for Computational Mathematics at University of Augsburg.
We always consider $\Omega=[0,1]^2$.
	Our meshes are constructed out of blocks as depicted in Figure \ref{fig:mesh}: A mesh size of $H=2^{-N}$ with $N\geq 1$ means that the mesh consists of $N\times N$ blocks of Figure \ref{fig:mesh}.
\begin{figure}
	\centering
	\includegraphics[width=0.3\textwidth, trim=42mm 100mm 42mm 100mm, clip=true]{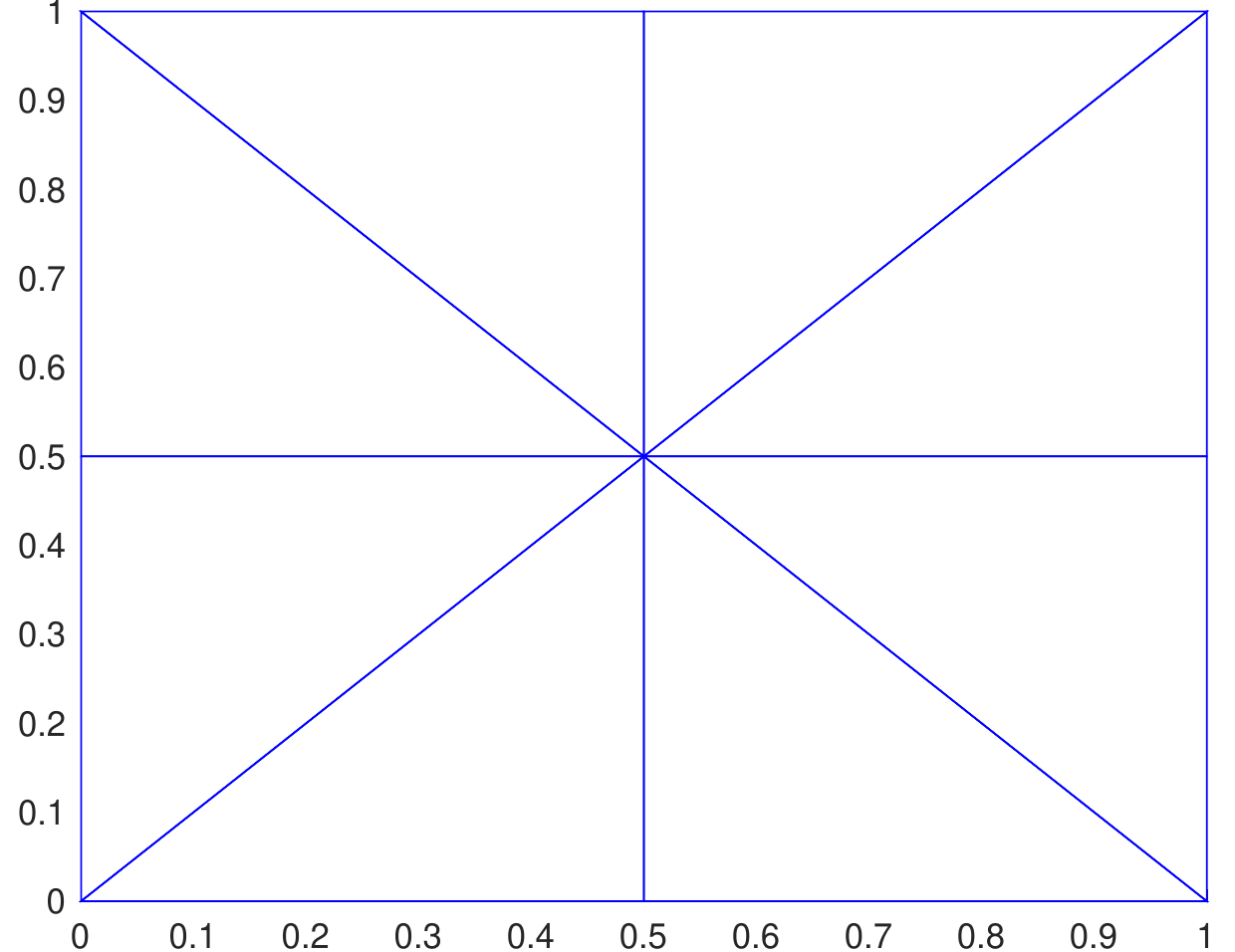}
	\caption{Building block for the meshes in the numerical experiments}
	\label{fig:mesh}
\end{figure}
The fine mesh has $h=2^{-8}$ and is $\UT$-conform in all settings described below except from the circular inclusion in Section \ref{subsec:circular}.
This fine mesh is used for the corrector computations and, additionally, for the computation of a reference solution $u_h$ using standard FEM in Sections \ref{subsec:square} and \ref{subsec:multiscale}.
The LOD solution is computed on a series of meshes with $H=2^{-1}, \ldots, 2^{-6}$ and oversampling parameters $m\in \{1,2,3\}$. We refer to $(\operatorname{id}-\CQ_m)u_{H,m}$ from \eqref{eq:lod} as the LOD solution and to $u_{H,m}$ as the macroscopic part of the LOD solution. Note that $u_{H,m}$ lies in the standard FE space.
For comparison, we also compute the standard FE solution on the coarse grids $\CT_H$ as well as the $L^2(\Omega)$-projection of the exact or reference solution onto $V_H$. The latter is referred to as the $L^2$-best approximation in $V_H$.
We compute the absolute error of the LOD solution in the $H^1(\Omega)$-semi-norm and compare it to the absolute error of the standard FEM.
From \eqref{eq:errorLOD}, we expect linear convergence of this LOD error.
Moreover, we also consider the absolute error of the macroscopic part of the LOD solution in the $L^2(\Omega)$-norm and compare it to the absolute errors of the FEM solution and the $L^2$-best approximation in $V_H$.
We expect that the macroscopic error of the LOD behaves like the $L^2$-best approximation error (cf.\ \eqref{eq:errormacroLOD}).

Finally, we note that, although our theory guarantees well-posedness of the corrector problems only if the contrast is outside a sufficiently large interval, which is larger than the analytical one,  we never experienced any well-posedness issues in practice.

\subsection{Flat interface with known exact solution}
\label{subsec:flat}

We define $\Op=\{x \in \Omega \; | \; x_2<0.5-2^{-7}\}$ and $\Om$ accordingly as
$\Om=\{x \in \Omega \; | \; x_2>0.5-2^{-7}\}$. We set $\Sp = 1$ and consider two
different cases where $\Sm = 2$ or $1.1$.
We shifted the interface $\Gamma$ from the middle line in order to have meshes $\CT_H$ that do not resolve the interface and that are not symmetric
for any $H$. Hence, we expect a poor performance of the standard FEM. 
In this case, $C_\pm(T)$ can be analytically computed: We obtain $C_\pm(T)=2\sqrt{\frac{0.5+2^{-7}}{0.5-2^{-7}}}$, such that $a(\cdot, \cdot)$ is $\UT$-coercive if $\frac{\sigma_-}{\sigma_+}>\frac{0.5+2^{-7}}{0.5-2^{-7}}\approx 1.0317$, see also \cite{CC13Tcoercivemesh}.
Hence, the model problem is well-posed for our choices of $\Sm$, but note that the condition for $\UT_H$-coercivity is most probably violated.

We consider the
following piecewise smooth function fulfilling homogeneous Dirichlet boundary conditions
\[u(x_1,x_2)=\begin{cases}
-\sigma_- x_1 (x_1-1) x_2 (x_2-1) (x_2-l),\qquad &\qquad(x_1,x_2)\in \Op,\\
x_1(x_1-1) x_2(x_2-1) (x_2-l),\qquad &\qquad(x_1,x_2)\in \Om,
\end{cases}\]
where $l=0.5-2^{-7}$ stands for the interface location.
The right-hand side $f$ is computed so that $u$ is the exact solution.
Precisely, $f(x_1,x_2)=\sigma_-(2x_2(x_2-1)(x_2-l)+x_1(x_1-1)(6x_2-2(l+1))$ and we note that $f$ is globally smooth.

\begin{figure}
	\includegraphics[width=0.47\textwidth, trim=40mm 95mm 42mm 94mm, clip=true, keepaspectratio=false]{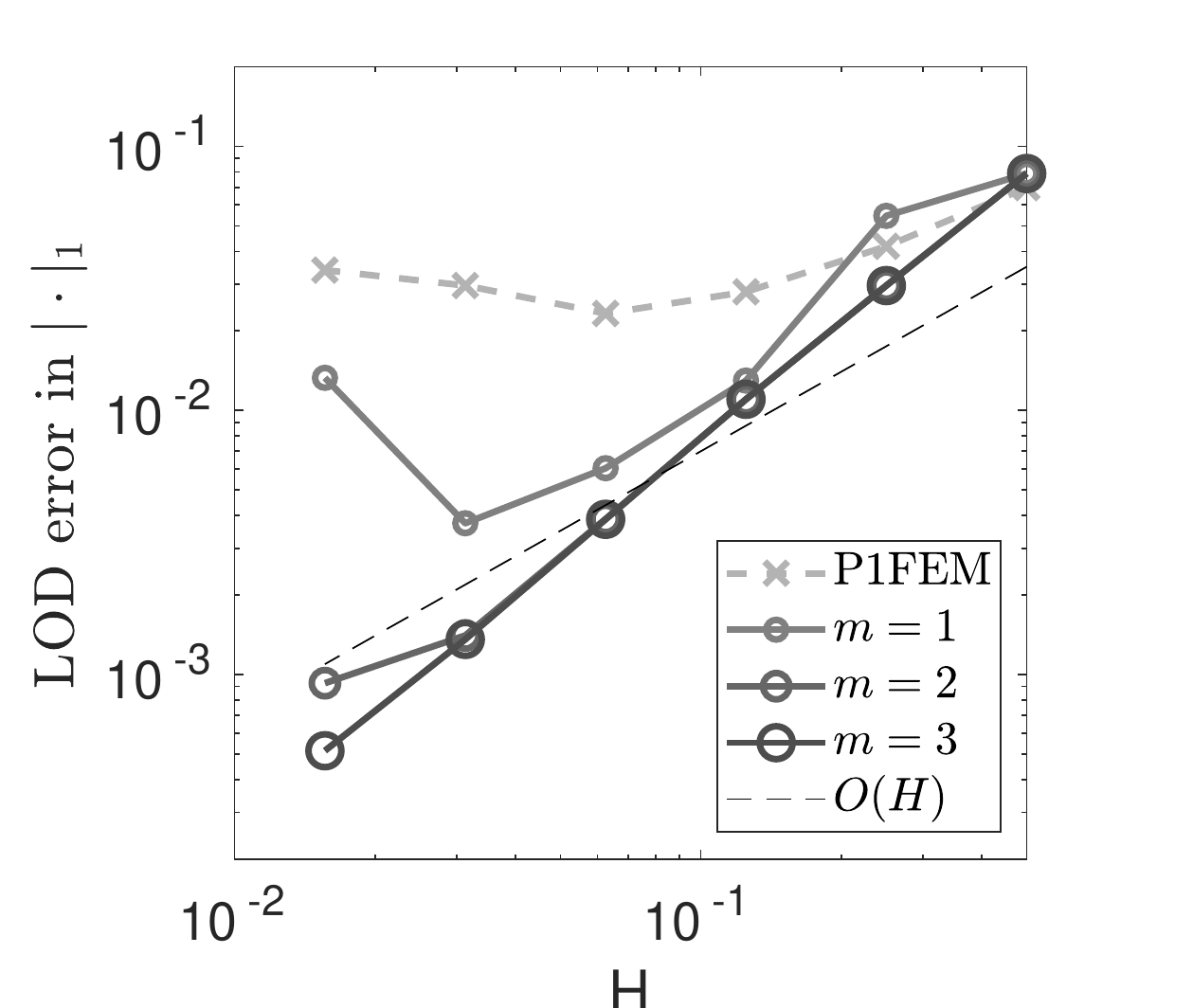}%
	\hspace{5ex}%
	\includegraphics[width=0.47\textwidth, trim=40mm 95mm 42mm 94mm, clip=true, keepaspectratio=false]{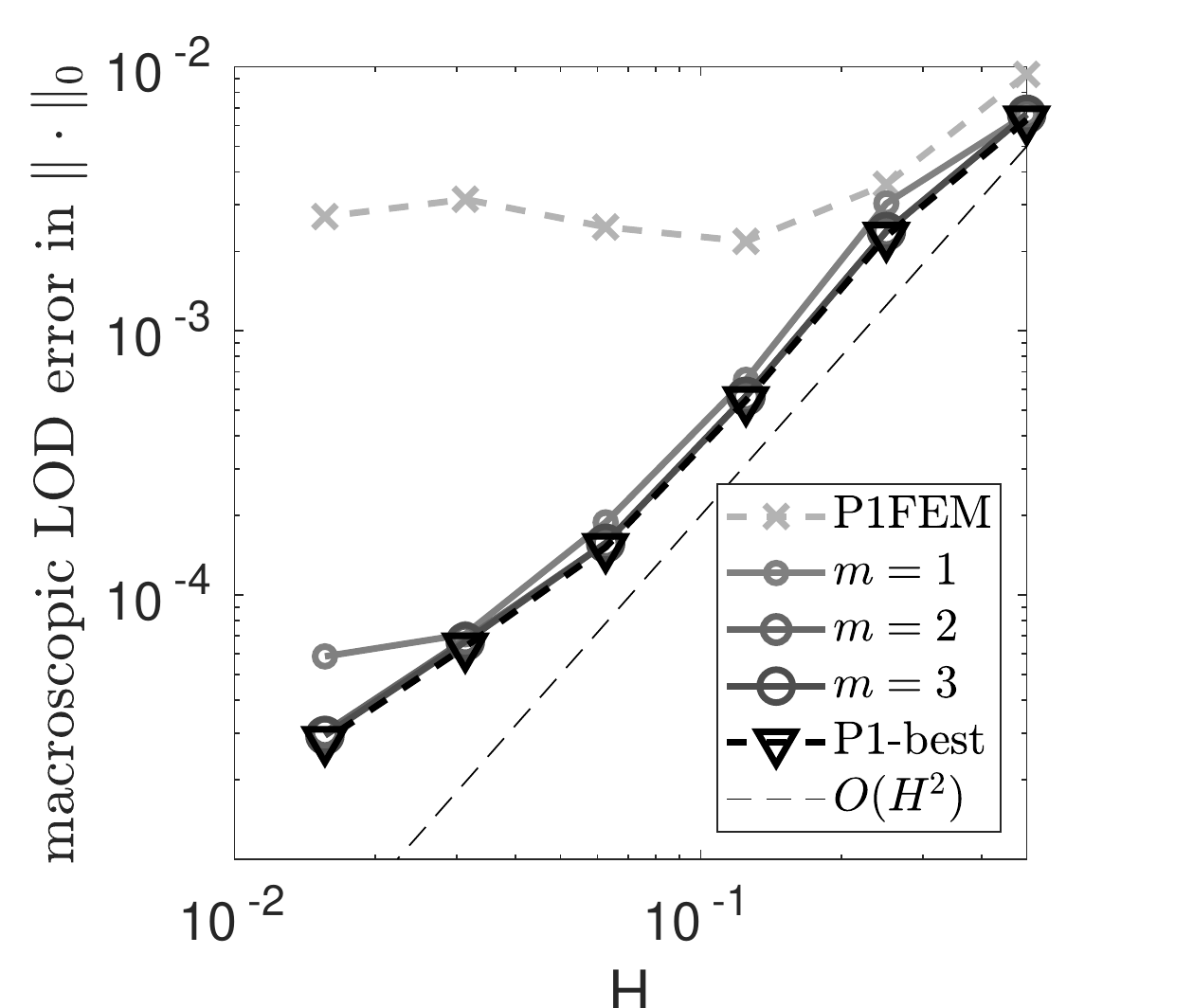}\\   
	\includegraphics[width=0.47\textwidth, trim=40mm 95mm 42mm 94mm, clip=true, keepaspectratio=false]{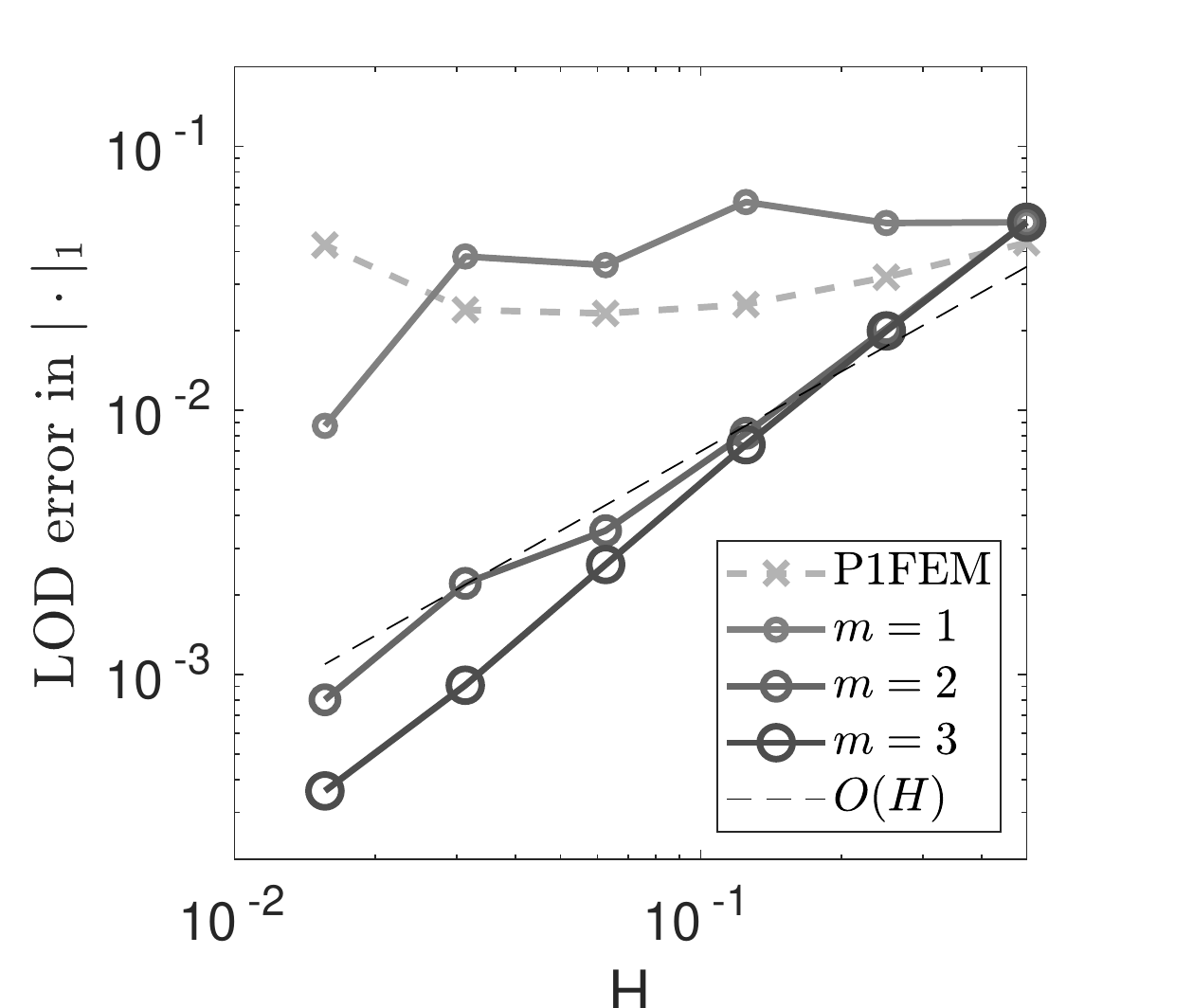}%
	\hspace{5ex}%
	\includegraphics[width=0.47\textwidth, trim=40mm 95mm 42mm 94mm, clip=true, keepaspectratio=false]{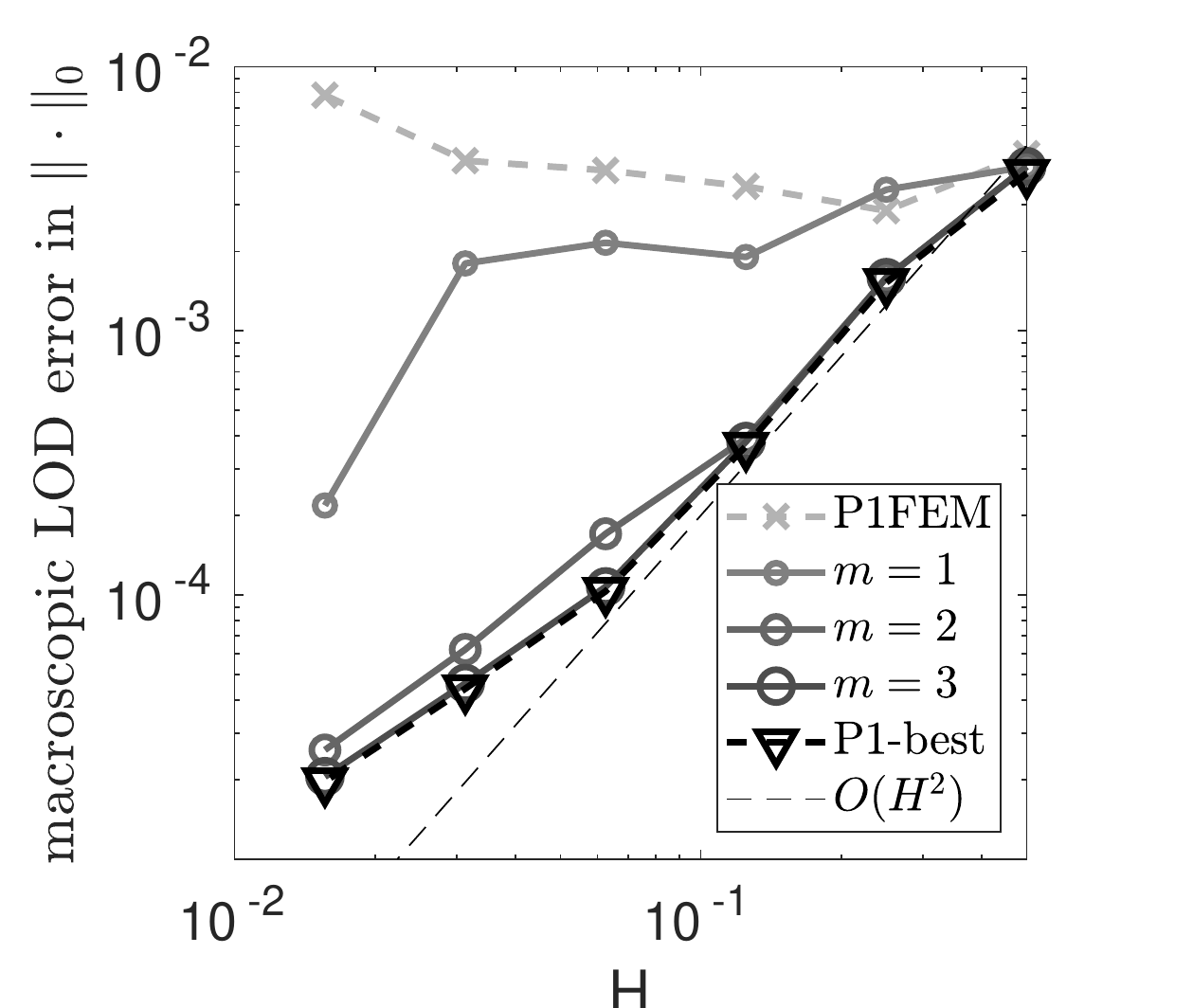}   
	\caption{Convergence histories for the flat interface with $\sigma_-=2$ (top) and $\sigma_-=1.1$ (bottom) in Section \ref{subsec:flat}.}
	\label{fig:flat-convhist}
\end{figure}

The LOD error (in the $H^1(\Omega)$-semi-norm) and the macroscopic LOD error (in the $L^2(\Omega)$-norm) for both choices of $\sigma_-$ are depicted in Figure \ref{fig:flat-convhist}.
We observe that an oversampling parameter $m=3$ is sufficient to produce faithful LOD approximations.
The LOD error in both cases converges linearly as expected and the macroscopic LOD error follows the $L^2$-best approximation.
Note that the latter converges quadratically due to the piecewise smoothness of $u$.
This nicely illustrates the findings of Theorem \ref{thm:lod}.
In contrast to the good performance of the LOD, we see the failure of the standard FEM in Figure \ref{fig:flat-convhist}.
This is of course expected from the fact that $\CT_H$ does not resolve the interface.
Moreover, we observe that for $\sigma_-=1.1$ we should select $m=3$ as oversampling parameter in the LOD, whereas for $\sigma_-=2$, $m=2$ already yields good results, see Figure \ref{fig:flat-convhist} top and bottom left.
This effect is connected to the $\tilde\alpha_\kappa$-dependency of the exponential decay: Since $\sigma_-=1.1$ is close to the critical interval, this constant in the $\UT_H$-coercivity is small so that the decay of the corrector is slow, which results in a larger oversampling region.

\begin{figure}
	\includegraphics[width=0.47\textwidth, trim=40mm 95mm 39mm 94mm, clip=true, keepaspectratio=false]{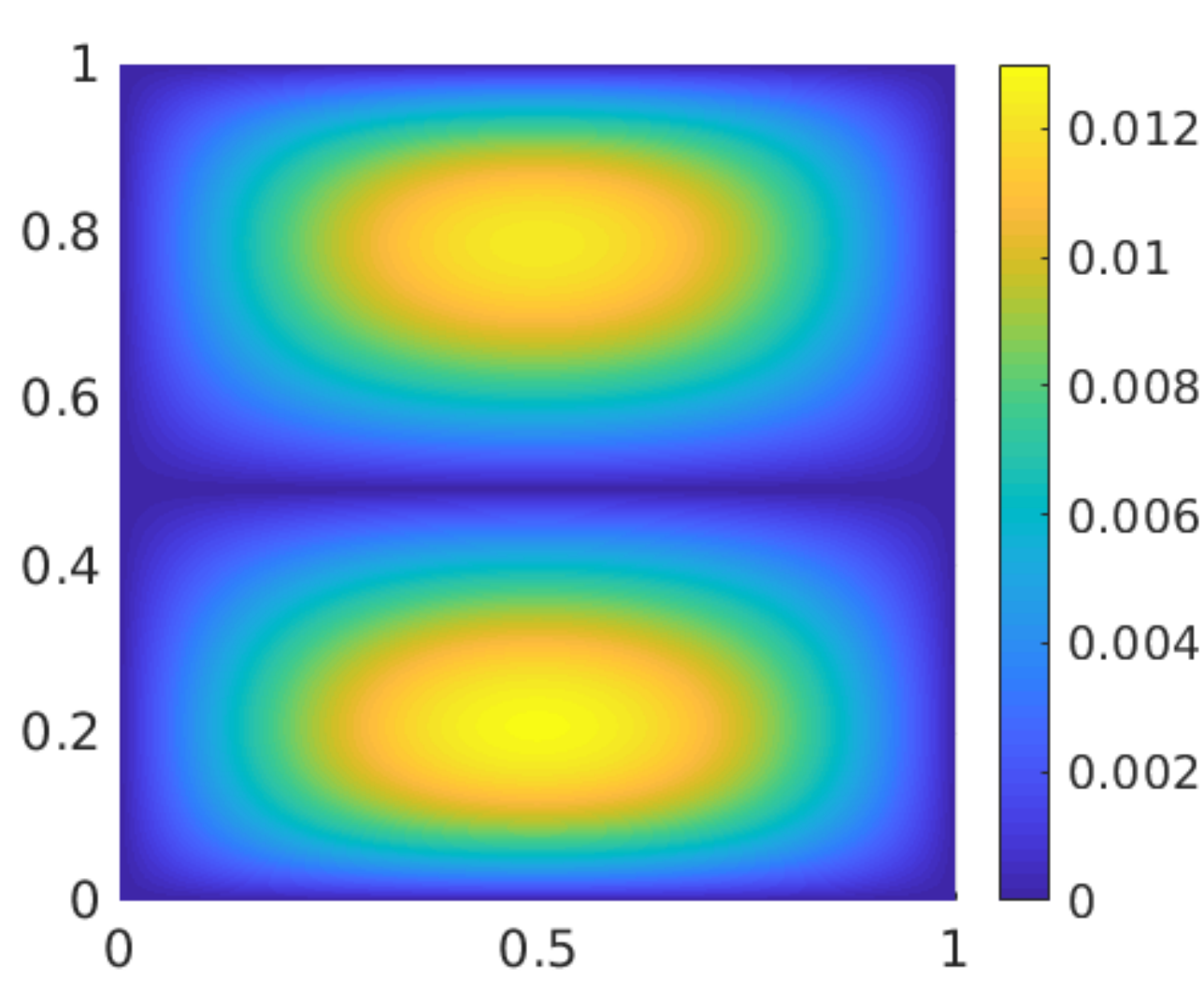}%
	\hspace{5ex}%
	\includegraphics[width=0.47\textwidth, trim=40mm 95mm 39mm 94mm, clip=true, keepaspectratio=false]{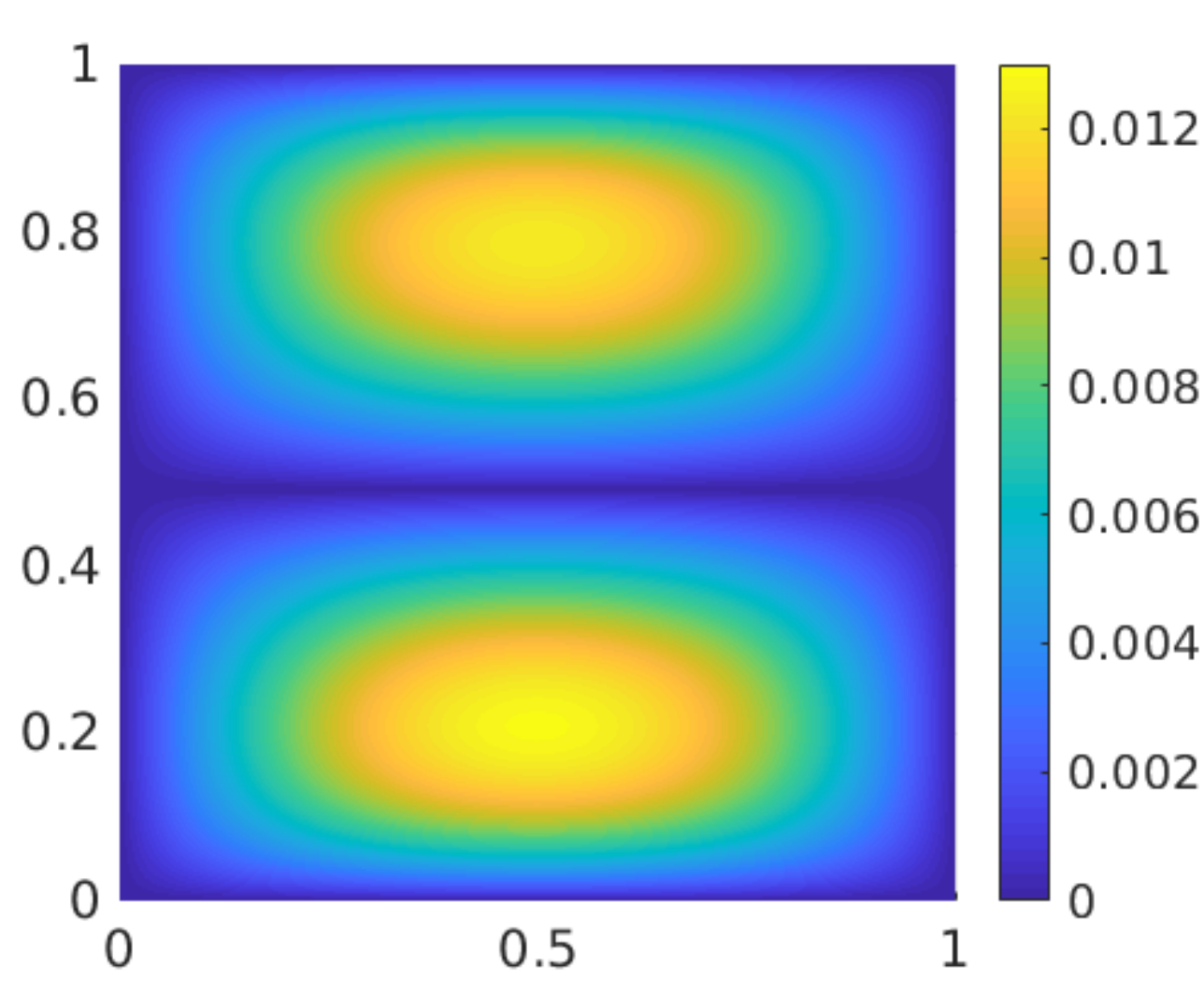}\\   
	\includegraphics[width=0.47\textwidth, trim=40mm 95mm 39mm 94mm, clip=true, keepaspectratio=false]{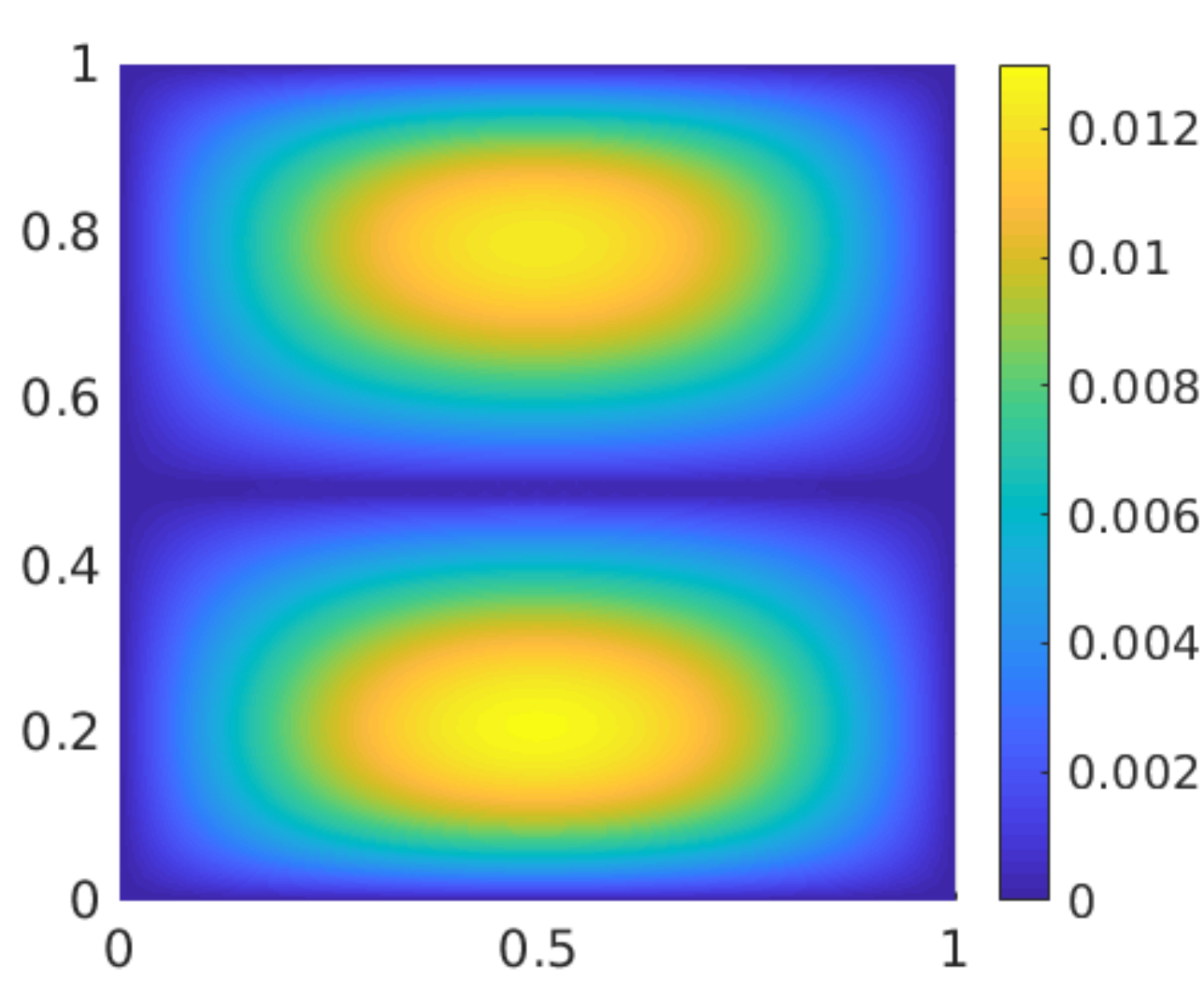}%
	\hspace{5ex}%
	\includegraphics[width=0.47\textwidth, trim=40mm 95mm 39mm 94mm, clip=true, keepaspectratio=false]{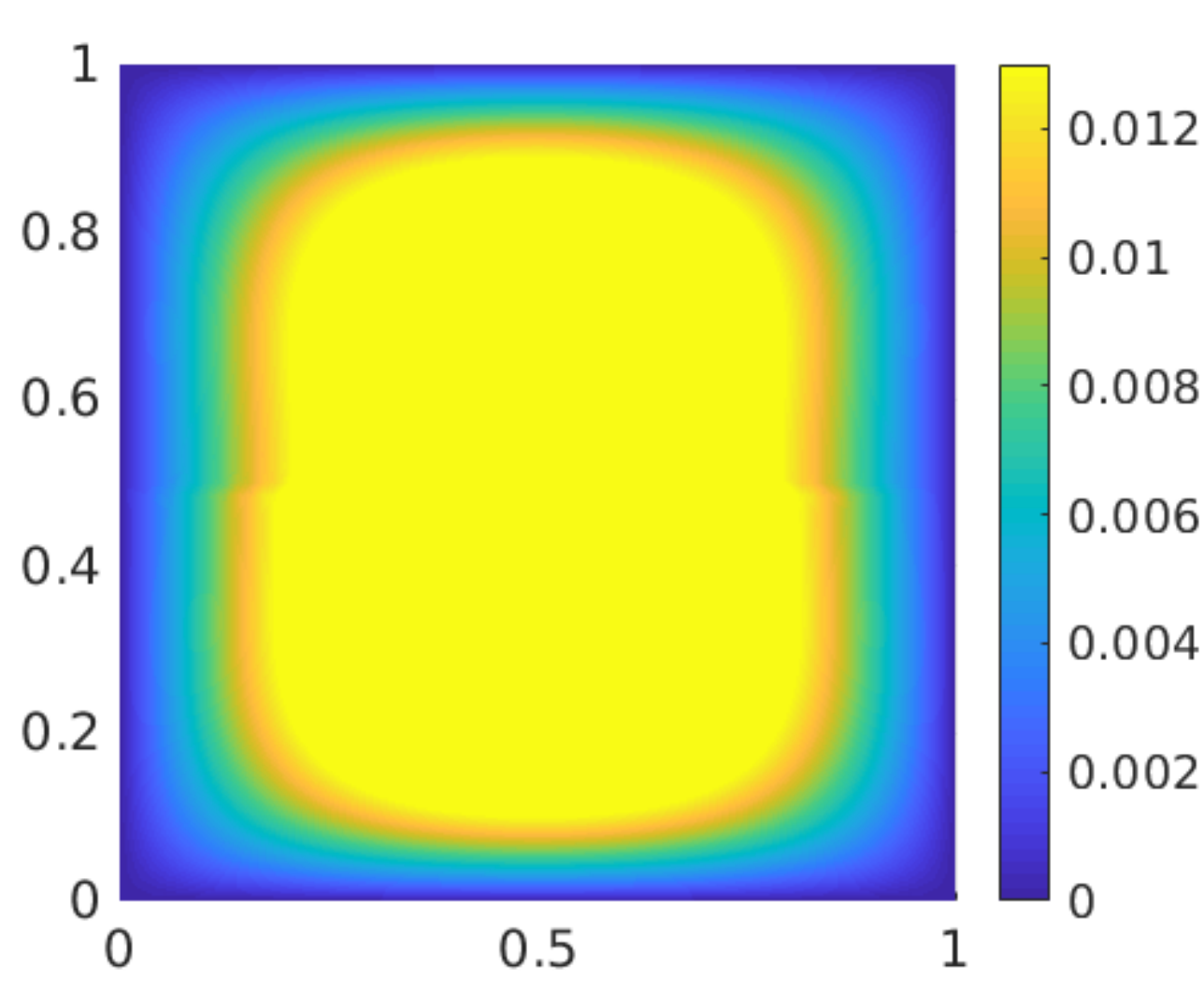}   
	\caption{Various solutions for the flat interface with $\sigma_-=1.1$ (Section \ref{subsec:flat}): exact solution $u$ (top left), LOD solution (top right), macroscopic part of the LOD solution ($u_{H,m}$, bottom left) and FE solution (bottom right).}
	\label{fig:flat-sols}
\end{figure}

We now compare for $H=2^{-6}$ and $m=3$ the LOD solution, its macroscopic part, and the FE solution to the exact solution in the case $\sigma_-=1.1$, see Figure \ref{fig:flat-sols}.
Strikingly, the FE solution has almost no resemblance with the exact solution, but the macroscopic part of the LOD (which lies in the same space $V_H$) is very close to the exact solution.
For this example, one can hardly make out any differences between the exact solution, the LOD solution and its macroscopic part, which clearly underlines the potential of our method.
In particular, we emphasize once more that good approximations (in an $L^2(\Omega)$-sense) exist in the coarse FE space $V_H$, which are found by our approach but not by the standard finite element method.
Though expected, it is interesting to see how drastically the (slight) unfit of the meshes to interface influences the performance of the standard FEM.

\begin{figure}
	\includegraphics[width=0.3\textwidth, trim=40mm 94mm 35mm 90mm, clip=true, keepaspectratio=false]{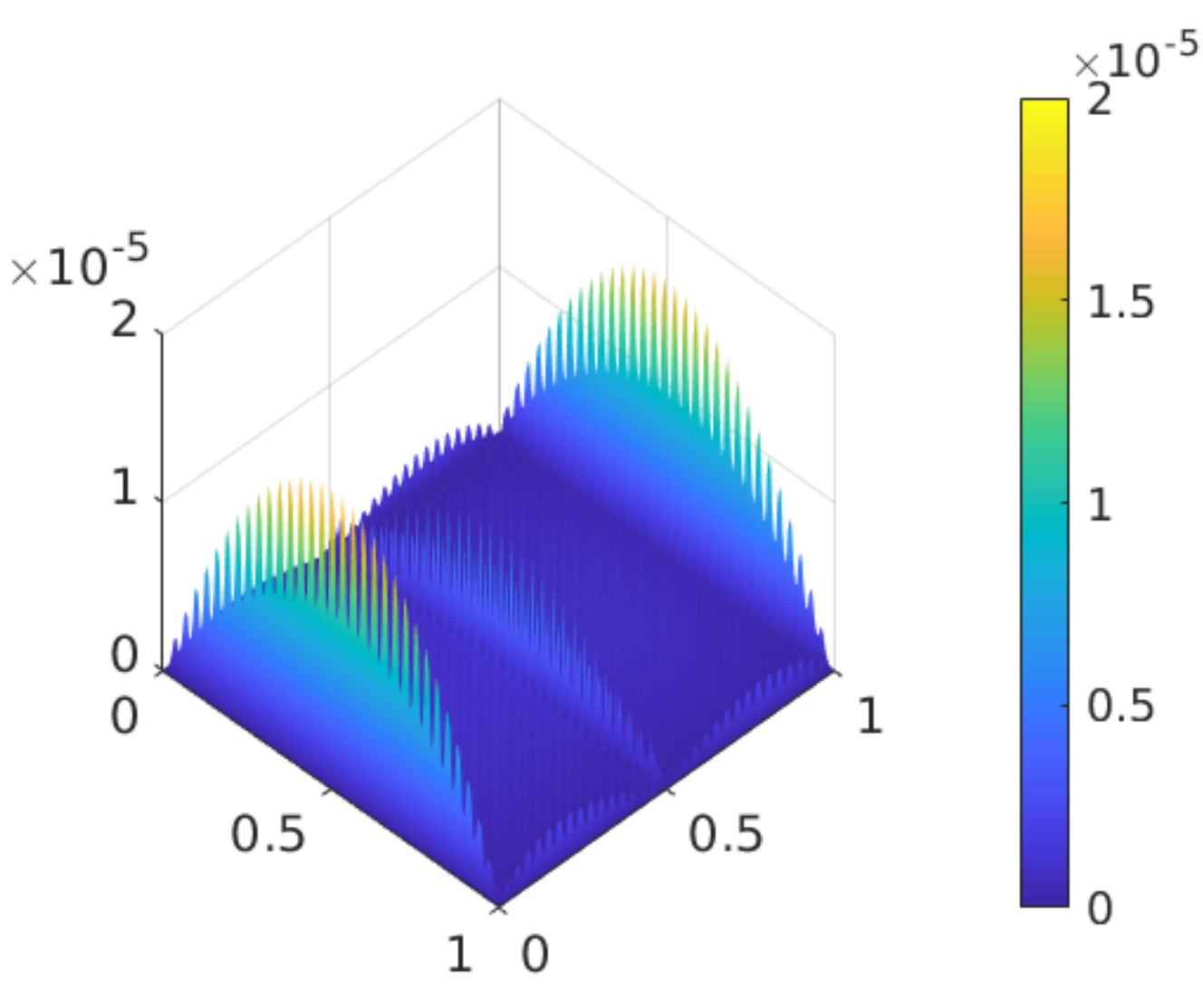}%
	\hspace{2ex}%
	\includegraphics[width=0.3\textwidth, trim=40mm 95mm 35mm 90mm, clip=true, keepaspectratio=false]{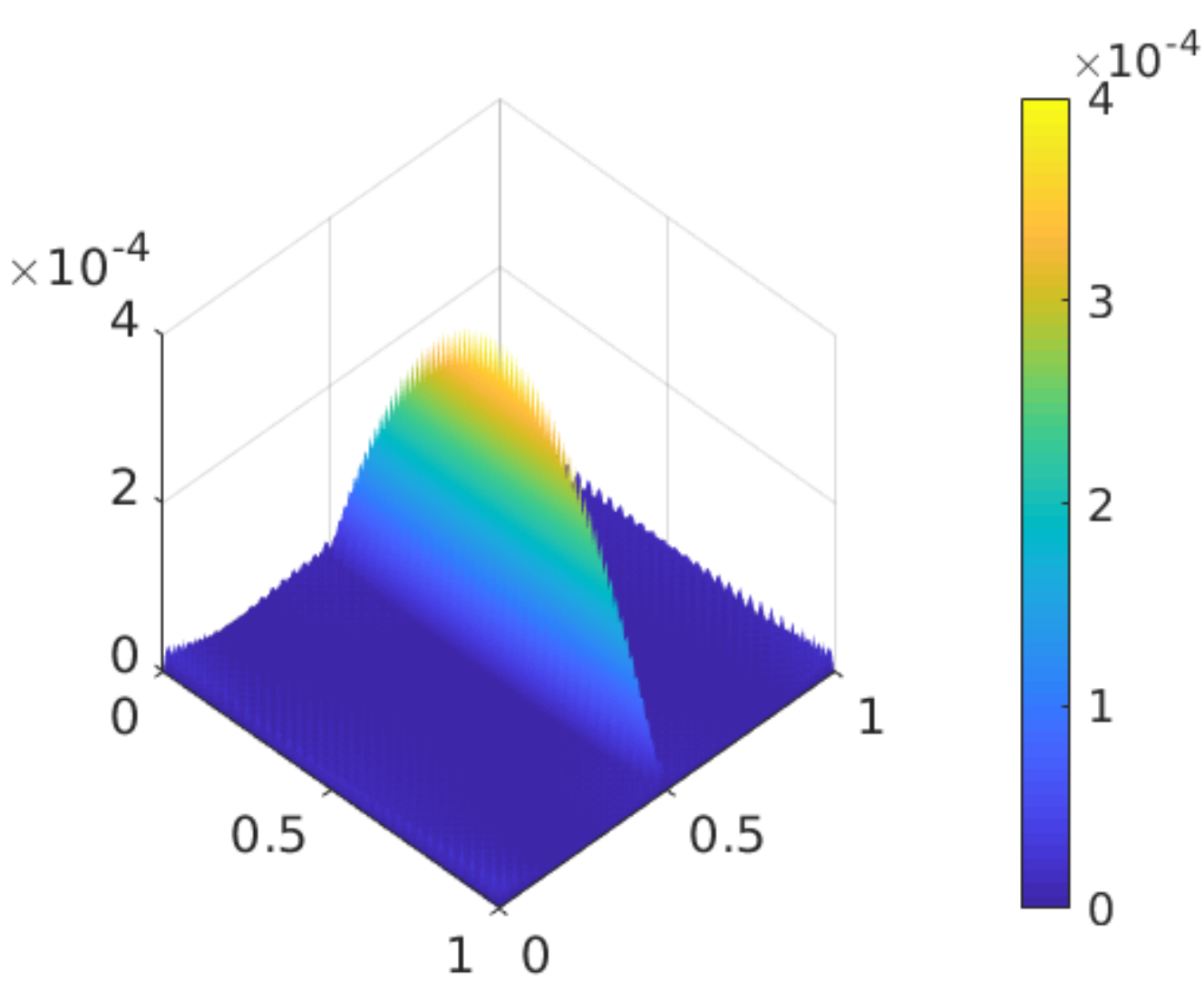}%
	\hspace{2ex}
	\includegraphics[width=0.3\textwidth, trim=40mm 95mm 33mm 90mm, clip=true, keepaspectratio=false]{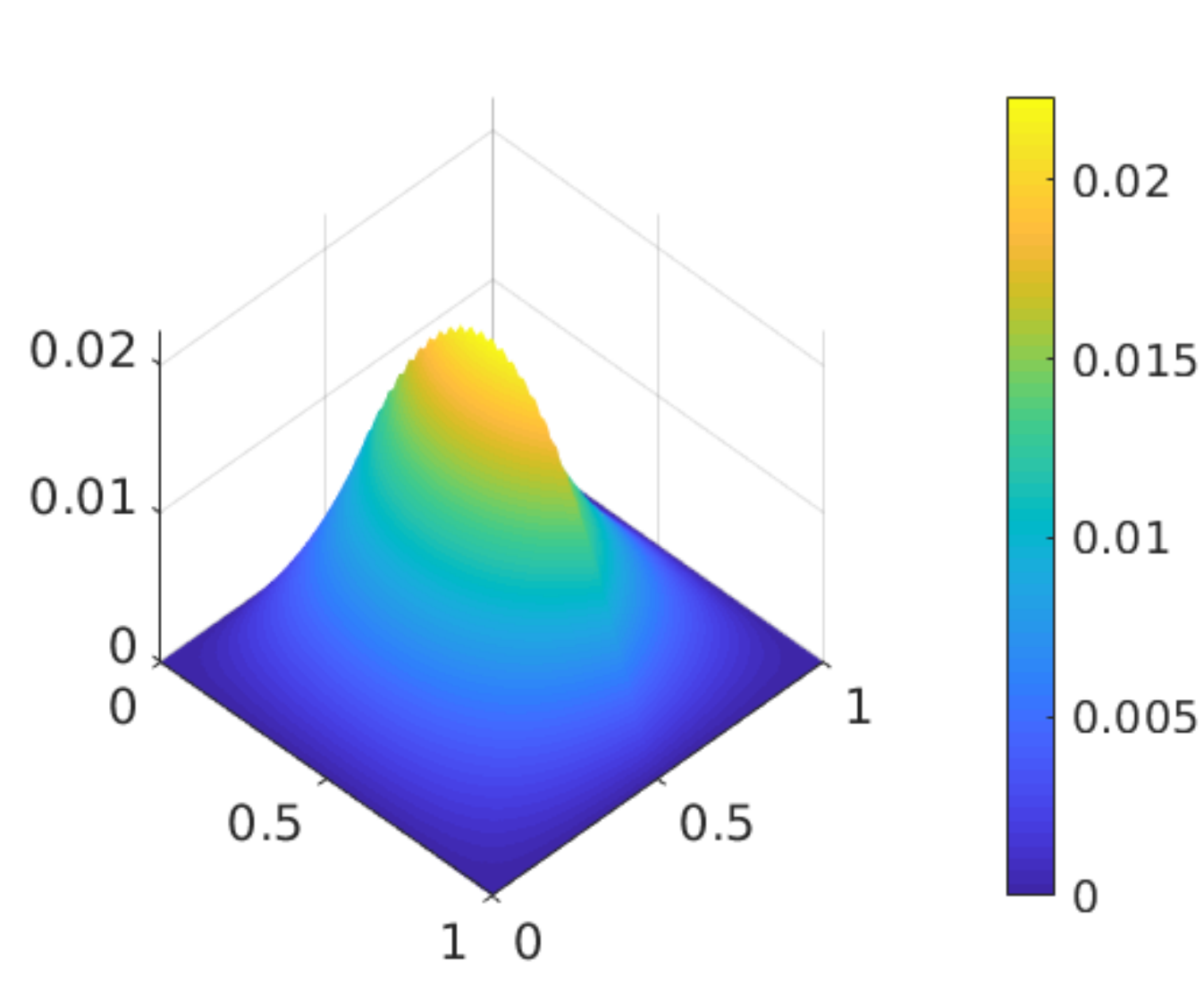}   
	\caption{Errors of the different solutions to $u$ for the interface with $\sigma_-=1.1$ (Section \ref{subsec:flat}): error of LOD solution (left), error of macroscopic part of LOD solution (middle), and error of FE solution (right).}
	\label{fig:flat-errors}
\end{figure}

To see more details, we visualize the absolute errors of the three solutions (i.e., $(\operatorname{id}-\CQ_m)u_{H,m}$, $u_{H,m}$ and $u_H$) to $u$ in Figure \ref{fig:flat-errors}.
Here, we clearly see a difference in the error distribution. The FE error (right) is very large close to the interface and this error spreads out over a large part of the domain.
In contrast, the macroscopic part of the LOD solution (middle) has a much smaller error which is furthermore very confined to the interface.
A localization of the error close to the interface is expected because on the one hand, this jump in the coefficient is not resolved by the mesh and because on the other hand, the interesting effects happen there.
In the full LOD solution (left), the error at the interface is largely reduced by the upscaling procedure so that interface and boundary errors are now of the same order.

\subsection{Square inclusion}\label{subsec:square}
We consider $\Om=[0.25, 0.75]^2$ and $\Op$ as the complement. 
The coefficient $\sigma$ is spatially varying, more precisely $\sigma|_{\Op}(x)=0.75+0.125\cos(2\pi\frac{x_1}{\varepsilon})+0.125\sin(2\pi\frac{x_2}{\varepsilon})$ and $\sigma|_\Om(x)=-5+0.5\sin(2\pi\frac{x_1}{\varepsilon})+0.5\cos(2\pi\frac{x_2}{\varepsilon})$ with $\varepsilon=2^{-7}$.
According to \cite{BCC12signchangingTcoercive}, the model problem is $\UT$-coercive for this geometry and choice of $\sigma$.
We set $f=0.1\chi_{\{x_2<0.1\}}+\chi_{\{x_2>0.1\}}$ to have a right-hand side only in $L^2(\Omega)$ and, at the same time, to keep the sign-change in $\sigma$ and the jump in $f$ apart from each other.
\begin{figure}
	\includegraphics[width=0.47\textwidth, trim=25mm 80mm 27mm 80mm, clip=true, keepaspectratio=false]{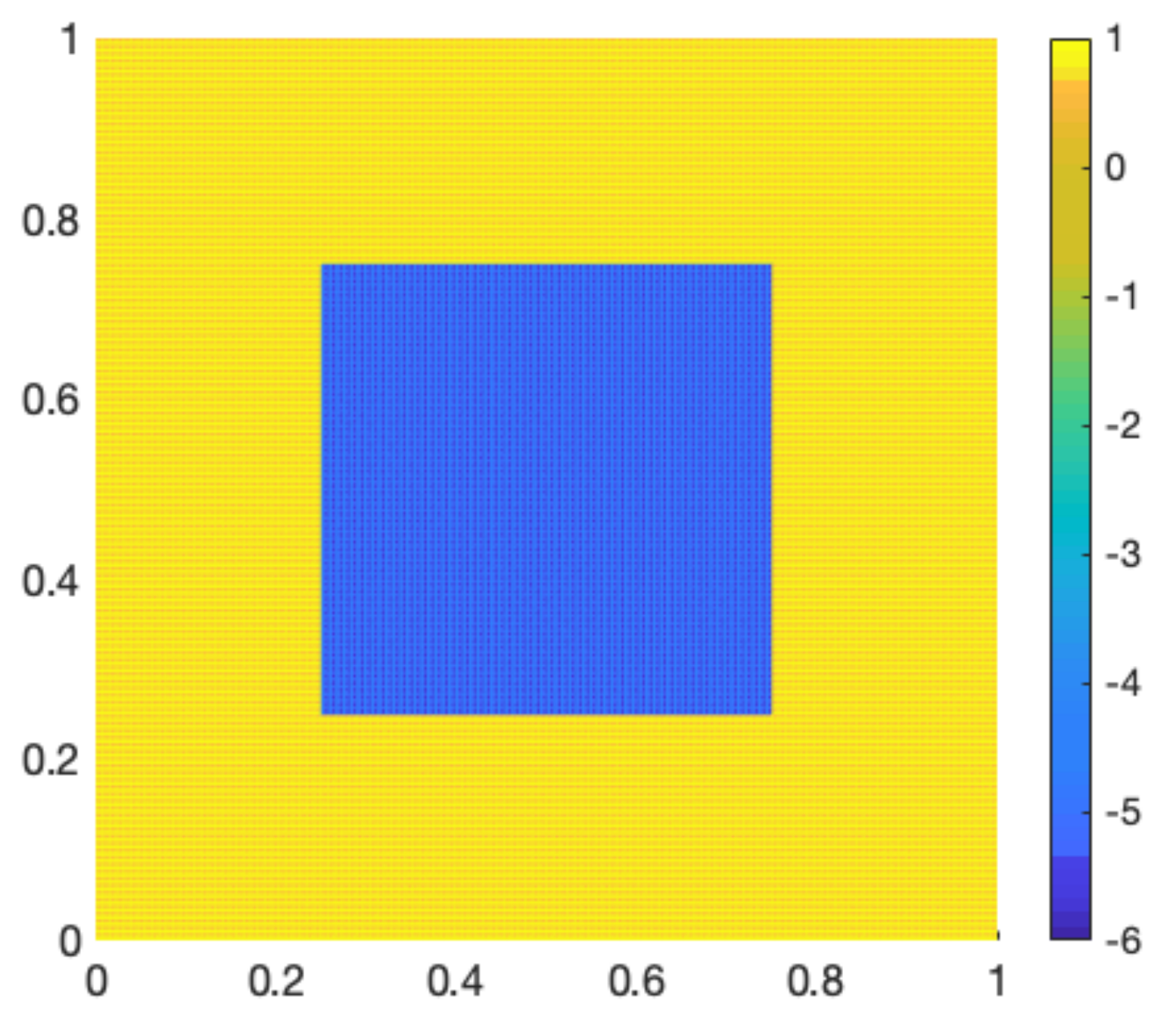}%
	\hspace{3ex}%
	\includegraphics[width=0.47\textwidth, trim=25mm 80mm 23mm 78mm, clip=true, keepaspectratio=false]{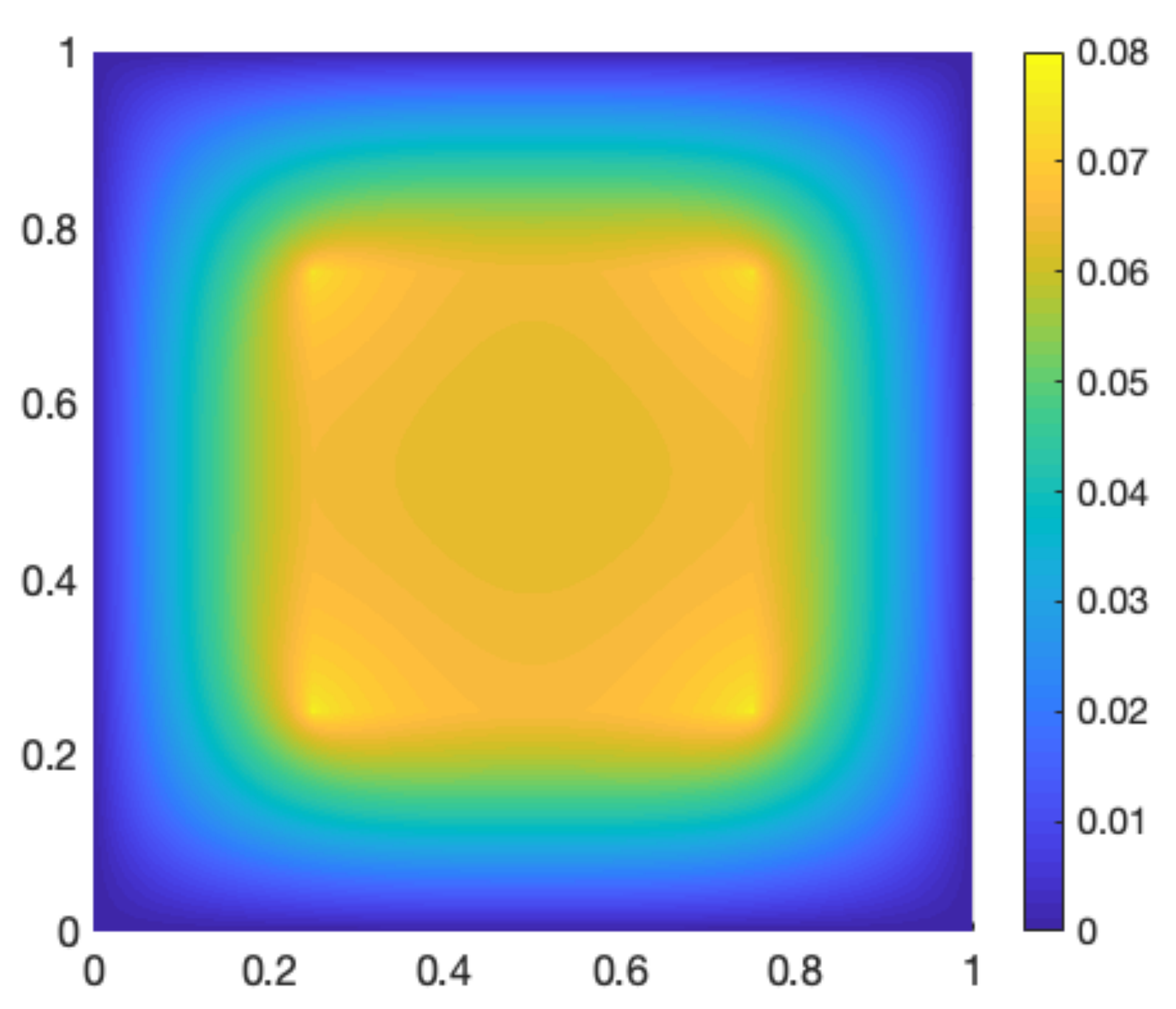}
	\caption{Coefficient (left) and fine FE solution (right) for the experiment in Section \ref{subsec:square}.}
	\label{fig:square-coeffsol}
\end{figure}
The coefficient $\sigma$ and the reference solution $u_h$, computed by a standard FEM on the fine mesh $\CT_h$, are depicted in Figure \ref{fig:square-coeffsol}.
Note that the fine mesh resolves the oscillations of $\sigma$. All coarse meshes $\CT_H$ resolve the interface and are $\UT$-conform, but note that they do not resolve the multiscale variations of $\sigma$.

\begin{figure}
	\includegraphics[width=0.47\textwidth, trim=20mm 75mm 22mm 75mm, clip=true, keepaspectratio=false]{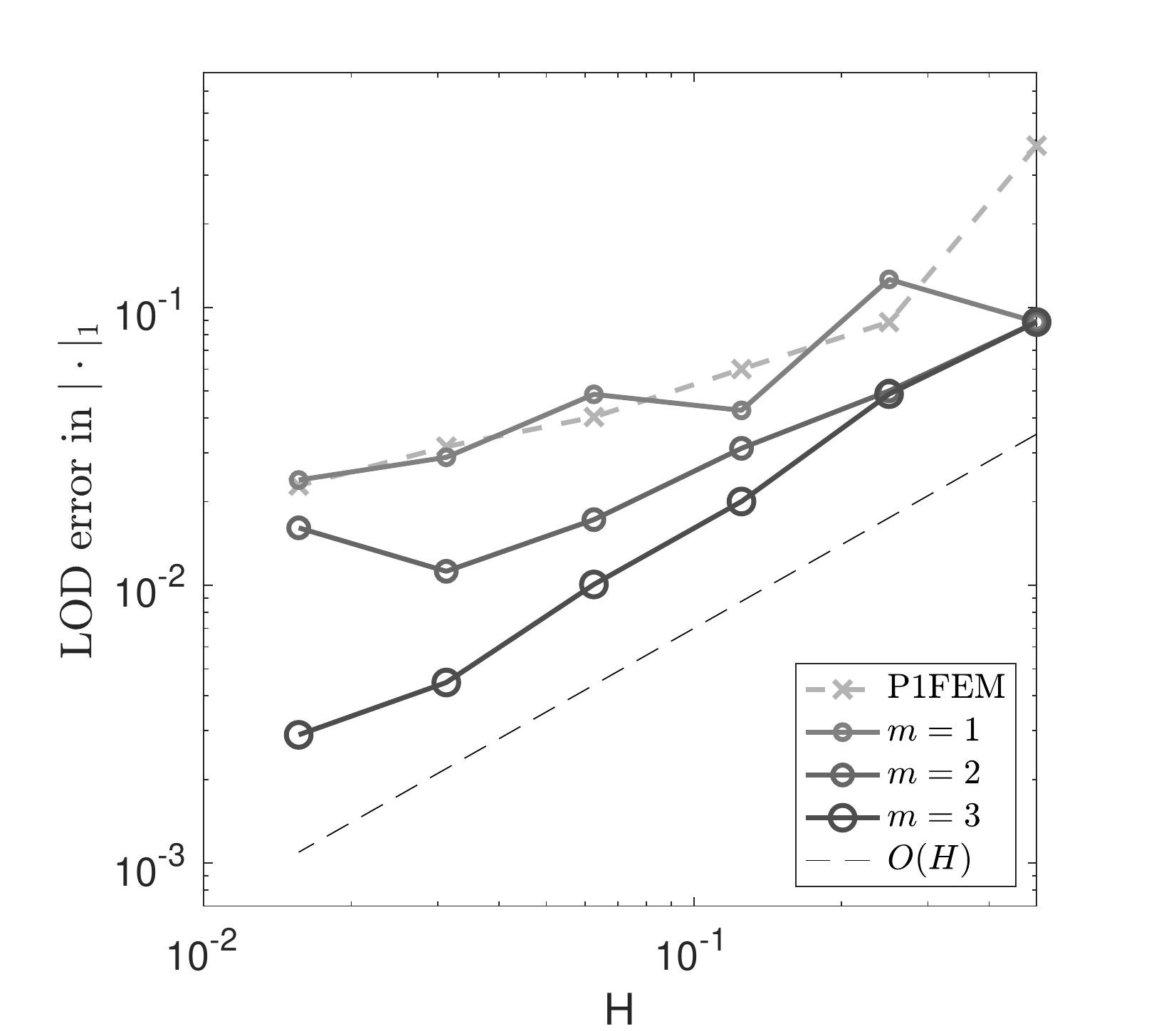}%
	\hspace{5ex}%
	\includegraphics[width=0.47\textwidth, trim=20mm 75mm 22mm 75mm, clip=true, keepaspectratio=false]{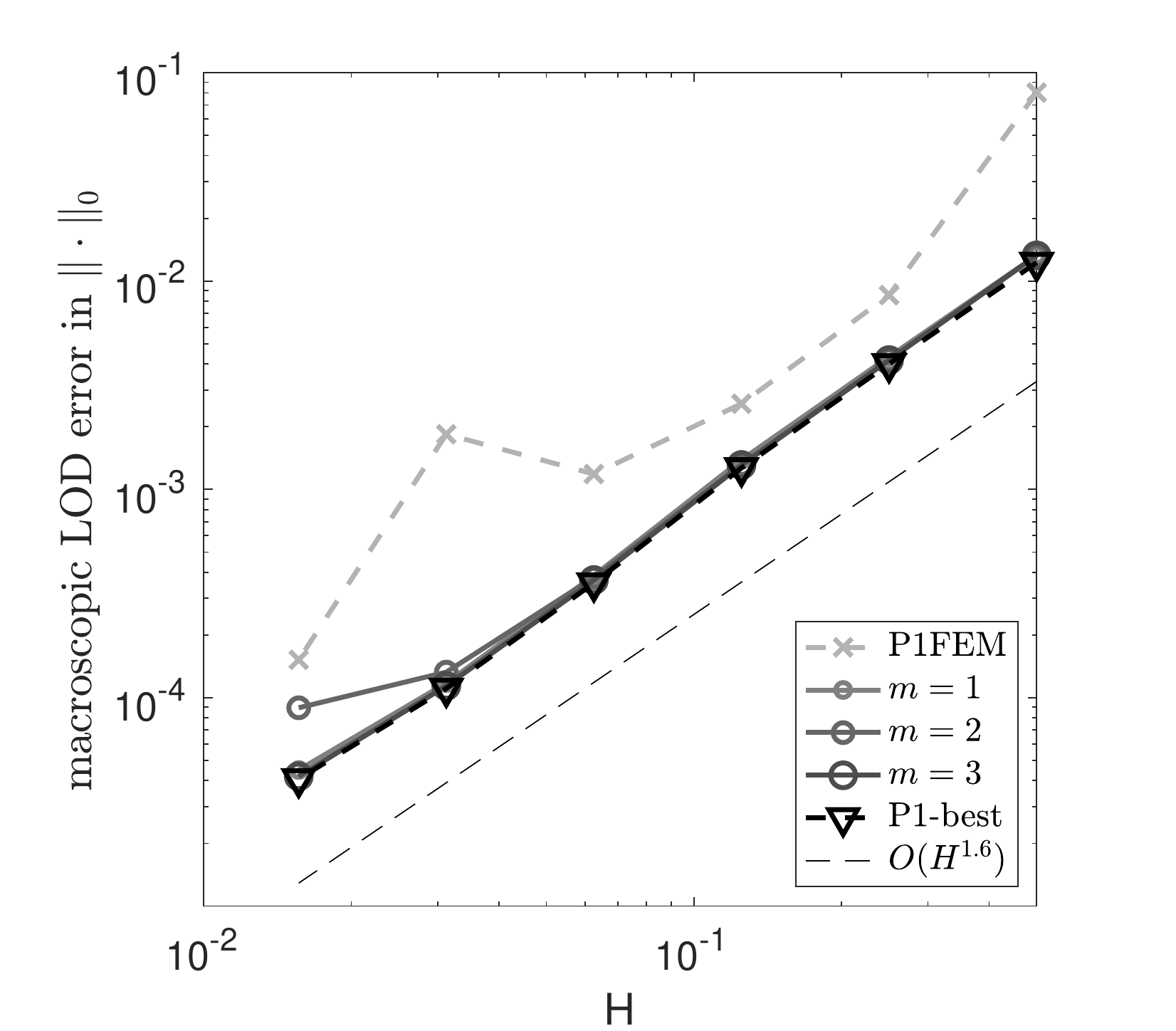}   
	\caption{Convergence histories for the square inclusion in Section \ref{subsec:square}.}
	\label{fig:square-convhist}
\end{figure} 

As in the previous section, we depict the convergence histories for the LOD error in the $H^1(\Omega)$-semi-norm and the macroscopic LOD error in the $L^2(\Omega)$-norm in Figure \ref{fig:square-convhist}.
We again observe the expected overall linear convergence of the LOD solution in the $H^1(\Omega)$-semi-norm. Moreover, the macroscopic LOD error follows the $L^2$-best approximation in the FE space, the best one can hope for.
The error for the standard finite element method is mostly decaying as well, but at a higher level in comparison to the LOD solution with $m=3$. Moreover, the rate of convergence is definitely lower and we even have a stagnation of the error in the $L^2$-norm at around $H=2^{-5}$, where the coefficient variations are not yet resolved.
Note that for this experiment, the $L^2$-best approximation no longer converges quadratically.
More precisely, both the macroscopic LOD error and the $L^2$-best approximation converge at an average approximate rate of $1.66$ as we calculated by taking the average of the experimental orders of convergence.
The regularity $u\in H^{1+\lambda}(\Omega)$ of the exact solution was studied for the present configuration with piecewise constant $\sigma$ in \cite{BDR99signchangingsing, NV11signchaningapost}. Inserting into these results the average values of $\sigma|_\Op$ and $\sigma|_\Om$, one obtains $\lambda\approx0.52$. The $L^2$-best approximation is thus converging slightly faster than the simple ad-hoc regularity calculation for constant coefficients predicts.

This experiment underlines the applicability and advantages of the method for oscillating coefficients.
Further, we emphasize that we have linear convergence of the LOD error in the $H^1$-semi-norm although the exact solution is definitely not in $H^2(\Omega)$ due to the corners at the interface.

\subsection{Circular inclusion with known exact solution}\label{subsec:circular}
We consider $\Om=B_{0.2}((0.5, 0.5))$, i.e., a circle with radius $0.2$ around the point $(0.5, 0.5)$, and $\Op$ the complement.
Since the boundary of $\Om$ is smooth, the critical interval consists only of the value $-1$. Hence, we choose $\sigma_+=1$ and $\sigma_-=2$ as in Section \ref{subsec:flat}.
We select a radially symmetric exact solution with homogeneous Dirichlet boundary conditions as follows. Let $(r,\varphi)$ denote the standard polar coordinates and set $\tilde r = r-0.5$. Then $u$ is given by
\[u(\tilde r)=
\begin{cases}
A\tilde r^2(\tilde r-0.2)(\tilde r-0.4)^2, \qquad &\qquad\tilde r<0.2,\\
-A\sigma_- \tilde r^2(\tilde r-0.2)(\tilde r-0.4)^2, \qquad &\qquad0.2<\tilde r<0.4,\\
0 \qquad &\qquad\text{else}
\end{cases}\]
and $f$ is calculated accordingly.
The scalar factor $A$ is used to scale the solution $u$ to an $L^\infty(\Omega)$-norm of order $1$, we pick here $A=10000$.
Note that the right-hand side $f$ is piecewise smooth and does \emph{not} possess a singularity at $(0.5, 0.5)$.
The exact solution $u$ is depicted in Figure \ref{fig:circular-conv}, left.

\begin{figure}
	\includegraphics[width=0.47\textwidth, trim=40mm 95mm 42mm 94mm, clip=true, keepaspectratio=false]{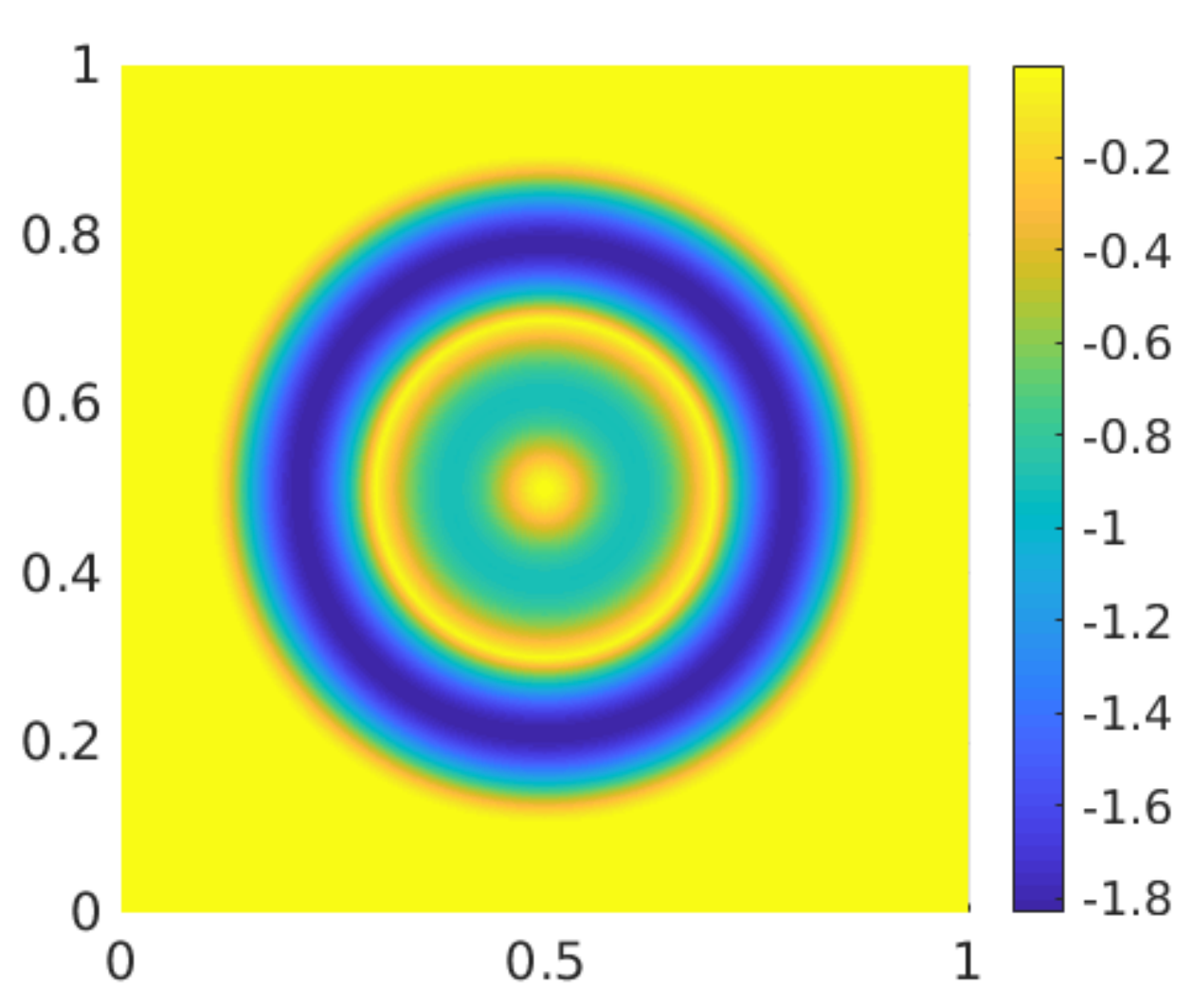}%
	\hspace{5ex}%
	\includegraphics[width=0.47\textwidth, trim=40mm 95mm 42mm 94mm, clip=true, keepaspectratio=false]{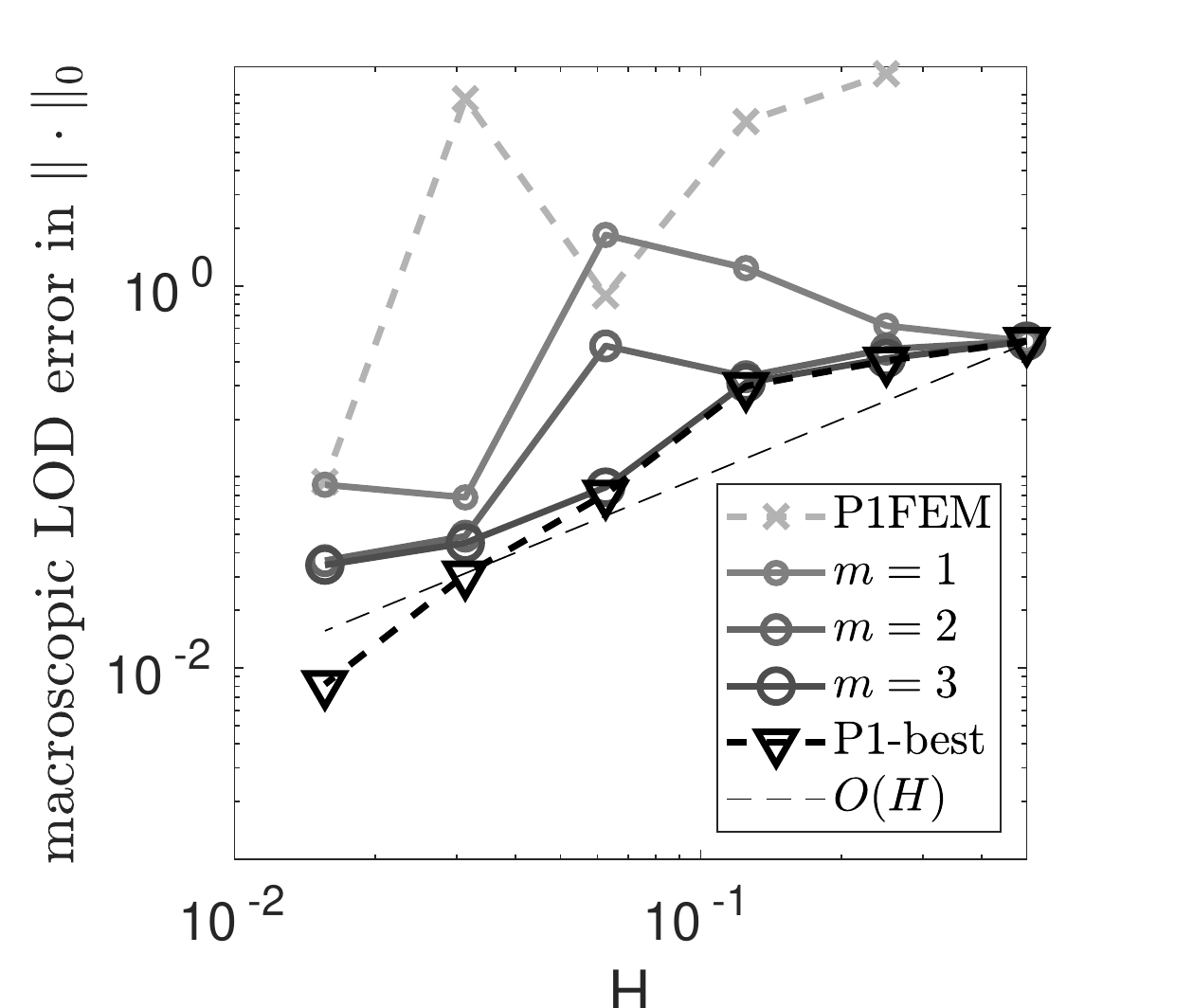}   
	\caption{Exact solution (left) and convergence history for the macroscopic LOD error in the $L^2$-norm (right) for the circular inclusion in Section \ref{subsec:circular}.}
	\label{fig:circular-conv}
\end{figure}

The curved interface is never resolved, neither by the coarse meshes $\CT_H$ nor by the fine reference mesh $\CT_h$. In particular, the standard FEM solution on $\CT_h$ may be not very reliable, which implies that the fine discretization in the LOD method might not be a faithful approximation either.
We note that in this example, a simple use of a isoparametric elements will most probably yield a good approximation with less computational effort than the LOD, but we nevertheless check the convergence rates of our method.
In the present example, the absolute $L^2(\Omega)$-error between the exact solution $u$ and the FEM solution on the fine grid $\CT_h$ is of order $10^{-2}$.
Nevertheless, the convergence plot of the macroscopic LOD solution in the $L^2(\Omega)$-norm in Figure \ref{fig:circular-conv} shows rather promising results. At least for $m=2,3$, the macroscopic LOD error still follows the best approximation error -- at least for coarse mesh sizes $H$.
We observe a deviation from this desired best-approximation error for finer meshes because the discretization error on the underlying fine mesh $\CT_h$ starts to dominate.
Given these considerations and emphasizing once more that neither $\CT_h$ nor the coarse meshes resolve the interface, the convergence results of Figure \ref{fig:circular-conv} are very satisfying.

\subsection{Multiscale sign-changing coefficient}
\label{subsec:multiscale}
We consider a multiscale, sign-changing coefficient as depicted in Figure \ref{fig:multiscale-coeffsol}, left.
It is periodic on a scale $\varepsilon=2^{-5}$ and takes the values $-4$ (blue) and $1$ (yellow).
We set $f\equiv1$ and compute a standard FE solution $u_h$ on the mesh $\CT_h$ as reference, see Figure \ref{fig:multiscale-coeffsol}, right. Note that $\CT_h$ resolves all the jumps of the coefficient so that we can hope that $u_h$ is a good approximation of the unknown exact solution $u$.
\begin{figure}
	\includegraphics[width=0.47\textwidth, trim=40mm 95mm 42mm 94mm, clip=true, keepaspectratio=false]{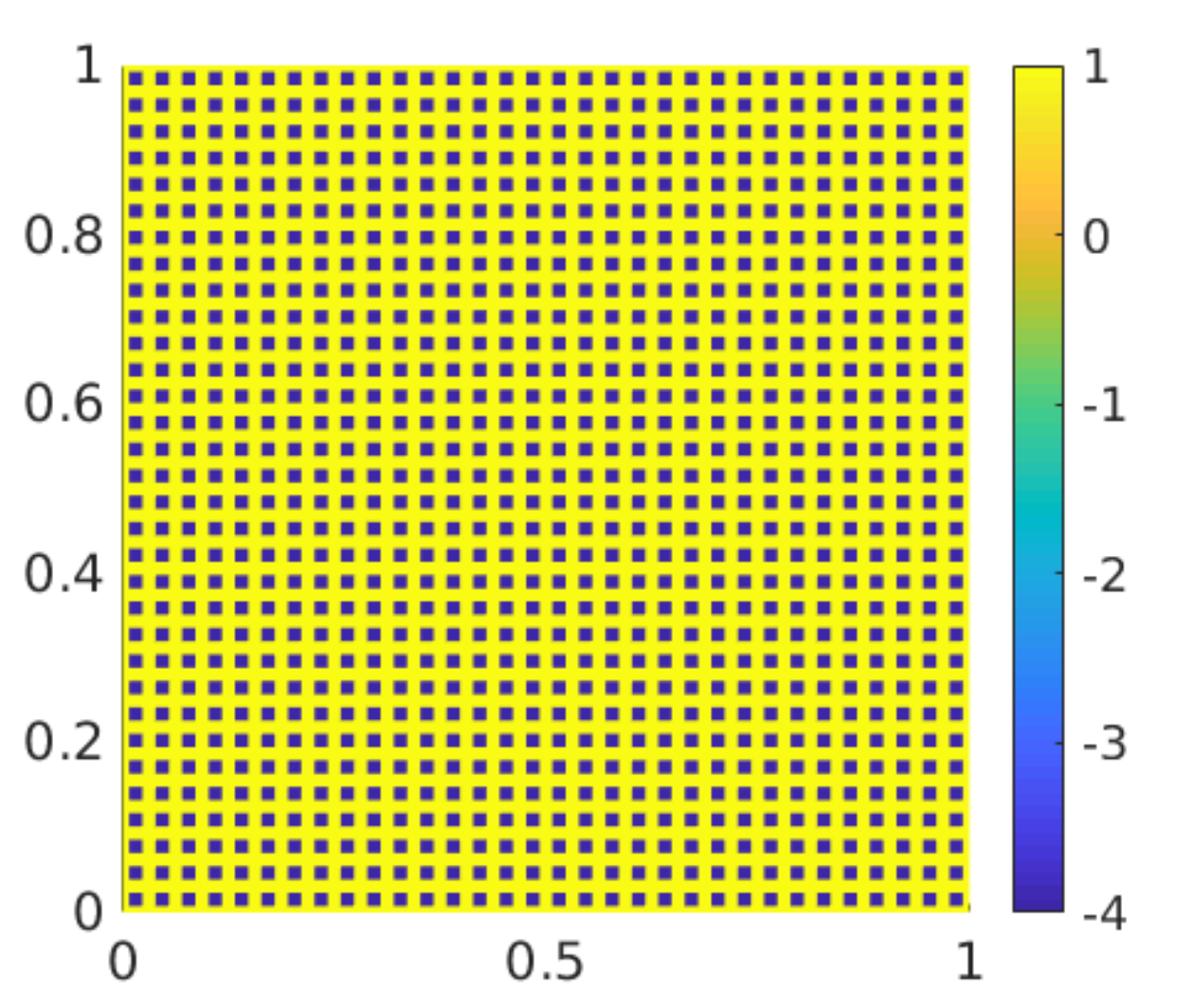}%
	\hspace{3ex}%
	\includegraphics[width=0.47\textwidth, trim=40mm 95mm 38mm 90mm, clip=true, keepaspectratio=false]{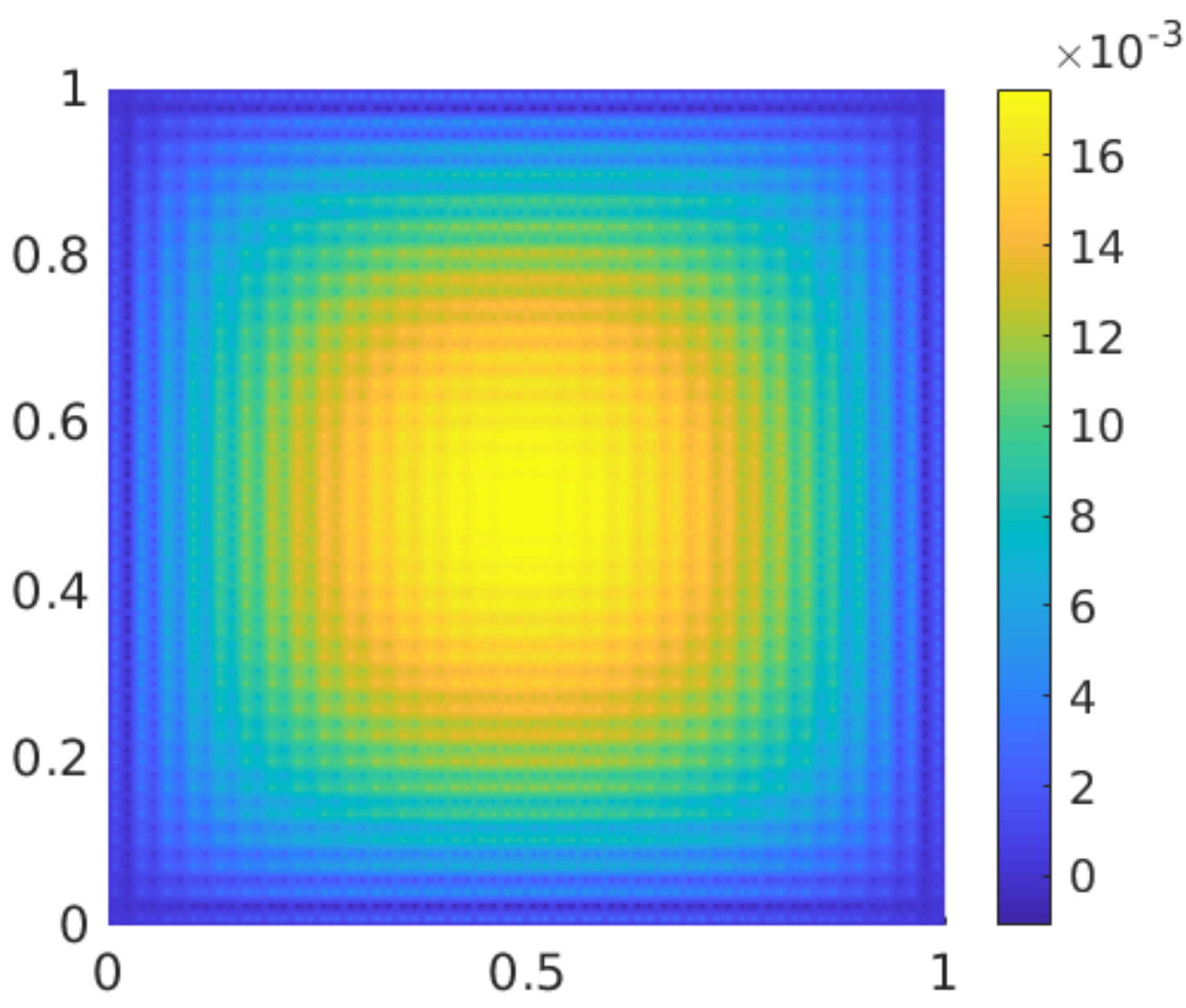}
	\caption{Coefficient (left) and fine FE solution (right) for the experiment in Section \ref{subsec:multiscale}.}
	\label{fig:multiscale-coeffsol}
\end{figure}
In this example, we illustrate the homogenization feature of the LOD and its attractive performance even in the pre-asymptotic region, i.e., for meshes that do not resolve the discontinuities of the coefficient.
For the coarse mesh $\CT_H$ with $H=2^{-4}$ and $m=3$, we depict the LOD solution, its macroscopic part, and the FE solution in Figure \ref{fig:multiscale-sols}.
First of all, we observe that the standard FEM fails on this coarse mesh because the multiscale features of the coefficient are not resolved.
To be more precise, FEM on the coarse mesh essentially calculates the solution to a diffusion problem with the coefficient $\tilde{\sigma}$ as the element-wise (arithmetic) mean of $\sigma$, i.e., $\tilde{\sigma}|_T =|T|^{-1}\int_T \sigma \, dx$ for all mesh elements $T$.
Since for all coarse mesh elements $T$, we have $|T\cap \Om|=\frac 14 |T|$ and $|T\cap \Op|=\frac 34|T|$ this average of $\sigma$ equals $-\frac 14$ in this example, which nicely explains the ``bump'' pointing in the negative direction in Figure~\ref{fig:multiscale-sols} (right).
This observation is already expected and well understood for the classical elliptic diffusion problem, see \cite{Pet15LODreview} for an excellent review.
In contrast, the LOD produces faithful approximations. Its macroscopic part can be seen as a homogenized solution and already contains the main characteristic features of the solution.
The full LOD solution also takes finescale features into account and thereby is even closer to the reference solution. This of course comes at the cost of higher computational complexity.

\begin{figure}
	\includegraphics[width=0.3\textwidth, trim=40mm 94mm 38mm 90mm, clip=true, keepaspectratio=false]{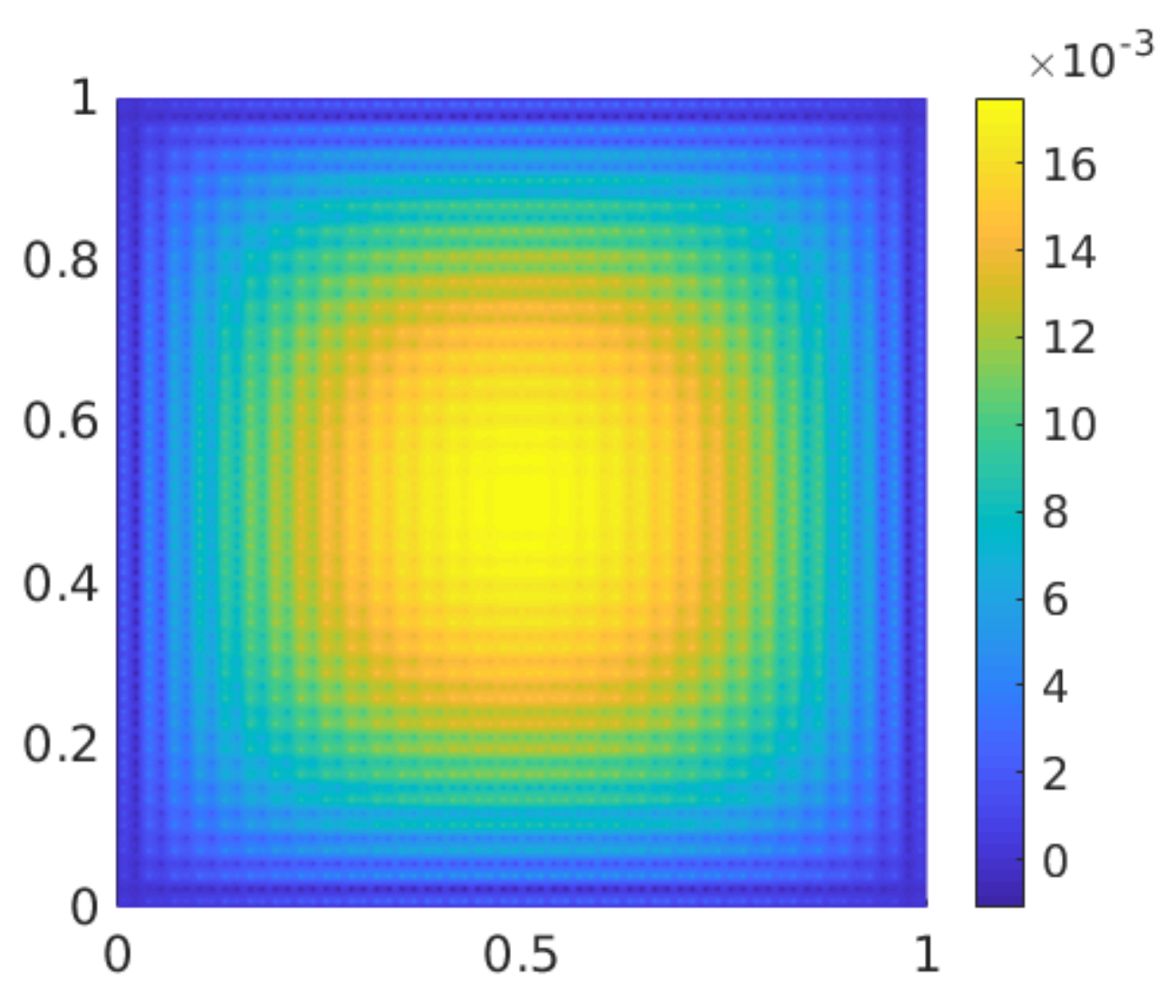}%
	\hspace{2ex}%
	\includegraphics[width=0.3\textwidth, trim=40mm 95mm 38mm 90mm, clip=true, keepaspectratio=false]{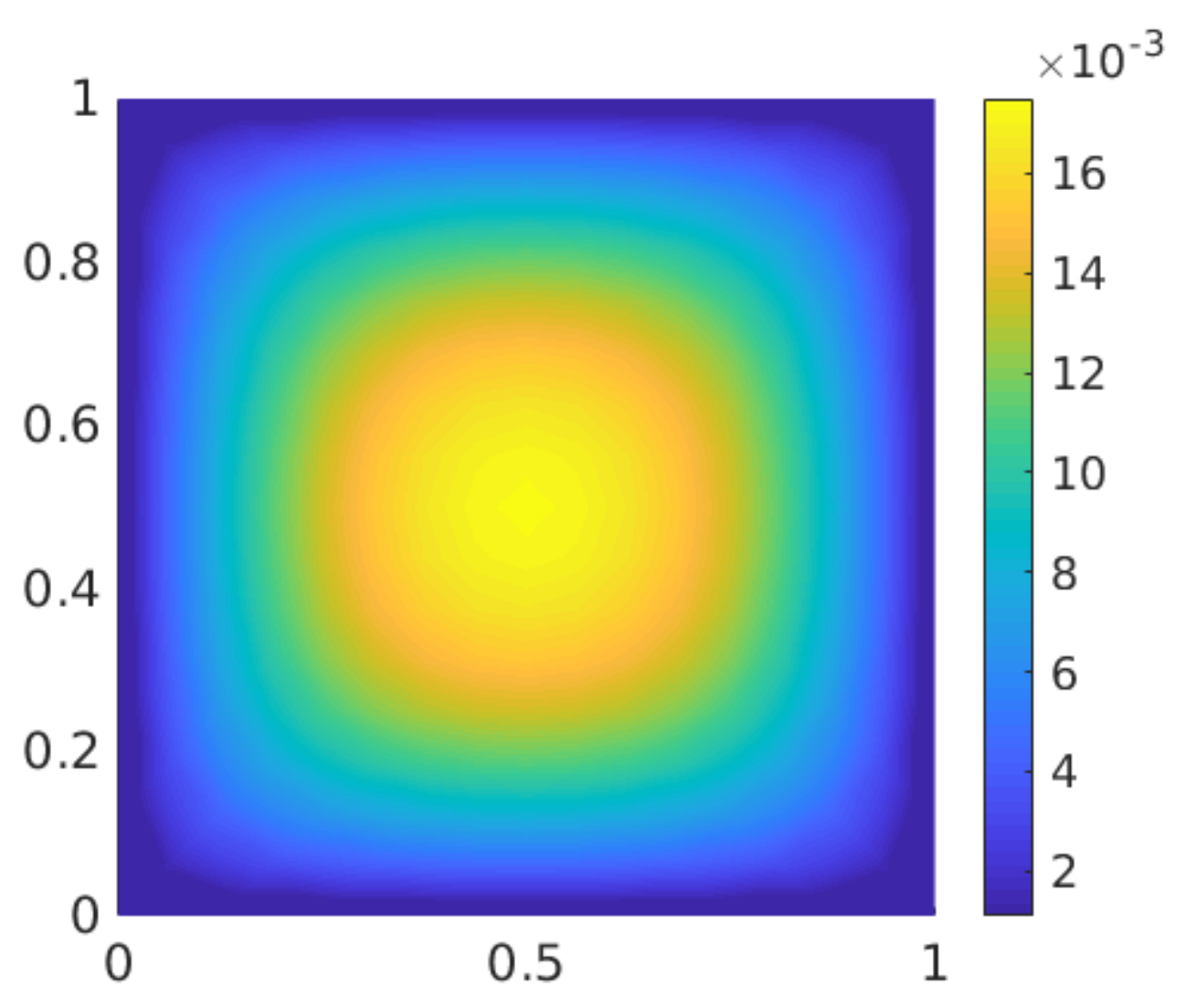}%
	\hspace{2ex}
	\includegraphics[width=0.3\textwidth, trim=40mm 95mm 38mm 90mm, clip=true, keepaspectratio=false]{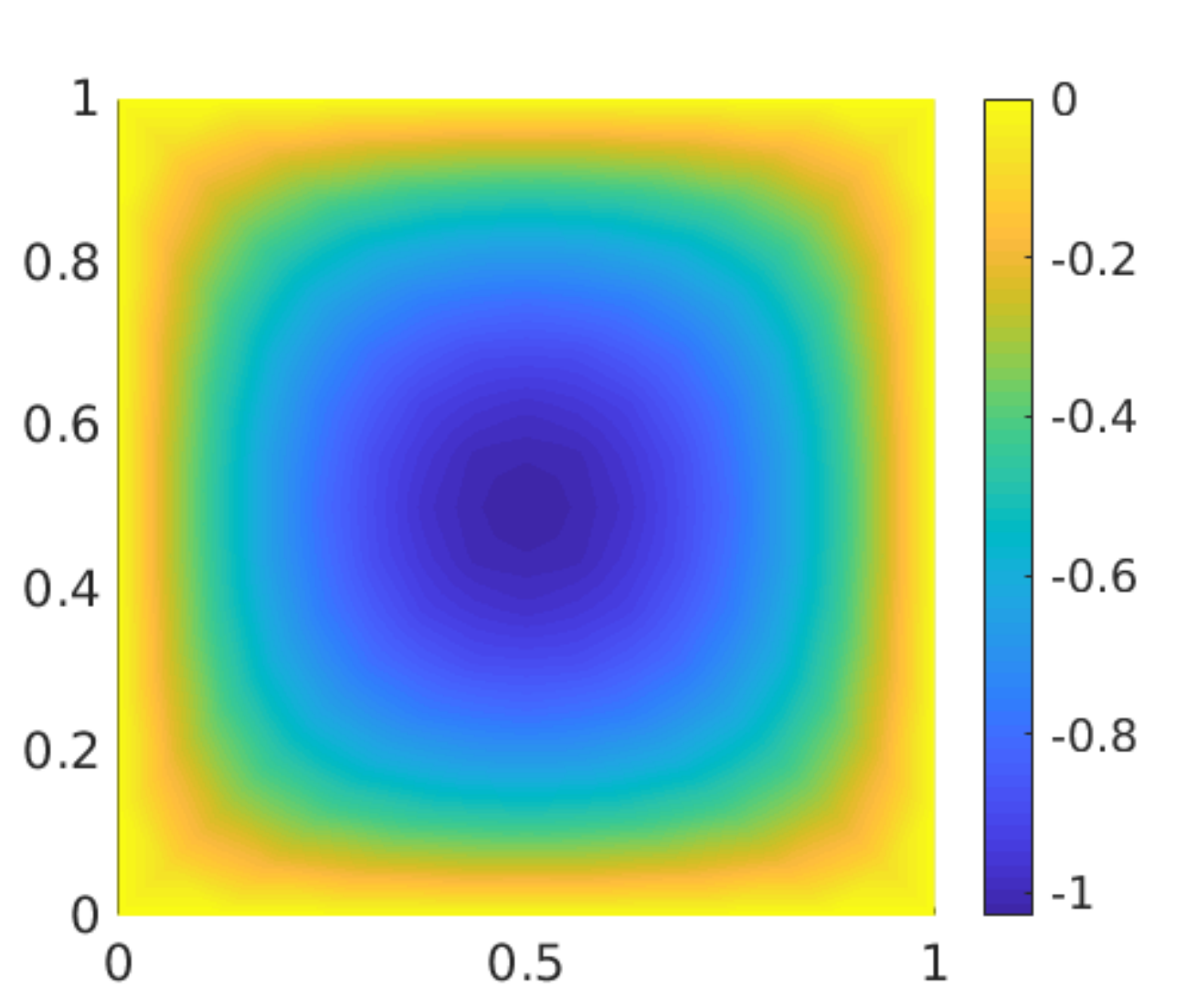}   
	\caption{LOD solution (left), macroscopic part of LOD solution (middle), and  FE solution (right) for $H=2^{-4}$ and $m=3$ in the experiment of Section \ref{subsec:multiscale}.}
	\label{fig:multiscale-sols}
\end{figure}

\section*{Conclusion}
We presented and analyzed a generalized finite element method in the spirit of the
Localized Orthogonal Decomposition for diffusion problems with sign-changing coefficients.
Standard finite element basis functions are modified by including local corrections.
The stability and the convergence of the method were analyzed under the assumption that
the contrast is ``sufficiently large''.
Our analysis involves a discrete $\UT$-coercivity argument, as well as
``symmetrized'' patches to compute the correctors associated with the elements
close to the sign-changing interface.
Numerical experiments illustrated the theoretically predicted optimal convergence rates.
Furthermore, they showed the applicability of the method for general coarse meshes, which do not
resolve the interface, and highly heterogeneous coefficients.

The numerical experiments also outlined some possible future research questions.
If the contrast is close to the critical interval, the patches for the corrector
computations need to be rather large. This contrast-dependency might be reduced with
the norm considered in  \cite{CV18signchangingapost}, where we mention the connection
with the LOD approach in weighted norms \cite{HM17lodcontrast,PS16lodcontrast}.

\section*{Acknowledgments}
Major parts of this work were carried out while BV was affiliated with University of Augsburg. BV's work at KIT is funded by the Deutsche Forschungsgemeinschaft (DFG, German Research Foundation) -- Project-ID 258734477 -- SFB 1173.
We are very grateful to the referees for their valuable comments that helped to improve the paper.

\bibliographystyle{abbrv}
\bibliography{references}

\appendix
\section{Technical results used in Section \ref{sec:intpol}}\label{sec:appendix}

In this section, we prove a few technical results used in Section \ref{sec:intpol} combining standard scaling arguments for classical FE functions.
Throughout the appendix, we use the notation introduced in Sections \ref{subsec:meshes} and \ref{sec:intpol}.
Classical finite element scaling arguments use the mapping of elements in the mesh $\CT_H$ onto the reference element. We use the standard notation of $\widehat \cdot$ for quantities (functions, constants, etc.) on the reference element. In particular, functions $\widehat v$ and $v$ are connected with each other via the standard reference element mapping.

\subsection{Key properties of the Oswald operator $I_H$}
\label{appendix_IH}

\begin{proof}[Proof of Lemma \ref{lem:estimatema}]
	Fix $\ba \in \CV_H$, $v \in H^1_0(\Omega)$, and recall the definition
	\begin{equation*}
	m^\ba(v) \eq \frac{1}{\sharp \ba} \sum_{K \in \CT_H^\ba} (P_K v)(\ba).
	\end{equation*}
	It is clear that
	\begin{equation*}
	|m^\ba(v)| \leq \frac{1}{\sharp \ba} \sum_{K \in \CT_H^\ba} \|P_K v\|_{0,\infty,K}
	\leq \max_{K \in \CT_H^\ba} \|P_K v\|_{0,\infty,K}
	=    \|P_{K_\star} v\|_{0,\infty,K_\star}
	\end{equation*}
	for some $K_\star \in \CT_H^\ba$. Then, since $w \eq P_{K_\star} v \in \CP_1(K_\star)$,
	the first estimate of \eqref{eq_inf_inv} shows that
	\begin{equation*}
	\|w\|_{0,\infty,K_\star}^2
	=
	\|\widehat w\|_{0,\infty,\widehat K}^2
	\leq
	\frac{\hC_{\mathrm inf}^2}{|\widehat K|^{2}} \|\widehat w\|_{0,\widehat K}^2
	=
	\frac{\hC_{\mathrm inf}^2}{|K_\star|^2}\|w\|_{0,K_\star}^2
	\leq
		\frac{\hC_{\mathrm inf}^2}{|K_\star|^2}\|v\|_{0,K_\star}^2,
	\end{equation*}
	from which \eqref{eq_estimate_ma} follows.
\end{proof}

\begin{proof}[Proof of Lemma \ref{lem:poincare}]
	
		Recall the notation for the reference patches introduced
		at Section \ref{section_reference_patch}.
		If $v \in H^1(\oa)$, then $\widehat v \eq v \circ \CF$ belongs to $H^1(\ho)$,
		and we have $m(v) = \widehat m(\widehat v)$, where
		\begin{equation*}
		\widehat m(\widehat v)
		\eq
		\frac{1}{\sharp \ba} \sum_{\widehat K \in \CF^{-1}(\CT_H^\ba)}
		(\CP_{\widehat K} \widehat v)(\boldsymbol 0).
		\end{equation*}
		
		Now, we observe that for $\widehat q \in \CP_0(\ho)$, $\widehat m(\widehat q) = 0$
		implies that $\widehat q = 0$. Then, a standard contradiction argument
		(see for instance \cite[proof of Theorem 3.1.1]{Ciarlet}) shows that there
		exists a constant $\hC$ such that
		\begin{equation*}
		\|\hw\|_{0,\ho} + |\hw|_{1,\ho}
		\leq
		\hC \left (|\widehat m(\hw)| + |\hw|_{1,\ho}\right ),
		\quad \forall \hw \in H^1(\ho)
		\end{equation*}
		justifying estimate \eqref{eq_poincare_tmp}.
		
		Hence, applying \eqref{eq_poincare_tmp}, we have
		\begin{equation*}
		\|\hv\|_{0,\ho} \leq \hC_{\mathrm P} |\hv|_{1,\ho}.
		\end{equation*}
		At this point, employing (element-wise) usual scaling arguments, we easily see that
		\begin{equation*}
		\|v\|_{0,\oa} \leq \max_{K \in \CT_H^\ba}
		\max_{\hK \in \widehat \CT} \sqrt{\frac{|K|}{|\hK|}} \|\hv\|_{0,\ho}
		\end{equation*}
		and
		\begin{equation*}
		|v|_{1,\oa} \leq \max_{K \in \CT_H^{\ba}} \max_{\hK \in \widehat \CT}
		\sqrt{\frac{|\hK|}{|K|}}\frac{h_{\hK}}{\rho_K} \|\hv\|_{0,\ho}.
		\end{equation*}
		All in all, recalling that $h_{\hK} \leq 1$, we obtain that
		\begin{equation*}
		\|v\|_{0,\oa}
		\leq
		\hC_{\mathrm P}
		\max_{\hK,\hK' \in \widehat \CT} \sqrt\frac{|\hK|}{|\hK'|}
		\max_{K,K' \in \CT_H^\ba} \sqrt\frac{|K|}{|K'|}
		\frac{1}{\rho_K} |v|_{1,\oa},
		\end{equation*}
		from which the result follows.
	
\end{proof}

\subsection{Construction of dual functions}
\label{appendix_eta}

The main aim of this appendix is to construct the function $\eta^\ba$ used for the definition of
$\UT_H$ and study its scaling. For the ensuing construction to hold, we need to assume that $\CT_H$
resolves the interface $\Gamma$. We emphasize, however, that no symmetry of the mesh is required.
Moreover, we believe that a similar result holds if the interface does not cut the elements
``too badly''. We refer to \cite{HMW19lodfracture} for a similar discussion in a different context.

\begin{lemma}\label{lem:refeta}
	For all $\widehat \lambda \in L^2(\widehat K)$, there exists a unique
	$\widehat \eta \in H^1_0(\widehat K) \cap \CP_{d+2}(\widehat K)$ such that
	\begin{equation}
	\label{eq_def_eta_ref}
	(\widehat \eta,\widehat v)_{\widehat K} = (\widehat \lambda,\widehat v)_{\widehat K}
	\quad
	\forall \widehat v \in \CP_1(\widehat K),
	\end{equation}
	and we have
	\begin{equation}
	\label{eq_estimate_eta_ref}
	|\widehat \eta|_{1,\widehat K} \leq \hC_{\mathrm norm}\hC_{\mathrm inv}\|\widehat \lambda\|_{0,\widehat K}
	\end{equation}
	In addition, the equality
	\begin{equation}
	\label{eq_def_eta}
	(\eta,v)_{K} = (\lambda,v)_{K}
	\quad
	\forall v \in \CP_1(K)
	\end{equation}
	and the estimate
	\begin{equation}
	\label{eq_estimate_eta}
	|\eta|_{1,K} \leq \frac{\hC_{\mathrm norm}\hC_{\mathrm inv}}{\rho_K} \|\lambda\|_{0,K}
	\end{equation}
	hold true. Moreover, whenever $\lambda \in \CP_1(K)$, we have
		\begin{equation}
		\label{eq_eta_proj}
		\lambda = P_K \eta.
		\end{equation}
\end{lemma}

\begin{proof}
		Our proof rely on the bubble function $\widehat b$ defined in Section
		\ref{section_reference_element}.
	Let $\widehat \lambda \in L^2(\widehat K)$.
	There exists a unique $\widehat w \in \CP_1(\widehat K)$ such that
	\begin{equation*}
	(\widehat b \widehat w,\widehat v)_{\widehat K} = (\widehat \lambda,\widehat v)_{\widehat K}
	\quad
	\forall \widehat v \in \CP_1(\widehat K).
	\end{equation*}
	Then, one easily observes that
	$\widehat \eta \eq \widehat b \widehat w \in H^1_0(\widehat K) \cap \CP_{d+2}(\widehat K)$
	satisfies \eqref{eq_def_eta_ref}. Furthermore, picking the test function
	$\widehat v = \widehat w$ in the definition of $\widehat w$ and employing
	\eqref{eq_bubble_norm}, we have
	\begin{equation*}
	\|\widehat b^{1/2} \widehat w\|_{0,\widehat K}^2
	=
	(\widehat \lambda,\widehat w)_{\widehat K}
	\leq
	\|\widehat \lambda\|_{0,\widehat K}
	\|\widehat w\|_{0,\widehat K}
	\leq
	\hC_{\mathrm norm}
	\|\widehat \lambda\|_{0,\widehat K}
	\|\widehat b^{1/2} \widehat w\|_{0,\widehat K}
	\end{equation*}
	and \eqref{eq_estimate_eta_ref} follows recalling \eqref{eq_inf_inv} since
	\begin{equation*}
	\|\widehat \eta\|_{0,\widehat K}
	=
	\|\widehat b \widehat w\|_{0,\widehat K}
	\leq
	\|\widehat b^{1/2} \widehat w\|_{0,\widehat K}
	\leq
	\widehat C_{\mathrm{norm}}
	\|\widehat \lambda\|_{0,\widehat K},
	\end{equation*}
	and
	\begin{equation*}
	|\widehat \eta|_{1,\widehat K} \leq \hC_{\mathrm inv} \|\widehat \eta\|_{0,\widehat K},
	\end{equation*}
	as $\widehat \eta \in \CP_1(\widehat K)$.
	
	At this point \eqref{eq_def_eta} and \eqref{eq_estimate_eta} follows from usual
	scaling arguments, since $\widehat H = 1$, and
		\eqref{eq_eta_proj} is a direct consequence of \eqref{eq_def_eta}.
\end{proof}

\begin{proof}[Proof of Lemma \ref{lem:defetaa}]
	Let $\ba\in \CV_H^+\cup\CV_H^0$ be arbitrary but fixed. There exists an element
	$K_\star\in\CT_H$ such that $K_\star \subset \omega^\ba\cap\Op$. Following
	Lemma \ref{lem:refeta} we consider a function $\eta^\ba \in H^1_0(K_\star)$
	such that $P_{K_\star} \eta^\ba = \psi^\ba|_{K_\star}$. Then, we obtain for
	any $\ba^\prime\in \CV_H$ that
	\begin{equation*}
	m^{\ba'}(\eta^\ba)
	=
	\frac{1}{\sharp \ba'} \sum_{K \in \CT_H^{\ba'}} (P_K \eta^\ba)(\ba')
	=
	\frac{1}{\sharp \ba'} \psi^\ba(\ba')
	=
	\delta_{\ba,\ba'}.
	\end{equation*}
	
	On the other hand, using \eqref{eq_estimate_eta}, we have
	\begin{equation*}
	|\eta^\ba|_{1,\Omega_+}
	=
	|\eta^\ba|_{1,K_\star}
	\leq
	\frac{\hC_{\mathrm norm}\hC_{\mathrm inv}}{\rho_{K_\star}} \|\psi^\ba\|_{0,K_\star}
	\leq
	\hC_{\mathrm norm}\hC_{\mathrm inv}\frac{|K_\star|^{1/2}}{\rho_{K_\star}}.
	\end{equation*}
\end{proof}

\end{document}